\documentclass[10pt, reqno]{amsart}

\usepackage{latexsym, amsmath,amssymb}
\usepackage{amsthm}
\usepackage{amsfonts}
\usepackage{bbm}
\usepackage{geometry}

\usepackage{chngcntr}

\usepackage{color}
\usepackage{hyperref}

\usepackage[style=numeric, maxbibnames=10]{biblatex}
\newcommand{\N}{{\mathbb N}}
\newcommand{\R}{{\mathbb R}}

\newcommand{\g}{\mathfrak{g}}
\newcommand{\m}{\mathfrak{m}}
\newcommand{\f}{\mathfrak{f}}
\newcommand{\h}{\mathfrak{h}}
\newcommand{\rr}{\la r \ra}
\newcommand{\ttt}[1]{\tilde{#1}}
\newcommand{\bam}{\beta_{\approx m}}

\newcommand{\ang}{{\not\negmedspace\nabla}}

\newcommand{\la}{\langle}
\newcommand{\ra}{\rangle}

\newcommand\reallywidehat[1]{%
	\savestack{\tmpbox}{\stretchto{%
			\scaleto{%
				\scalerel*[\widthof{\ensuremath{#1}}]{\kern-.6pt\bigwedge\kern-.6pt}%
				{\rule[-\textheight/2]{1ex}{\textheight}}
			}{\textheight}%
		}{0.5ex}}%
	\stackon[1pt]{#1}{\tmpbox}%
}

\newcommand{\boxg}{\Box_\mathfrak{g}}
\newcommand{\parder}[2]{\frac{\partial #1}{\partial #2}}
\newcommand{\dd}{\hbox{ }d}

\newcommand{\LE}{{\mathcal {LE}}}
\newcommand{\LEs}{{\mathcal {LE}^*}}

\newcommand{\ls}[1]{\ell^1S(r^{#1})}

\newcommand{\sr}[1]{S_{rad}(r^{#1})}
\newcommand{\albe}{{\alpha\beta}}
\newcommand{\gade}{{\gamma\delta}}

\newcommand{\LEnorm}[1]{\sup_{m}\|\la r \ra^{-\frac{1}{2}}#1\|_{L^2(\R_+ \times A_m)}}
\newcommand{\LEonorm}[1]{ \|\partial #1\|_{LE} + \|\la r \ra^{-1}#1\|_{LE}}
\newcommand{\LEsnorm}[1]{ \sum_m \|\la r \ra^{\frac{1}{2}}#1\|_{L^2(\R_+ \times A_m)}}
\newcommand{\lesn}[1]{_{\mathcal{LE}^{*,#1}}}
\newcommand{\letn}[1]{_{\mathcal{LE}_\tau^{#1}}}

\newcommand{\len}[1]{_{\mathcal{LE}^{#1}}}
\newcommand{\lrsupo}{_{L_r^2L_\omega^\infty(A_m)}}
\newcommand{\lrlo}{_{L_r^2L_\omega^2(A_m)}}
\newcommand{\ltwoam}{_{L^2(A_m)}}
\newcommand{\ltwor}{_{L^2(\R^3)}}
\newcommand{\les}{_{\mathcal{LE}^*}}

\newcommand{\vijk}{v_{ijk}}
\newcommand{\vleijk}{v_{\le i \le j \le k}}
\newcommand{\vmlow}{v_m^{low}}
\newcommand{\vmhigh}{v_m^{high}}

\newcommand{\vf}{T^i\Omega^jS_r^k}
\newcommand{\supn}[1]{\sup_{i+j+k \le #1}}

\newcommand{\enitr}{e^{-i\tau\rr}}
\newcommand{\drfracr}{(\partial_r +\frac{1}{r})}

\newcommand{\bm}{\beta_{\approx m}}
\newcommand{\rx}{|x|}
\newcommand{\ry}{|y|}

\newcommand{\znl}[2]{_{Z^{#1, #2}}}
\renewcommand{\l}{\lambda}
\newcommand{\el}{E_\l^{\nu}}
\newcommand{\gl}{g_\l^{\nu}}
\newcommand{\eln}[2]{E_{#1}^{#2}}
\newcommand{\gln}[2]{g_{#1}^{#2}}
\newcommand{\zln}[2]{\zeta_{#1}^{#2}}
\newcommand{\tvk}{\tilde{\varphi}_{\kappa+1}}

\renewcommand{\epsilon}{\varepsilon}

\counterwithin{equation}{section}
\newtheorem{theorem}{Theorem}[section]

\newtheorem{lemma}[theorem]{Lemma}

\newtheorem{proposition}[theorem]{Proposition}
\newtheorem{definition}[theorem]{Definition}

\begin{document}

\title[Local Decay in the Asymptotically Flat Stationary Setting]{The effect of metric behavior at spatial infinity on pointwise wave decay in the asymptotically flat stationary setting.}
\author{Katrina Morgan}
\address{Northwestern University, Evanston, IL}
\email{katrina.morgan@northwestern.edu}
\thanks{The author was supported in part by the NSF under Grant No. 1440140 while the author was in residence at the Mathematical Sciences Research Institute in Berkeley, California}

\begin{abstract}
The current work considers solutions to the wave equation on asymptotically flat, stationary, Lorentzian spacetimes in (1+3) dimensions. We investigate the relationship between the rate at which the geometry tends to flat and the pointwise decay rate of solutions. The case where the spacetime tends toward flat at a rate of $|x|^{-1}$ was studied in \cite{tat2013}, where a $t^{-3}$ pointwise decay rate was established. Here we extend the result to geometries tending toward flat at a rate of $|x|^{-\kappa}$ and establish a pointwise decay rate of $t^{-\kappa-2}$ for $\kappa \in \N$ with $\kappa \ge 2$. We assume a weak local energy decay estimate holds, which restricts the geodesic trapping allowed on the underlying geometry. We use the resolvent to connect the time Fourier Transform of a solution to the Cauchy data. Ultimately the rate of pointwise wave decay depends on the low frequency behavior of the resolvent, which is sensitive to the rate at which the background geometry tends to flat.
\end{abstract}

\maketitle


\begin{section}{Introduction}
This work examines the effect of the  far away metric behavior on pointwise wave decay on asymptotically flat, stationary backgrounds in $(1+3)$ dimensions. A spacetime geometry is asymptotically flat if the metric coefficients tend toward the flat Minkowski metric ($\mathfrak{m} = \hbox{diag}(-1,1,1,1)$ in $(t,x)$ coordinates) as $r \to \infty$. Here and throughout the paper we take $r := |x|$. A geometry is said to be stationary if the metric coefficients are time independent. The main result of this paper quantifies the relationship between the rate at which the background geometry tends to flat and the rate of pointwise wave decay.

The pointwise decay rates established in this work interpolate between two known cases: the flat Minkowski spacetime and asymptotically flat spacetimes tending toward flat at a rate of $r^{-1}$. Table 1 summarizes these results, which hold for compactly supported initial data. In the case of \cite{tat2013} and the current work, the assumptions on the initial data can be weakened.
\begin{table}
	 		\begin{center}
		\begin{tabular}{|c|c|c|}
			\hline
			 & \textsc{Metric Behavior} & \textsc{Pointwise Wave Decay} \\
			\hline
			\cite{tat2013}: & $\mathfrak{g} = flat + \mathcal{O}(r^{-1})$ & $ |u(t,x)| \lesssim_x t^{-3}$ \\ 
			\hline
			  \textbf{Current Work}: & $\mathfrak{g} = flat + \mathcal{O}(r^{-\kappa})$ & $|u(t,x)| \lesssim_x t^{-\kappa-2}$  \\
			 \hline
			  Sharp Huygens': & $\mathfrak{g} = flat$ & $|u(t,x)| \lesssim_x t^{-\infty}$ \\
			 \hline
		\end{tabular}
		\end{center}
	\caption{Summary of Decay Rates}
	\label{table:1}
	\end{table}

Since we are working in three spatial dimensions, sharp Huygens' principle says that solutions to the wave equation on the flat Minkowski spacetime decay all the way to 0 in finite time at each point in space. In Table \ref{table:1} we make a weaker statement that solutions to the wave equation have arbitrarily fast polynomial time decay. 

The results in \cite{tat2013} and the current work rely on an assumption that a dispersive estimate called \textit{weak local energy decay} holds. Heuristically, this estimate imposes sufficient restrictions on the behavior of the geometry within compact regions so only the long range metric behavior is left to be studied. A more detailed summary of local energy decay is provided later in the introduction.

Asymptotically flat spacetimes arise in general relativity, which has motivated a variety of mathematical questions about wave behavior in this setting. For example, the Schwarzschild metric describing the geometry of space in the presence of a single, non-rotating black hole and the Kerr metric describing spacetime in the presence of a single, axially symmetric, rotating black hole both tend toward flat at a rate of $r^{-1}$. A conjecture posited by physicist Richard Price in the 1970's in \cite{price}, known as Price's Law, predicted a $t^{-3}$ decay rate for waves  on the Schwarzschild metric. The question of proving Price's law was explored in \cite{tat2013} and also in \cite{donn}, where they analyze the wave behavior via spherical modes using the spherical symmetry of the Schwarzschild metric. Pointwise decay rates for the Kerr spacetime were studied in \cite{dafrod} and \cite{finster}. In \cite{metato} the authors proved Price's Law for non-stationary asymptotically flat spacetimes and established the $t^{-3}$ decay rate for a class of perturbations of the Kerr spacetime. The techniques in \cite{donn}, \cite{tat2013}, and the current work involve taking the Fourier transform in time and therefore do not readily extend to non-stationary geometries. 
	
Weak local energy decay on the Schwarzschild geometry was established in \cite{blue2003}, \cite{dafermos2009red}, and \cite{mametato}. For the Kerr spacetime with low angular momentum, weak local energy decay estimates were proved in \cite{andblue15}, \cite{dafermos2009red}, and \cite{dafrodshlap16}. The assumptions in \cite{tat2013} therefore hold for Schwarzschild and Kerr with low angular momentum. A major challenge in obtaining local energy estimates for the Kerr and Schwarzschild geometries is the presence of trapping, in which a portion of the wave flow remains within a fixed set. 
	
A natural question arising from Tataru's result in \cite{tat2013} is: What aspects of the Schwarzschild geometry dictate Price's Law? There are three locations that are a priori suspected to affect this decay rate: the event horizon, the photon sphere, and the behavior of the perturbation at spatial infinity.  The current work shows that the metric behavior at spatial infinity dictates the pointwise decay rate of waves when the weak local energy estimate holds. On the Schwarzschild background, trapping occurs in two areas called the event horizon and the photon sphere. The trapping at the event horizon has been shown to be trivial due to what is known as the red-shift effect, which guarantees energy decay along the trapped rays (\cite{dafrod}, \cite{dafermos2009red}). The photon sphere corresponds to a fixed radius, and rays initially tangential to this surface remain there for all time. The behavior on Kerr backgrounds is more complicated. The trapping at the event horizon is similarly known to be trivial, but the other trapped set does not occur on a fixed radius and can only be described in phase space. In order to deal with trapping, a weak local energy estimate with zero coefficients on the trapped set is often introduced. If this holds, then one obtains local energy estimates on the trapped set with a derivative loss. Our definition of the weak local energy decay estimate includes this derivative loss.
	
Questions similar to the aim of this paper were studied in \cite{bonyhaf2} and \cite{bonyhaf1} where the authors established local decay rates for waves on asymptotically flat, stationary spacetimes which tend toward flat at different rates. There are several key differences compared with the current work. First, we handle full Lorentzian perturbations of flat Minkowski space rather than restricting to perturbations of the Laplacian. This leads to the metrics considered in this paper containing $dtdx_i$ terms, which results in mixed space-time differential operators in our wave operator. Second, we allow for the possibility of unstable trapping on our background. In \cite{bonyhaf2} and \cite{bonyhaf1}, a nontrapping assumption is used in order to obtain decay for the high frequency part of a solution to the wave equation (it is not needed for the low frequency part). Third, our result improves upon the established decay rates. Finally we note that \cite{bonyhaf1} considers $(1+n)$ dimensional geometries for $n \ge 2$ and \cite{bonyhaf2} considers $n$ odd with $n \ge 3$. The current work only studies $(1+3)$ dimensional spacetimes.
	
\subsection{The Wave Equation}
			The flat wave operator is given by
	\[ \Box = -\partial_t^2 + \sum_{i=1}^3 \partial_{x_i}^2 = - \partial_t^2 + \Delta_x \]
	where $\Delta_x$ is the spatial Laplacian. Throughout the paper we write $\Delta = \Delta_x$. Similarly we write $\nabla = \nabla_x$ for the spatial gradient. When both time and spatial derivatives are considered we use $\partial$.
	 
The wave operator associated to a Lorentzian metric $\mathfrak{g} = \mathfrak{g}_{\alpha \beta}d\alpha d\beta$ with signature (3,1) is given by
	\begin{equation} \label{boxg}
		\Box_\mathfrak{g} = \frac{1}{\sqrt{|\mathfrak{g}|}}\partial_{\alpha} \sqrt{|\mathfrak{g}|}\mathfrak{g}^{\alpha \beta}\partial_{\beta}
	\end{equation}	  
where $|\mathfrak{g}|$ is the determinant of the matrix associated to the metric, and $\alpha$ and $\beta$ are summed over both time and space dimensions. We use Latin indices $i,j$ to indicate only spatial dimensions are being considered and Greek indices to indicate both space and time dimensions are being considered. When we wish to specify Cartesian vs. spherical coordinates, we use $\alpha, \beta$ for Cartesian and $\gamma, \delta$ for spherical coordinates.

The flat metric (i.e. the Minkowski metric) is given in rectangular coordinates by
	\begin{equation} 
		\mathfrak{m} = - dt^2 + \sum_{i=1}^3 dx_i^2. 
	\end{equation}
Taking $\mathfrak{g}=\mathfrak{m}$ in \eqref{boxg} thus yields $\Box_\mathfrak{m} = \Box$, as one would expect.

\subsection{Energy Estimates}
	We are interested in the Cauchy problem
	\begin{equation} 
		(\boxg + V)u = f, \qquad u(0,x) = u_0, \qquad \partial_t u(0,x) = u_1 		
		\label{inhcauchy}
	\end{equation}
	where $V$ is a scalar potential. The assumptions placed on $\mathfrak{g}$ and $V$ are given in Section \ref{mainthm}. The Cauchy data at time $t$ is denoted $u[t] = \Big(u(t,\cdot), \partial_tu(t,\cdot)\Big)$.

		\begin{definition}	
			We say the evolution \eqref{inhcauchy} satisfies the uniform energy bounds if:
				\begin{equation} \label{unifenbd}
					\| u[t]\|_{\dot H^{k,1} \times H^k} \leq c_k (\|u[0]\|_{\dot H^{k,1} \times H^k} +\|f\|_{H^k L^1}),
	\qquad t \geq 0, \quad k \geq 0.
				\end{equation}
		\end{definition}	
		
	Here $H^k$ denotes the usual Sobolev space, and we say $\phi \in \dot{H}^{k,1}$ if $\nabla \phi \in H^k$.
	
	\subsubsection{Local Energy Decay}
		Local energy decay estimates originated in the work of Morawetz (\cite{mora}) where the author established a dispersive estimate for solutions to the flat wave equation. In \cite{kss} the authors presented a new approach for proving existence of solutions for nonlinear waves which relied on obtaining a Morawetz-type estimate. The use of local energy estimates has since become a standard tool for studying nonlinear wave equations (e.g. \cite{bonyhaf3}, \cite{hmssz}, \cite{metnaksog}, \cite{sogwan}, \cite{metsog06}, \cite{lintoh18}, \cite{luk2010null}, \cite{yang2015global}, among many others). 
		
		The original Morawetz estimate considered a solution $u$ to the homogeneous flat wave equation with initial data $u_0, u_1$ and states
	\[ \int_0^t \int_{\R^3} \frac{1}{|x|} |\ang u|^2 (t,x) \dd t dx \lesssim \| \nabla u_0\|_{L^2}^2 + \|u_1\|_{L^2}^2. \]
    Restricting to compact regions in space, one is able to obtain similar bounds on $u$ and its derivatives (see e.g. \cite{kss}, \cite{smithsogge}, and \cite{sterb}). Our definitions for the local energy norms will restrict to dyadic spatial regions. We use $\rr$ to indicate a smooth function of $r$ such that $\rr \ge 1$ and $\rr = r$ for $r >2$, and we define $A_m := \{ x : 2^{m} \le \la r \ra \le 2^{m+1} \}$. One benefit of using these dyadic regions is that $r \approx 2^m$ on the region of integration, so the weights in the local energy norm can roughly be treated as constant within the region of integration.

		   The local energy norm we use is defined by
			\[
				\|u\|_{LE} = \LEnorm{u}.
			\]
		Its $H^1$ analogue is given by
			\[
				\|u\|_{LE^1} = \LEonorm{u},
			\]
		and the dual norm is given by
			\[
				\|f\|_{LE^*} = \LEsnorm{f}.
			\]
		For functions with higher regularity we define the following norms
			\[
				\|u\|_{LE^{1,N}} = \sum_{j \le N} \|\partial^j u\|_{LE^1}, \qquad \|f\|_{LE^{*,N}} = \sum_{j \le N} \|\partial^j f\|_{LE^*}.
			\]
		The spatial counterparts of the $LE$ and $LE^*$ space-time norms are
			\[
				\|v\|_\LE = \sup_m \| \la r \ra^{-\frac{1}{2}} v\|_{L^2(A_m)}; \qquad \|g\|_\LEs = \sum_m \| \la r \ra^{\frac{1}{2}} g\|_{L^2(A_m)}
			\]
		with the higher regularity norms defined by
			\[
				\|v\|_{\mathcal{LE}^{N}} = \sum_{j \le N} \|\nabla^jv\|_{\mathcal{LE}}, \qquad \|g\|_{\mathcal{LE}^{*,N}} = \sum_{j \le N} \|\nabla^jg\|_\LEs.
			\]			

		\begin{definition}
			We say the evolution \eqref{inhcauchy} satisfies the local energy decay estimate if:
				\begin{equation} \label{led}
	 				\| u\|_{LE^{1,N}} \leq c_N (\|u[0]\|_{H^{N,1} \times H^N} + \|f\|_{LE^{*,N}}), \qquad N \geq 0.
				\end{equation}
		\end{definition}
	
	Heuristically, the local energy decay estimate holds if the underlying geometry allows waves to spread out enough so that the energy within compact spatial regions decays sufficiently quickly to be integrable in time. Local energy decay has been used to establish other dispersive estimates such as Strichartz estimates (global, mixed norm estimates) in \cite{mettat}, \cite{toh}, and \cite{mametato} and pointwise estimates in \cite{dafrod2}, \cite{metato}, \cite{tat2013} (among others). 
	
	The local energy decay estimate is known to hold in several nontrapping geometries. For sufficiently small perturbations of flat space without trapping, local energy decay was established in \cite{alinhac}, \cite{metsog06}, and \cite{mettat1}. The case of stationary product manifolds was considered in \cite{burq1998}, \cite{bonyhaf1}, and \cite{sogwan}. The nontrapping case was studied more generally in \cite{mst}. If trapping occurs then the local energy decay estimate does not hold (\cite{ralston}, \cite{sbierski2015characterisation}).

	\subsubsection{Weak Local Energy Decay}
		Trapping on the background geometry may be stable or unstable. A spacetime with trapping where every trapped geodesic is unstable may still admit a weaker form of the local energy decay estimate. In the case of trapping, there is necessarily a loss of derivatives on the right hand side of the estimate (see e.g. \cite{bcmp}). 	
			\begin{definition}
	 		We say the evolution \eqref{inhcauchy} satisfies the weak local energy decay estimate if:
				\begin{equation} \label{wled}
	 				\| u\|_{LE^{1,N}} \leq c_N (\|u[0]\|_{\dot H^{N+3,1} \times H^{N+3}} + \|f\|_{LE^{*,N+3}} ), \qquad N \geq 0.
				\end{equation}
		\end{definition}

\subsection{Vector Fields and Weighted Sobolev Spaces}
	Our argument will use vector field methods. Specifically we are interested in the vector fields:
	\begin{itemize}
		\item Rotations: $\Omega = \{\Omega_{ab} \hbox{ } | \hbox{ } a, b = 1,2,3\}$ where $\Omega_{ab} = x_a\partial_b - x_b\partial_a$.
		\item Translations: $T = \{T_a \hbox{ }|\hbox{ } a = 1,2,3  \}$ where $ T_a = \nabla_a$.
		\item Scaling: $S = S_r - S_{\tau}$ where $ S_r = r\partial_r$ and $ S_{\tau} = \tau \partial_\tau.$
	\end{itemize}
Note that the scaling vector field we use is taken in time frequency space and therefore differs from the scaling vector field in physical space which is given by $r\partial_r + t \partial_t$. This is because we use the vector field arguments only on the time Fourier transform side. We denote the collection of all such vector fields by $\Gamma = \{ \Omega, T, S \} $. We write $\Gamma^{<n}$ to denote a linear combination of $\Gamma^\alpha$ for $|\alpha| < n$: $\Gamma^{<n} := \sum_{|\alpha| < n} c_\alpha \Gamma^\alpha$.
		
	We use the vector fields to define a weighted Sobolev type norm. We will assume the initial data lies in such a space. The weighted Sobolev spaces $Z^{n,q}$ are defined by
			\begin{equation}
				\| \phi \|_{Z^{n,q}} = \sup_{i + j + k \le n} \| \la r \ra^q T^i \Omega^j S_r^k \phi \|_{\LEs}.
			\end{equation}

\subsection{Symbol Classes}
	We will assume that the metric coefficients of the background geometry belong to certain symbol classes.  
	
	The symbol classes $S(r^q), \ell^1S(r^q), S(\log r)$ are defined as follows:
			\begin{align*}
				&\phi(x) \in S(r^q)  \Leftrightarrow \|\la r \ra^{j-q} \partial^j f(x) \|_{L^{\infty}(\R^3)} \lesssim_j 1 \qquad j \in \{ 0, 1, 2, \dots\} \\
				&\phi(x) \in \ell^1 S(r^q)  \Leftrightarrow \sum_m 2^{m(j-q)}  \| \partial^j f(x) \|_{L^{\infty}(A_m)} \lesssim_j 1 \qquad j \in \{ 0, 1, 2, \dots\} \\
				&\phi(x) \in S(\log r) \Leftrightarrow \|(\log\la r \ra)^{-1} f(x) \|_{L^{\infty}(\R^3)} \lesssim 1 \hbox{ and } \|\la r \ra^j \partial^j f(x) \|_{L^{\infty}(\R^3)} \lesssim_j 1,\\
				& \qquad \qquad \qquad \qquad \qquad \qquad \qquad \qquad \qquad \qquad \qquad \qquad j \in \{1, 2, 3, \dots\}.
			\end{align*}
		If $\phi \in S(r^q)$ is radial, we write $\phi \in S_{rad}(r^q)$. We indicate radial functions in the other symbol classes analogously. 
		
	In some of our calculations we use the notation $\rho_\ell^q$ to indicate a representative of the symbol class $\ell^1S(r^q)$. Similarly, we use $\rho^q$ to represent $S(r^q)$ and $\rho_r^q$ to represent $S_{rad}(r^q)$. We allow $\rho_\ell^q$, $\rho_r^q$, and $\rho^q$ to stand for different functions at each appearance.

\subsection{Statement of Main Theorem} \label{mainthm}
	
	We consider a Lorentzian metric $\g$ with the following properties:
	
	\noindent \underline{Metric Assumptions}
	\begin{enumerate}
		\item $\g$ is stationary (i.e. the metric coefficients are time independent).
		\item The submanifolds $t = constant$ are space-like (i.e. the induced metric on the spatial submanifolds is positive definite).
		\item Let $\kappa \in \N$ with $\kappa \ge 2$. The metric $\g$ is asymptotically flat in the sense that $\mathfrak{g}$ can be written as
	\[ \g = \mathfrak{m} + \mathfrak{f} + \mathfrak{h} \]
where
	\[\mathfrak{f} = \mathfrak{f}_{00}(x)dt^2 + \mathfrak{f}_{0i}(x)dtdx_i +  \mathfrak{f}_{ij}(x)dx_idx_j \]
	with $\mathfrak{f}_{\alpha\beta} \in \ell^1 S(r^{-\kappa})$ for $\alpha, \beta \in \{0,1,2,3\}$ and
	\[ \mathfrak{h} = \mathfrak{h}_{tt}(r)dt^2 + \mathfrak{h}_{tr}(r)dtdr + \mathfrak{h}_{rr}(r)dr^2 + \mathfrak{h}_{\omega\omega}(r)r^2d\omega^2 \]
	with $\mathfrak{h}_{\gamma\delta} \in S_{rad}(r^{-\kappa})$ for $\gamma, \delta \in \{t,r,\omega \}$. Here $d\omega^2 = d\theta^2 + \sin^2\theta d\phi^2$.
	\end{enumerate}
	We note these assumptions but with $\kappa =1$ match those in \cite{tat2013}, and thus the results apply here. We will appeal to Tataru's results for steps of the proof of the main theorem which are not sensitive to the rate at which the background geometry tends to flat.
	
	\begin{theorem} \label{thetheorem}
		Let $\g$ be a (1+3)-dimensional spacetime satisfying metric assumptions 1-3 above. Let $V$ be a potential of the form
		 \begin{equation} 
			V(x) = V_{\ell}(x) + V_{r}(r), \qquad V_{\ell} \in \ell^1 S(r^{-\kappa-2}), \quad V_{r} \in S_{rad}(r^{-\kappa-2}). 
		\end{equation}
		Assume the homogeneous Cauchy problem
		\begin{equation} \label{homcauchy}
			(\Box_g +V) u(t,x) = 0, \quad u(0,x) = u_0, \quad \partial_t u(0,x) = u_1 
		\end{equation}
		satisfies the uniform energy bound \eqref{unifenbd} and the weak local energy decay assumption \eqref{wled}. If $u$ solves \eqref{homcauchy} with $u_0 \in Z^{\nu+1,\kappa}$ and  $u_1 \in Z^{\nu,\kappa+1}$ for $\nu \ge 31\kappa+168$, then in normalized coordinates (see section 2) $u$ satisfies the bounds
		\begin{align}
			|u(t,x)| &\lesssim \frac{1}{\la t \ra \la t - r \ra^{\kappa+1}} \big( \|u_0\|_{Z^{\nu+1,\kappa}} + \|u_1\|_{Z^{\nu,\kappa+1}} \big)\\
			|\partial_tu(t,x)| &\lesssim \frac{1}{\la t \ra \la t - r \ra^{\kappa+2}} \big( \|u_0\|_{Z^{\nu+1,\kappa}} + \|u_1\|_{Z^{\nu,\kappa+1}} \big).
		\end{align}
	\end{theorem}

\subsection{Cutoff and Bump Functions}
	The function $\chi_{<1}(r)$ is defined to be a smooth function which is 1 for $r \le 1$ and 0 for $ r \ge 2$. We define $\chi_{>1}(r) := 1 - \chi_{<1}(r)$ so that $\chi_{>1}(r)$ is a smooth function which is 1 for $r \ge 2$ and 0 for $ r \le 1$. We define $\chi_{\approx 1}(r)$ to be a smooth function which is $1$ for $1 \le r \le 2$ and 0 for $r < 2$ and $r > 4$.
	
	We define $\beta_{\approx m}(r)$ for $m \ge 0$ to be a smooth partition of unity which is subordinate to the dyadic intervals $A_m$. 
	
	 When restricting $r$ by dyadic regions, we use $\chi_{\approx m}$ to indicate $\chi_{\approx m}(r) = \chi_\approx (\frac{r}{2^m})$, $\chi_{>m}(r) = \chi_{>1}(\frac{r}{2^m})$, and $\chi_{<m}(r)=\chi_{<1}(\frac{r}{2^m})$. We note $\partial_r \chi_{<m}(r)$ and $\partial_r \chi_{>m}(r)$ are each supported on $\rr \cong 2^m$ so that 
	 \[ \|\partial_r \chi_{<m}(r)\|_{L^2(A_m)} \cong \|\partial_r \chi_{>m}(r)\|_{L^2(A_m)} \cong 2^\frac{3m}{2}. \]	
	 
	 In other contexts where we restrict to $r<R$ or $r>R$ for some constant $R$ we write $\chi_{\approx R} = \chi_\approx (\frac{r}{R})$, etc..

\subsection{Argument Summary}
   	We first fix a coordinate system that allows us to write the operator $(\Box_g + V)$ in the form
	\begin{equation} \label{pform1}
		P = - \partial_t^2 + \Delta + \partial_t P^1 + P^2  
	\end{equation}
    where $P^1$ and $P^2$ are spatial operators of order 1 and 2, respectively. The coefficients of the operators depend on the metric coefficients assumed in the main theorem. We then use the resolvent (denoted $R_\tau$) to connect the time Fourier transform of a solution $u$ to the Cauchy problem \eqref{homcauchy} with the initial data. 

	We define the resolvent to be the inverse of the image of $P$ under the time Fourier transform, when the inverse exists. We will establish that if $u$ solves \eqref{homcauchy} then
	\begin{equation}  
		\hat{u}(\tau) = R_\tau(-i \tau u_0 + P^1 u_0 - u_1). \label{thesetup}
	\end{equation}
    The final pointwise decay rate is then proved by analyzing the resolvent and inverting the Fourier transform. 
    
    Our argument will be different for high frequencies ($|\tau| \gtrsim 1$) and low frequencies ($|\tau| \lesssim 1$).  Roughly speaking, the low frequency behavior is sensitive to the metric behavior at spatial infinity while the high frequency behavior is sensitive to trapping. We assume the weak local energy decay estimate holds so that some trapping may occur, but this estimate provides enough information to obtain decay for the high frequency part of our solution $u$. It is the low frequency behavior that depends on the metric perturbation at spatial infinity and dictates the pointwise decay rate. We obtain an expansion in powers of $r^{-1}$ for the resolvent at zero frequency and use this to calculate the error in the estimate $R_\tau u_0 \approx (R_0 u_0) \enitr$ for the resolvent at low frequencies. We then apply the inverse Fourier transform to the terms arising in this estimate. The behavior of these terms dictates the final pointwise decay rate.

    This approach is due to \cite{tat2013}. A key difference in our analysis is that we need to go further down in the expansion of the zero resolvent in order to obtain improved decay rates. Changing the expansion then affects the error in the estimate for $R_\tau u_0$ when $|\tau|$ is small. The rate at which the background geometry tends toward flat (indicated by the parameter $\kappa$ in the statement of the main theorem) ultimately determines how far down in the expansion of the zero resolvent we are able to go, which determines the error terms in our low frequency resolvent estimate and in turn determines the result of inverting the Fourier transform.

    \underline{Outline of the Paper}
    \begin{itemize}
	\item Section 2: Replace $\Box_\mathfrak{g} +V$ by $P$ as in \eqref{pform1}.
	\item Section 3: Define the resolvent and appeal to \cite{tat2013} to establish the desired mapping properties.
	\item Section 4: Analyze the resolvent at zero frequency. This analysis depends on the rate at which the background geometry tends toward flat.
	\item Section 5: Estimate the low frequency resolvent by $R_\tau u_0 \approx (R_0 u_0)\enitr$ and calculate the error using the results of section 4.
	\item Section 6: Establish pointwise bounds on derivatives of the resolvent that will be used when inverting the Fourier transform. We do not improve upon the bounds established in \cite{tat2013}, but we do track the resulting regularity requirements more precisely and correct one proposition statement.
	\item Section 7: Invert the Fourier transform to prove theorem \eqref{thetheorem}.
    \end{itemize}
	
\end{section}

\begin{section}{Coordinate Change}
In this section we establish a normalized coordinate system in which the operator $\Box_g +V$ in the statement of Theorem \eqref{thetheorem} can be replaced by an operator $P$ of the form
    \begin{equation} \label{pform3}
		    P = -\partial_t^2 + \Delta + \partial_t P^1 + P^2 
	\end{equation}
where
		\begin{equation} \label{p1form}
			P^1 = \partial_i p^i_1 +p_1^i \partial_i, \qquad p^i_1 \in \ell^1S(r^{-\kappa}) 
		\end{equation}
and
        
		\begin{equation} \label{p2form}
		\begin{split}
		    &P^2 =  \partial_i p^{ij}_2 \partial_j + p^\omega_2 \Delta_\omega + V_{\ell} + V_{r},\\
		    &p^{ij}_2 \in \ell^1S(r^{-\kappa}); \quad V_{\ell} \in \ell^1S(r^{-\kappa-2}); \quad p^\omega_2, V_{r} \in S_{rad}(r^{-\kappa-2}).
		\end{split}
		\end{equation}
The calculations in this section encode the geometric assumptions into the differential operator. Throughout the rest of the paper we will work in the normalized coordinates established here. The statement of the main theorem is given in these coordinates.

Metric Assumption 3 (given in the Introduction) can be restated in dual coefficients and using spherical coordinates as
	\begin{equation} \label{dualform} 
	    \g^\gade = \m^\gade + \f^\gade + \h^\gade 
	\end{equation}
where
	\[ 
	\Big[ \m^{\gade} \Big] = 
	\left[ \begin{array} {cccc}
		-1 & 0 & 0 & 0\\
		0 & 1 & 0 & 0\\
		0 & 0 & \frac{1}{r^2} & 0\\
		0 & 0 & 0 & \frac{1}{r^2\sin^2\theta}
	\end{array} \right], \qquad \Big[ \f^{\gade} \Big] = 
	\left[ \begin{array} {cccc}
		\f^{tt} & \f^{tr} & \frac{f^{t\theta}}{r} & \frac{f^{t\phi}}{r\sin\theta} \\
		\f^{rt} & \f^{rr} & \frac{f^{r\theta}}{r} & \frac{f^{r\phi}}{r\sin\theta} \\
		\frac{\f^{\theta t}}{r} & \frac{\f^{\theta r}}{r} & \frac{f^{\theta\theta}}{r^2} & \frac{f^{\theta \phi}}{r^2\sin\theta} \\
		\frac{\f^{\phi t}}{r\sin\theta} & \frac{\f^{\phi r}}{r\sin\theta} & \frac{f^{\phi\theta}}{r^2\sin\theta} & \frac{f^{\phi \phi}}{r^2\sin^2\theta} \\
	\end{array} \right],
	\]
and
	\[ \Big[ \h^{\gade} \Big] =
	\left[ \begin{array} {cccc}
		\h^{tt} & \h^{tr} & 0 & 0\\
		\h^{rt} & \h^{rr} & 0 & 0\\
		0 & 0 & \frac{\h^{\omega\omega}}{r^2} & 0\\
		0 & 0 & 0 & \frac{\h^{\omega\omega}}{r^2\sin^2\theta}
	\end{array} \right]
	\]
with $\f^{\gade} \in \ls{-\kappa}$ and $\h^\gade \in \sr{-\kappa}$. Furthermore $[\f^\gade]$ and $[\h^\gade]$ are symmetric (i.e. $\f^\gade=\f^{\delta\gamma}$ and $\h^\gade=\h^{\delta\gamma}$). 

	\begin{lemma} \label{ncoord}
	There exists a coordinate system so that $\g$ satisfies Metric Assumptions 1-3 as well as the additional condition
		\[ \mathfrak{h}^{rr} = -\mathfrak{h}^{tt}, \quad \mathfrak{h}^{tr} = 0. \]
	\end{lemma}
	\begin{proof}
	In order to achieve $\h^{tr} =0$, we reset $t$ via the coordinate change
			\[ d T = dt - \chi_{>R}\frac{\mathfrak{h}^{rt}}{1+\mathfrak{h}^{rr}} dr \]
    where $R$ is a constant chosen to be sufficiently large so that $1+\mathfrak{h}^{rr} \gtrsim 1 $ for $r>R$. Note we have $\chi_{>R}\frac{\mathfrak{h}^{rt}}{1+\mathfrak{h}^{rr}} \in S_{rad}(r^{-\kappa})$.
    
    To see Assumption 1 still holds, we write the coordinate change as $T = t + Q(r)$ where $Q'(r) = -\chi_{>R}\frac{\mathfrak{h}^{rt}}{1+\mathfrak{h}^{rr}}$. Thus $\parder{}{T} = \parder{}{t}$ so the metric coefficients remain independent of the time variable $T$.
    
    To see Assumption 2 still holds, we calculate
        \[ \la dT, dT \ra = \la dt, dt \ra - 2 \chi_{>R}\frac{\mathfrak{h}^{rt}}{1+\mathfrak{h}^{rr}} \la dt, dr \ra + \left(\chi_{>R}\frac{\mathfrak{h}^{rt}}{1+\mathfrak{h}^{rr}}\right)^2 \la dr, dr \ra. \]
    Choosing $R$ sufficiently large so that $\chi_{>R}\frac{\mathfrak{h}^{rt}}{1+\mathfrak{h}^{rr}}$ is sufficiently small, the sign of $\g^{TT}$ is the same as the sign of $\g^{tt}$. The signature of the metric does not change under the change of coordinates, so the $t=constant$ submanifolds remain positive definite.
    
    To establish Assumption 3 we need only calculate $\g^{T\gamma}$ for $\gamma \in \{ T, r, \theta, \phi \}$ since $r,\theta$, and $\phi$ are unchanged. Direct calculation yields
        \begin{align*}
            &\g^{TT} = -1 + 
                \underbrace{\f^{tt}-2\chi_{>R}\frac{\mathfrak{h}^{rt}}{1+\mathfrak{h}^{rr}} (\f^{tr} + \h^{tr}) + \left(\chi_{>R}\frac{\mathfrak{h}^{rt}}{1+\mathfrak{h}^{rr}}\right)^2 (1 + \f^{rr} + \h^{rr})}_{\in \ls{-\kappa}} + \h^{tt}\\
            &\g^{Tr} = 
                \f^{tr}-\chi_{>R}\frac{\mathfrak{h}^{rt}}{1+\mathfrak{h}^{rr}} \f^{rr} + \chi_{<R} \h^{tr} \in \ls{-\kappa} \\
            &r \g^{T\theta} = 
                \f^{t\theta}- \chi_{>R}\frac{\mathfrak{h}^{rt}}{1+\mathfrak{h}^{rr}}\f^{r\theta} \in \ls{-\kappa}\\
            &r\sin\theta \g^{T\phi} = 
                \f^{t\phi}- \chi_{>R}\frac{\mathfrak{h}^{rt}}{1+\mathfrak{h}^{rr}}\f^{r\phi} \in \ls{-\kappa}.
        \end{align*}
    Thus after relabeling, $\g$ can be written as in $\eqref{dualform}$ with $\h^{tr}=0$, as desired.
    
    Next we achieve $\h^{rr}=-\h^{tt}$ via the coordinate change 
        \[ d\rho = \Big( 1 + \chi_{>R} \frac{-\h^{tt}-\h^{rr}}{1+\h^{rr}} \Big)^{\frac{1}{2}}dr  \]
    where $R$ is a constant chosen so  $1+\h^{rr} \gtrsim 1$ for $ r>R$. Note we have $\chi_{>R} \frac{-\h^{tt}-\h^{rr}}{1+\h^{rr}} \in \sr{-\kappa}$ and $\Big( 1 + \chi_{>R} \frac{-\h^{tt}-\h^{rr}}{1+\h^{rr}} \Big)^{\frac{1}{2}} \in S_{rad}(1)$. The $t=constant$ subspaces are invariant under the change of coordinates and thus remain positive definite and the metric coefficients remain independent of $t$. It follows that Assumptions 1 and 2 still hold.
	
	We now calculate $\g^\gade$ in the new coordinate system. Since $t,\theta$, and $\phi$ are unchanged, we need only calculate $\g^{\rho\gamma}$ for $\gamma \in \{ t, \rho, \theta, \phi \}$. Direct calculation yields
	    \begin{align*}
	        &\g^{\rho \rho} = 1 + 
	            \underbrace{\f^{rr} + \chi_{<R} (\h^{rr} + \h^{tt}) + \chi_{>R} \frac{-\h^{tt}-\h^{rr}}{1+\h^{rr}}\f^{rr}}_{\in \ls{-\kappa}} -\h^{tt}\\
	        &\g^{\rho t} = \Big( 1 + \chi_{>R} \frac{-\h^{tt}-\h^{rr}}{1+\h^{rr}} \Big)^{\frac{1}{2}} \f^{rt} \in \ls{-\kappa} \\
	        &r\g^{\rho\theta} = \Big( 1 + \chi_{>R} \frac{-\h^{tt}-\h^{rr}}{1+\h^{rr}} \Big)^{\frac{1}{2}}\f^{r\theta} \in \ls{-\kappa} \\
	        &r\sin\theta\g^{\rho\phi} = \Big( 1 + \chi_{>R} \frac{-\h^{tt}-\h^{rr}}{1+\h^{rr}} \Big)^{\frac{1}{2}}\f^{r\phi} \in \ls{-\kappa}.
	    \end{align*}
	 After relabeling, $\g$ can now be written as in $\eqref{dualform}$ with the additional assumption $\h^{rr}=-\h^{tt}$, as desired.
	\end{proof}

	\begin{proposition} \label{Pop}	
		In normalized coordinates (established in Lemma~\ref{ncoord}), the operator $\Box_\g +V$ can be replaced by an $L^2$ self-adjoint operator, which can be written as in \eqref{pform3} where \eqref{p1form} and \eqref{p2form} hold.
	\end{proposition}
    \begin{proof}
    The result of the proposition is obtained via direct calculation. We outline the key steps of the process. Converting from spherical to rectangular coordinates, the normalized metric $\g$ can be written as
        \[ \Big[ \g^\albe \Big] = \Big[ \m^\albe \Big] + \Big[ \f^\albe \Big] + \Big[ \h^\albe \Big]  \]
	where
	    \[ \Big[ \m^\albe \Big] = 
	        \left[ \begin{array}{cccc}
		    -1 & 0 & 0 & 0\\
		    0 & 1 & 0 & 0\\
		    0 & 0 & 1 & 0\\
		    0 & 0 & 0 & 1 \end{array}
	           \right], \qquad 
	       \Big[ \f^\albe \Big] = 
	       \left[ \begin{array}{cccc}
	            \f^{tt} & \f^{t1} & \f^{t2} & \f^{t3} \\
	        	\f^{t1} & \f^{11} & \f^{12} & \f^{13} \\
	        	\f^{t2} & \f^{12} & \f^{22} & \f^{23} \\
	        	\f^{t3} & \f^{13} & \f^{23} & \f^{33} \end{array}
            	\right],
	    \]
    and
	    \[ \Big[ \h^\albe \Big] = \left[ \begin{array}{cccc}
		    \h^{tt} & 0 & 0 & 0\\
		0 & \h^{\omega\omega} & 0 & 0\\
		0 & 0 & \h^{\omega\omega} & 0\\
		0 & 0 & 0 & h^{\omega\omega} \end{array}
	\right] - \frac{\h^{tt} + \h^{\omega\omega}}{r^2} \left[ \begin{array}{cccc}
		0 & 0 & 0 & 0 \\
		0 & x_1^2 & x_1x_2 & x_1x_3 \\
		0 & x_1x_2 & x_2^2 & x_2x_3 \\
		0 & x_1x_3 & x_2x_3 & x_3^2
	\end{array}
	\right]
	\]
    with $\f^\albe \in \ls{-\kappa}$ and $\h^\albe \in \sr{-\kappa}$.
    
    To make the operator self-adjoint, we conjugate by $|\g|^{\frac{1}{4}}$ where $|\g| = |\det(\g)|$:
    	\[ \Box_\g + V \to |\g|^\frac{1}{4} (\Box_\g + V) |\g|^{-\frac{1}{4}} \]
    The -1 coefficient of $\partial_t^2$ in \eqref{pform3} is achieved through multiplication by $(-\g^{tt})^{-1}$, but we split this factor into multiplication by $(-\g^{tt})^{-1/2}$ on the left and right so the operator remains self-adjoint. Thus we will replace $\Box_\g + V$ by
	    \[ P = |\g|^{1/4}(-\g^{tt})^{-1/2}(\Box_g + V)(-\g^{tt})^{-1/2}|\g|^{-1/4}. \]
    
    Commuting and writing the operator in divergence form we find
        \begin{equation} \label{pdiv}
            P = \partial_\alpha A^2B^\albe \partial_\beta + A(\partial_\alpha B^\albe)(\partial_\beta A) + AB^\albe (\partial_\albe A) + (\g^{tt})^{-1}V  
        \end{equation}
    where
        \[ A:= (-\g^{tt})^{-1/2}|\g|^{-1/4} \quad \hbox{ and } \quad B^\albe:= |\g|^{\frac{1}{2}}\g^\albe. \]
    Since $|\g| - 1, \g^{tt}+1 \in \ls{-\kappa} + \sr{-\kappa}$, the scalar terms in \eqref{pdiv} are of the form $V_\ell + V_r$ as in \eqref{p2form}. 
    
    Consider the term $\partial_\alpha A^2B^\albe \partial_\beta = \partial_\alpha (-\g^{tt})^{-1}\g^{\alpha\beta}\partial_\beta$ in \eqref{pdiv}. When $\alpha=\beta=t$ we find
        \begin{equation} \label{ptt}
		    \partial_t (-\g^{tt})^{-1} \g^{tt} \partial_t = - \partial_t^2.
	    \end{equation} 
	When either $\alpha = t$ or $\beta = t$ (but not both), the desired form, $\partial_t(\partial_i p_1^i + p_1^i \partial_i)$ with $p_1^i \in \ls{-\kappa}$, follows from the observation
        \[ (-\g^{tt})^{-1} \g^{ti} = \frac{\f^{ti}}{1-\f^{tt}-\h^{tt}} \in \ls{-\kappa}, \qquad i \in \{ 1, 2, 3\}. \]
	   
    When $\alpha, \beta \in \{ 1,2,3 \}$ we use $i, j$ instead of $\alpha, \beta$ since we are only considering spatial terms. If $i = j$ then
        \begin{equation} \label{gttgii}
            (-\g^{tt})^{-1}\g^{ii} = 1 + p_2^{ii} + (\h^{tt}+\h^{\omega \omega})(1-x_i^2r^{-2})
        \end{equation}
    where
        \[ p_2^{ii} := \frac{\f^{ii}+\f^{tt} + (\h^{tt}+\h^{\omega\omega})(1-x_i^2r^{-2})(\f^{tt}+\h^{tt})}{1-\f^{tt}-\h^{tt}} \in \ls{-\kappa}. \]
    If $i \ne j$ then we find
        \begin{equation} \label{gttgij}
		    (-\g^{tt})^{-1}\g^{ij} =  p_2^{ij} - (\h^{tt}+\h^{\omega \omega})x_ix_jr^{-2}, \qquad i \ne j 
	    \end{equation}
	where
	    \[ p_2^{ij} := \frac{\f^{ij} - (\f^{tt} +\h^{tt})(\h^{tt}+\h^{\omega\omega})x_ix_jr^{-2}}{1-\f^{tt}-\h^{tt}} \in \ls{-\kappa}. \]
	
	Combining \eqref{gttgii} and \eqref{gttgij} yields
	    \[ (-\g^{tt})^{-1}\g^{ij} = \delta_{ij} + p_2^{ij} + (\h^{tt} + \h^{\omega\omega})(\delta_{ij}-x_ix_jr^{-2}) \]
	and we find
	    \[ \partial_i (-\g^{tt})^{-1}\g^{ij} \partial_j = \Delta + \partial_i p_2^{ij} \partial_j + p_2^\omega\Delta_\omega \]
    where $p_2^\omega = (\h^{tt} +\h^{\omega\omega})r^{-2} \in \sr{-\kappa-2}$ and $p_2^{ij}$ are as above. This concludes the proof of the proposition.
    \end{proof}
    
    Throughout the rest of the paper we take $P$ to be the operator established in Proposition \ref{Pop}.
\end{section}

\begin{section}{Preliminary Results on the Resolvent}
 The results in this section are not sensitive to the rate at which the background geometry tends toward flat. Instead, they rely on the energy assumptions in the statement of the main theorem. The results of section 2 mean we are interested in the Cauchy problem
	 \begin{equation} \label{pcauchy}
	 	Pu = f, \qquad u(0,\cdot) = u_0, \qquad \partial_t u(0,\cdot) = u_1. 
	 \end{equation}
The evolution satisfies the uniform energy bounds \eqref{unifenbd} and the weak local energy estimate \eqref{wled} since these estimates are coordinate independent. We define the resolvent as follows:
	\begin{definition}
		The operator $P_{\tau}$ associated to $P$ is given by $\partial_t \mapsto i\tau$ so that
		\[ P_{\tau} := \tau^2 + \Delta + i\tau P^1 + P^2. \]
	\end{definition}

	\begin{definition}
		The resolvent associated to $P$, denoted $R_{\tau}$, is defined by
		\[ R_{\tau} := P_{\tau}^{-1}\] when it exists.
	\end{definition}

 The following proposition shows that the uniform energy assumption guarantees that $R_\tau$ exists when $\tau <0$ and links the time Fourier transform of a solution $u$ to the initial data. It is this link that we will exploit in order to establish the final pointwise decay. The results of Proposition \ref{resex} are established in \cite{tat2013} section 3, but we include a proof since \eqref{hatures} is a fundamental piece of our argument.
    
    \begin{proposition} \label{resex}
		Assume \eqref{pcauchy} satisfies the uniform energy bounds \eqref{unifenbd}. If $\Im \tau < 0$, the operator $P_{\tau}: H^2 \to L^2$ is one-to-one and the range of $P_\tau$ is dense in $L^2$. Furthermore, if $u$ satisfies \eqref{pcauchy} then for $\Im \tau <0$ we have
		\begin{equation} \label{hatures}
			\hat{u}(\tau,x) = R_{\tau} (\hat{f}(\tau) -i\tau u_0 + P^1 u_0 - u_1).
		\end{equation}
	\end{proposition}
	\begin{proof}
		Let $Q_{\tau}$ be a family of $\tau$ dependent operators which are defined by $Q_\tau g = \hat{u}(\tau)$ where $u(t,x)$ solves the homogeneous Cauchy problem
		\[ Pu = 0, \quad u(0) = 0, \quad \partial_t u(0) = -g \in L^2. \]
We will show for $\Im \tau <0$ that $Q_\tau P_\tau g = P_\tau Q_\tau g = g$ (i.e. $P_\tau$ is invertible and $R_\tau=Q_\tau$). 
		
		First we establish $L^2$ based bounds on $Q_\tau g$. By assumption, the evolution \eqref{pcauchy} satisfies the uniform energy bounds \eqref{unifenbd}, which translate into $L^2$ based bounds for $Q_\tau g$. Using the notation $\| \phi \|_{\dot{H}^{N,1}} = \| \nabla \phi \|_{H^N}$ and setting $u(t,x)\equiv 0$ for $t<0$ we use the Minkowski Integral Inequality and \eqref{unifenbd} to find for any $N \ge 0$
		\begin{align*}
			\| Q_{\tau} g \|_{\dot{H}^{N,1}} &\le \sum_{j \le N} \int_0^{\infty} \left( \int_{\R^3} \left|e^{-it\tau}\nabla^{j+1} u(t,x) \right|^2 \hbox{ } dx \right) ^{1/2} dt 
			    \lesssim \frac{1}{|\Im \tau|} \| g \|_{H^N}.\\
		\end{align*}
	Similarly we calculate	
		\begin{align*}
			|\tau| \| Q_\tau g \|_{H^N}
			    \le \int_0^\infty e^{t\Im \tau} \|\partial_t u(t,\cdot)\|_{H^N(\R^3)} \dd t  
			    \lesssim \frac{1}{|\Im \tau|} \| g\|_{H^N}. 
		\end{align*}
		Therefore if $g \in H^N$ then $Q_\tau g \in H^{N+1}$ for $\Im \tau < 0$. 
		
		In general taking the time Fourier transform of  $Pu$ (again setting $u(t,x) \equiv 0$ for $t <0$) and integrating by parts yields
				\begin{align*} 
			0 = \int e^{-it\tau} Pu  \dd t = P_\tau \hat{u}(\tau) + \partial_tu(0) + i\tau u(0) - P^1u(0)
		\end{align*}
		so that 
		\begin{equation} \label{ptaugen}
			P_\tau  \hat{u}(\tau) = (Pu)^{\hat{}} - i\tau u(0) + P^1u(0) - \partial_t u(0).
		\end{equation}
		
		Given $g \in H^1$, we have $Q_\tau g \in H^2$, so $Q_\tau g$ is in the domain of $P_\tau$. Applying \eqref{ptaugen} to our definition of $Q_\tau g$, we find $ P_\tau Q_\tau g = g$. Thus $H^1$ is contained in the range of $P_\tau:H^2 \to L^2$, so the range is dense in $L^2$.
		
		Next we aim to show $Q_\tau P_\tau g = g$. To this end, we claim that if $u(t,x)$ solves the nonhomogeneous Cauchy problem 
		\[ Pu = f, \quad u(0,\cdot) = u_0, \quad \partial_t u(0,\cdot) = u_1 \]
		then 
		\begin{equation} \label{hatu}
			\hat{u}(\tau) = Q_\tau (\hat{f}(\tau) - i\tau u_0 + P^1u_0 - u_1).
		\end{equation} 
		
	Once \eqref{hatu} is established, we can show $Q_\tau P_\tau g = g$. Indeed, assume \eqref{hatu} holds and let $g(x) \in H^2$ be given. We define $u(t,x) := g(x)\mathbf{1}_{t\ge 0}$, where $\mathbf{1}_{t \ge 0}$ is an indicator function that is $1$ for $t \ge 0$ and $0$ otherwise. Taking the time Fourier transform of $u(t,x)$ yields $\hat{u}(\tau) = \frac{1}{|\Im \tau|}g$. By \eqref{ptaugen} we have
		\begin{equation} \label{nextref}
		    \frac{1}{|\Im \tau|} P_\tau g = (Pu)^{\hat{}} - i\tau u(0) + P^1u(0) - \partial_t u(0). 
		\end{equation}
		Then applying $Q_\tau$ to \eqref{nextref} and using \eqref{hatu} gives 
		\begin{align*}
			Q_\tau P_\tau g &= |\Im \tau| Q_\tau \Big((Pu)^{\hat{}} - i\tau u(0) + P^1u(0) - \partial_t u(0) \Big) = g,
		\end{align*}
		as desired.
		
		It is left to show \eqref{hatu}. To do this we use Duhamel's formula and find
		\begin{equation} \label{unohat}
			u(t,x) = \int_0^t u_a(t-s,x;s) \hbox{ } ds + \partial_tu_b + u_c + u_d, 
		\end{equation}
	where $u_1 (t,x;s), u_2(t,x), u_3(t,x),$ and $u_4(t,x)$ solve the following:
	\begin{align*}
		Pu_a &= 0 \qquad u_a(0, x; s) = 0 \qquad \partial_tu_a(0,x;s) = - f(s,x) \\
		Pu_b &= 0 \qquad u_b(0,x) =0 \qquad \quad \partial_tu_b(0,x) = u_0 \\
		Pu_c &= 0 \qquad u_c(0,x) =0 \qquad \quad \partial_tu_c(0,x) = -P^1u_0 \\
		Pu_d &= 0 \qquad u_d(0,x) =0 \qquad \quad \partial_tu_d(0,x) = u_1.
	\end{align*}
	Note to find $\partial_tu(0,t)=u_1$, we use $Pu_b = 0$ to write $\partial_t^2 u_b = (\Delta + \partial_tP^1 +P^2)u_b. $
	We calculate $\hat{u}(\tau)$ by taking the time Fourier transform of each term in \eqref{unohat}. For the first term we switch the order of integration, change variables by $t \mapsto t+s$, then switch back the order of integration to find
	\begin{align*}
		\int_0^{\infty} e^{-it\tau} \int_0^t u_a(t-s,x;s) \dd s dt = \int_0^{\infty}e^{-it\tau} \beta(t,x;\tau) dt
	\end{align*}
	where $\beta(t,x; \tau) = \int_0^{\infty} e^{-is\tau}u_a(t,x;s) \dd s$. Since $\beta(t,x;\tau)$ satisfies
	\[ P\beta = 0, \qquad \beta(0,x;\tau) = 0, \qquad \partial_t\beta(0,x; \tau) = -\hat{f}(\tau,x) \]
	we have $\hat{\beta}(\tau,x;\tau) = Q_\tau \hat{f}(\tau).$
Applying the time Fourier transform to the remaining terms in \eqref{unohat} yields
		\[ \hat{u}(\tau) = \hat{\beta}(\tau, x ; \tau) + i\tau \hat{u}_b(\tau) + \hat{u}_c(\tau) + \hat{u}_d(\tau) = Q_\tau \Big(\hat{f}(\tau)-i\tau u_0 + P^1u_0 -u_1\Big) \]
as desired. This concludes the proof of \eqref{hatu} and thus the proof of the proposition.
	\end{proof}

If the weak local energy decay estimate \eqref{wled} also holds, then we are able to obtain $L^2$-based resolvent bounds which are stronger than those established in the proof of Proposition \ref{resex} and are uniform as $\Im \tau \to 0$. This makes it possible to extend the resolvent continuously to the real axis. The results proving the stronger $L^2$-based bounds and the continuous extension of $R_\tau$ to $\tau \in \R$  are established in \cite{tat2013} and are not affected by our change to the assumed rate at which the background geometry tends toward flat. Therefore we state the key results from \cite{tat2013} without proof.
	
The $\mathcal{LE}_{\tau}$ norm, in which we measure the resolvent $v = R_{\tau} g$, is defined by
	\begin{equation} 
		\|v\|_{\mathcal{LE}_{\tau}^N} = \| (|\tau| + \la r \ra^{-1})v \|_{\mathcal{LE}^{N}} + \| \nabla v \|_{\mathcal{LE}^{N}} + \|(|\tau| + \la r \ra^{-1})^{-1}\nabla^2 v \|_{\mathcal{LE}^{N}}. \label{letnorm}
	\end{equation}

	\begin{proposition}[see {\cite[Proposition 9 and Corollary 12]{tat2013}}] \label{resbnd}
		Assume \eqref{pcauchy} satisfies the uniform energy bounds \eqref{unifenbd} and the weak local energy estimate \eqref{wled}. If $\Im \tau \le 0$ and $g \in \mathcal{LE}^{*,N+4}$ for fixed $N \in \N$, then $v = R_{\tau} g$ satisfies
		\begin{equation} \label{resest}
			\|v\|_{\mathcal{LE}_\tau^N} \lesssim \|g\|_{\mathcal{LE}^{*,N+4}}.
		\end{equation}
	\end{proposition}	

	\begin{proposition}[see {\cite[Proposition 10 and Corollary 12]{tat2013}}] \label{resbnd2}
		Assume the Cauchy problem \eqref{pcauchy} satisfies the uniform energy bounds \eqref{unifenbd} and the local energy estimate \eqref{wled}.
		
		If $\Im \tau \le 0$, and $g \in \mathcal{LE^*}$ satisfies
		\begin{equation} \label{gijkbnd}
			\|T^i\Omega^jS^k g \|\lesn{} \lesssim  1, \quad i + 4j +16k < M 
		\end{equation}
		for some positive integer $M$, then
		\begin{equation} \label{vijkbnd}
			\| T^i\Omega^jS^k(R_\tau g) \|\letn{}  \lesssim 1, \quad i+4j + 16k < M-4.
		\end{equation}
	\end{proposition}

\subsection{Strategy to Obtain Decay Rate}
To prove Theorem \ref{thetheorem}, we must prove pointwise decay rates for solutions to the homogeneous Cauchy problem \eqref{pcauchy} with $f=0$. Therefore by \eqref{hatures} we have
	\[
		u(t,x) = \frac{1}{\sqrt{2\pi}}\int_{\Im \tau =- \epsilon} R_\tau(-i\tau u_0 + P^1 u_0 -u_1) e^{it\tau} \dd \tau 
	\]
for $\epsilon >0$. The continuous extension of $R_\tau$ to $\tau \in \R$ allows us to take the limit as $\epsilon \to 0$ to obtain
	\begin{equation} \label{ufin}
		u(t,x) = \frac{1}{\sqrt{2\pi}}\int_{\R} R_\tau(-i\tau u_0 + P^1 u_0 -u_1) e^{it\tau} \dd \tau.  
	\end{equation}
The pointwise decay rates are obtained by writing $e^{it\tau} = \frac{d}{d\tau}(\frac{e^{it\tau}}{it})$ and integrating by parts. Thus the final decay rate will depend on how many times we can integrate by parts, which depends on the regularity of the resolvent. 

	We separate the solution $u$ into low and high frequency parts by defining
	\begin{equation} \label{uhigha}
	    u_{>1}(t,x) := \frac{1}{\sqrt{2\pi}}\int_\R \chi_{<1}(|\tau|)R_\tau(-i\tau u_0 + P^1 u_0 -u_1) e^{it\tau} \dd \tau 
	\end{equation}
and 
	\begin{equation} \label{ulowa} 
	    u_{<1}(t,x) := \frac{1}{\sqrt{2\pi}}\int_\R \chi_{>1}(|\tau|)R_\tau(-i\tau u_0 + P^1 u_0 -u_1) e^{it\tau} \dd \tau.
	\end{equation}
Recall in Theorem \ref{thetheorem} we assume the initial data satisfies $u_0 \in Z^{\nu+1,\kappa}$ and $u_1 \in Z^{\nu,\kappa+1}$. Here $\nu$ is a sufficiently large constant depending on $\kappa$, and $\kappa$ indicates the rate at which the background geometry tends toward flat. The assumptions on the initial data mean we can write
    \[ -i\tau u_0 +P^1 u_0 - u_1 = \tau g_{\kappa}^{\nu+1} + g_{\kappa+1}^\nu \]
for some $g_{\kappa}^{\nu+1} \in Z^{\nu+1,\kappa}$ and some $g_{\kappa+1}^\nu \in Z^{\kappa+1,\nu}$. In sections \ref{zerores} and \ref{lowfreqan} we analyze the resolvent $R_\tau g$ near 0 frequency for general $g$ in an appropriately defined function space.
\end{section}

\begin{section}{The Zero Resolvent} \label{zerores}
In this section we obtain an expansion of $R_0 g$ in powers of $\rr^{-1}$. In section \ref{lowfreqan} we will approximate $R_\tau g \approx R_0 g e^{-i\tau r}$ for small $\tau$ and calculate the error using this expansion. Note for the arguments in this and the following section we harmlessly assume $r \ge 2$, since the main theorem holds for $r \lesssim 1$ using weak local energy decay and Sobolev embeddings. 

In our argument establishing an expansion of $R_0g$ for large $r$ we will find
	\[ (-\Delta) (\chi_{>R}R_0g) = h + \chi_{>R/2}P^2 (\chi_{>R}R_0g) \]
where $\|h\|_{Z^{n,\lambda}} \lesssim \|g\|_{Z^{n+4,\lambda}}$. This motivates Lemma \ref{freeres}, where we will obtain an expansion of $(-\Delta)^{-1}g$ for $g \in Z^{n,\lambda}$. 

\begin{lemma} \label{freeres}
	Let $g \in Z^{n,\lambda}$ with $\lambda,n \in \N$. We have the following representation for $(-\Delta)^{-1}g$:
	\begin{equation} 
		(-\Delta)^{-1}g = \sum_{j=0}^{\lambda-2} \big( {c}_j \cdot \nabla^{j}\la r \ra^{-1} + {e}_j(r)\cdot (\nabla^j \la r \ra^{-1}) \la r \ra^{j-\lambda+1} \big) + {d}(r) \cdot \nabla^{\lambda-1} \la r \ra^{-1} + q(x) \label{freeexp}
	\end{equation}
	where the coefficients satisfy
	\begin{equation} 
			\sum_{j=0}^{\lambda-2} \Big( |c_j| + \|e_j\|_{\ell^1S(1)} \Big) + \|d\|_{L^{\infty}} + \|S_rd\|_{\ell^1S(1)} +  \| q\|_{Z^{n+2,\lambda -2}} \lesssim \|g\|_{Z^{n,\lambda}}.  \label{freebndfin}
	\end{equation}
	For $\lambda = 1$ we have only the last two terms in~\eqref{freeexp}. 
\end{lemma}

\begin{proof}
	Set $v = (-\Delta)^{-1}g$. We begin by proving if $\la r \ra^\lambda g \in \mathcal{LE}^*$, then $v$ can be expressed as in \eqref{freeexp} where the following estimate holds
			\begin{equation}
		\sum_{j=0}^{\lambda-2} |c_j| + \|e_j\|_{\ell^1S(1)} + \|d\|_{L^{\infty}} + \|S_rd\|_{\ell^1S(1)} + \sum_{i \le 2} \| \la r \ra^{-2+\lambda+i} \nabla^i q\|_{\mathcal{LE^*}} \lesssim \| \la r \ra^\lambda g \|_{\mathcal{LE^*}}. \label{freebnd}
	\end{equation}	
	We will then consider the case $g \in Z^{n,\l}$ and prove \eqref{freebndfin} using \eqref{freebnd} and elliptic regularity arguments.
	
	We have $ v(x) \cong \int g(y)\frac{1}{|x-y|} \dd y$ using the kernel for the fundamental solution of the Laplacian in $\R^3$. We wish to bound the coefficients of the representation for $v$ by the size of $\la r \ra^\l g$ measured in $\mathcal{LE}^*$, which is defined by the behavior of $g$ on dyadic regions $A_m$. This motivates the following decomposition of $g$:
	\[ g_m := \beta_{\approx m}(r) g \quad \hbox{ and } \quad v_m := (-\Delta)^{-1} g_m = \int g_m(y) \frac{1}{|x-y|} \dd y. \]
			
	We further decompose each $v_m$ into $\vmlow$ and $\vmhigh$ as follows
	\begin{equation}
	\begin{split} 
		\vmlow &:= \chi_{<m+2}(r) v_m = \chi_{<m+2}(|x|) \int \beta_{\approx m}(|y|) g(y) \frac{1}{|x-y|} \dd y
	\end{split} \label{vmldef}
	\end{equation}
and
	\begin{equation}
	\begin{split} 
		\vmhigh &:= \chi_{>m+2}(r) v_m = \chi_{>m+2}(|x|) \int \beta_{\approx m}(|y|) g(y) \frac{1}{|x-y|} \dd y.
	\end{split} \label{vmhdef}
	\end{equation}
Thus we have $ v \cong \sum_{m\ge 0} \vmlow + \vmhigh$.

	We claim 
	\begin{equation}
		\sum_{i\le 2} \Big\| \rr^{-2+\lambda+i}\nabla^i \sum_{m\ge 0} \vmlow \Big\|\les \lesssim \Big\|\rr^{\lambda} g \Big\|\les \label{vlowgoal}
	\end{equation} 
so that $\sum_{m\ge 0}\vmlow$ can be included in the $q(x)$ term in \eqref{freebnd} (we will see that $\sum_{m \ge 0}\vmhigh$ also generates a term that will be included in $q(x)$). To prove \eqref{vlowgoal} we begin by calculating
	\begin{align*}
		|(-\Delta) \vmlow| &\lesssim |g_m| + 2^{-m}|\chi'_{<m+2}(|x|)|\int \Big|\bm(\ry)g(y)\Big|\frac{1}{|x-y|^2} \dd y\\
				&\qquad + 2^{-2m}|\chi''_{<m+2}(\rx)|\int \Big|\bm(\ry)g(y)\Big|\frac{1}{|x-y|} \dd y\\
			&\lesssim |g_m| + 2^{-\frac{3m}{2}} |\chi'_{<m+2}(|x|)| \| \bm g \|_{L^2} + 2^{-\frac{3m}{2}}|\chi''_{<m+2}(\rx)| \| \bm g \|_{L^2}.
	\end{align*}
The second inequality is obtained using the Cauchy-Schwarz inequality and the fact that the cutoffs in $x$ and $y$ give $|x-y|^{-1} \lesssim 2^{-m}$. Integration by parts then yields $ \| \nabla^2 \vmlow \|_{L^2} \lesssim \| g_m \|_{L^2}$ (there are no boundary terms since $v_m^{low}$ is compactly supported). Next we use the Hardy inequality and the fact that $\vmlow$ is supported on a bounded region to find 
	\begin{equation} \label{poin}
	    \| \nabla^i \vmlow \|_{L^2} \lesssim 2^{m(2-i)}\|g\|_{L^2({A_{m}})}, \quad i = 0,1,2.
	 \end{equation}
	
	Note if $|x| \approx 2^k$ and $|y| \approx 2^m$ with $k \le m-1$, we have $|\nabla^i |x-y|^{-1}| \lesssim 2^{-m(1+i)}$ for $i = 0,1,2$ and $ \| \nabla^i \vmlow \|_{L^2(A_k)} = \| \nabla^i v_m \|_{L^2(A_k)}.$ Using the Cauchy-Schwarz inequality, we find
	\begin{equation}
	\begin{split}
		\| \nabla^i v_m \|_{L^2(A_k)} \lesssim 2^{\frac{3k}{2}} 2^{m(\frac{1}{2}-i)}  \| g\|_{L^2(A_{m})}, \quad i = 0,1,2
	\end{split} \label{vmak}
	\end{equation}
when $k \le m-1$. Now we use \eqref{poin} and \eqref{vmak} to calculate
	\begin{align*}
		\| \la r \ra^{-2+\lambda + i} \nabla^i \sum_m \vmlow \|\les &\lesssim \sum_m \sum_{k< m+2} 2^{\frac{k}{2}} 2^{k(-2+\lambda + i)} \| \nabla^i \vmlow \|_{L^2(A_k)}  \\
			&\lesssim \sum_m \sum_{k<m} 2^{k(\lambda + i)} 2^{m(\frac{1}{2}-i)} \| g \|\ltwoam\\
				&\qquad + \sum_m \sum_{k \approx m}  2^{m(-\frac{3}{2}+\lambda + i)} \| \nabla^i \vmlow\|_{L_2(A_{\approx m})}\\
			&\lesssim \|\rr^\lambda g \|\les
	\end{align*}
	for $i = 0, 1, 2$. This concludes the proof of \eqref{vlowgoal}.
	
	Next we turn our attention to $\sum_m v_m^{high}$. For each $m$ we integrate over $|y| \approx 2^m$ with $|x| \ge 2^{m+2}$, so $\frac{1}{|x-y|}$ is smooth in the region of integration and Taylor's theorem applies. We define $y^j$ as follows. If the $n^{th}$ component of $\nabla^{j}$ is $\partial_{i_1}\partial_{i_2}\cdots \partial_{i_j}$ then the $n^{th}$ component of $y^j$ is $y_{i_1}y_{i_2}\cdots y_{i_j}$. Note that $|y^j| \lesssim |y|^j$. In this notation, Taylor's theorem yields
	\[ \frac{1}{|x-y|} = \sum_{j=0}^{\lambda-1} \frac{\nabla^j |x|^{-1} \cdot (y)^j}{j!} + R_{\lambda}^{x}(y) \]
where 
	\[ R_{\lambda}^{x}(y) = \lambda !^{-1} \nabla^{\lambda}\left(|x-ty|^{-1}\right)\cdot (y)^\lambda \]
for some $t \in (0,1)$. Therefore 
	\begin{equation} \label{vmhighexp}
	    v_m^{high}(x) = \sum_{j=0}^{\lambda-1} j!^{-1} \chi_{>m+2}(|x|) \int g_m(y)  (\nabla^j |x|^{-1}) \cdot y^j  \dd y + \chi_{>m+2}(|x|) \int g_m(y) R_{\lambda}^{x}(y) \dd y.
	\end{equation}
	
	We claim the last term in \eqref{vmhighexp} can be included in $q$ after summing over $m$. In other words, we wish to show
	\begin{equation} 
		\sum_{i\le 2} \| \rr^{-2+\lambda+i} \nabla^i \ttt{q} \|\les \lesssim \|\rr^\lambda g\|\les, \label{tsremgoal}
	\end{equation}
where $\ttt{q}$ is defined by
	\begin{align*} 
	    \tilde{q} := \sum_m \tilde{q}_m, \quad \hbox{and} \quad \ttt{q}_m := \chi_{>m+2}(r) \int g_m(y) R_{\lambda}^{x}(y) \dd y.
	\end{align*}
Note $|\nabla^\lambda(|x|^{-1})| \lesssim |x|^{-\lambda-1}$ and $|x-ty|^{-1} \lesssim |x|^{-1}$ since $|x| \ge 2|y|$ and $t \in (0,1)$. It follows that
	\begin{equation}
		\left| R_{\lambda}^{y}(x) \right| \lesssim \frac{|y|^{\lambda}}{|x|^{\lambda+1}}, \label{rembnd}
	\end{equation}
which implies for $l>m+2$
	\begin{equation} \label{tqmlinf}
		\| \tilde{q}_m \|_{L^{\infty}(A_l)} \lesssim 2^{-l(\lambda+1)} \| r^\lambda g_m\|_{L^1} \lesssim 2^{(m-l)(\lambda+1)} \|\la r \ra^{\frac{1}{2}} g_m\|_{L^2}.  
	\end{equation}
Straightforward calculation yields
	\begin{equation} 
		\| \la r \ra^{-2+\lambda} \tilde{q} \|_{\mathcal{LE}^*} \le \sum_m \sum_{l>m+2} \| \tilde{q}_m \|_{L^{\infty}(A_l)} 2^{l\lambda}. \label{tqles}
	\end{equation}
Combining \eqref{tqmlinf} and \eqref{tqles} then gives $\| \la r \ra^{-2+\lambda}\tilde{q}(x) \|_{\mathcal{LE}^*} \lesssim \| g \|_{\mathcal{LE^*}},$ as desired. The estimates for $\nabla^i \tilde{q}$ for $i =1,2$ follow analogously once we note $|\nabla^i R_{\lambda}^{y}(x)| \lesssim \frac{|y|^{\lambda}}{|x|^{\lambda+i+1}}$.
	
    To handle the first term in \eqref{vmhighexp} we define
	\[ c_j := \sum_m (j!)^{-1} \int g_m(y)y^j \dd y, \qquad {e}_j(r) := \sum_m (j!)^{-1} \rr^{\lambda-1-j}(-\chi_{<m+2}) \int g_m(y)y^j \dd y,\] \[ \hbox{and } \quad d(r) := \sum_m ((\lambda-1)!)^{-1} \chi_{>m+2}(r) \int g_m(y) y^{\lambda-1} \dd y, \]
	so
	\begin{align*} 
		\sum_m \sum_{j=0}^{\lambda-1}& (j!)^{-1} \chi_{>m+2}(|x|) \int g_m(y)(\nabla^j |x|^{-1}) y^j \dd y\\
			&= \sum_{j=0}^{\lambda-2} \Big( {c}_j \cdot \nabla^j \rr^{-1} + {e}_j(r) \cdot (\nabla^j \rr^{-1})\rr^{j-\lambda +1} \Big) + {d}(r)\cdot \nabla^{\lambda-1} \rr^{-1}. 
	\end{align*} 
	The desired bounds \eqref{freebnd} then follow directly using the inequality
		$ \| \la r \ra^p g \|_{L^1(A_m)} \lesssim \| \la r \ra^{p+\frac{3}{2}} g \|_{L^2(A_m)}, $
which is a straightforward application of Cauchy-Schwarz.

	Now let $g \in Z^{n,\lambda}$. Since $\|\la r \ra^\lambda g\|_{\mathcal{LE}^*} \lesssim \|g\|\znl{n}{\lambda}$, the above argument shows that $v$ admits a representation as in \eqref{freeexp} such that \eqref{freebnd} holds. To prove \eqref{freebndfin} we need to show the $q$ term now satisfies
	\begin{equation}
			\| q\|\znl{n+2}{\l-2} \lesssim \|g\|\znl{n}{\l}. \label{qgoal}
		\end{equation}	
Observe that for any $\phi$ with sufficient differentiability we have
	\begin{equation} \label{vftot}
		|T^i\Omega^jS_r^k \phi| \lesssim \sum_{l=0}^{i+j+k} |\rr^l \nabla^l \phi|. 
	\end{equation}
By \eqref{freebnd} and \eqref{vftot} we find
    \begin{equation}
		\|q\|\znl{2}{\l-2} \lesssim \|g\|\les = \|g\|\znl{0}{\l}. \label{qprebnd}
	\end{equation}
Furthermore we can use the definition of $Z^{n,\lambda}$ and \eqref{vftot} to establish
    \begin{equation} \label{goalredux}
        \|q\|\znl{n+2}{\lambda -2} \lesssim \|q\|\znl{2}{\lambda-2} + \|\nabla^2 q\|\znl{n}{\lambda}.
    \end{equation}
	
	Next we calculate
	\begin{equation}
	\begin{split}
		\| \nabla^2 q &\|_{Z^{n,\lambda}} 
		    \lesssim \sup_{i + j +k \le n} \sum_m 2^{\frac{m}{2}} 2^{m\lambda} \|  T^i \Omega^j S_r^k  \nabla^2 ( \chi_{\approx m} q ) \|_{L^2(\R^3)} \\
		    &\approx \sup_{i + j +k \le n} \sum_m 2^{\frac{m}{2}} 2^{m\lambda} \|  T^i \Omega^j S_r^k  \Delta ( \chi_{\approx m} q ) \|_{L^2(\R^3)} \\
		    &\lesssim \supn{n} \sum_m 2^{m(\lambda +\frac{1}{2})} \Big( \|\vf \chi_{\approx m} \Delta q\|_{L^2(\R^3)} + \|\vf (\Delta \chi_{\approx m}) q \|\ltwor \\
				&\qquad + \| \vf (\nabla \chi_{\approx m}) \cdot \nabla q \|\ltwor  \Big)\\
			&\lesssim \supn{n} \sum_m 2^{m(\lambda +\frac{1}{2})} \Big(  \|\vf \chi_{\approx m} \Delta q\|_{L^2(\R^3)} + \|2^{-2m} \vf  \chi''\left( \frac{r}{2^m} \right) q \|\ltwor\\
				&\qquad + 2^{-m}\| \vf  r^{-1} \chi'\left(\frac{r}{2^m}\right) q \|\ltwor + 2^{-m}\| \vf  \chi'\left(\frac{r}{2^m}\right) \nabla q \|\ltwor \Big)\\
			&\lesssim \|\Delta q \|\znl{n}{\l} + \|\nabla q\|\znl{n}{\l-1} + \|q\|\znl{n}{\l-2} 
	\end{split} \label{ellregbnd}
	\end{equation}
for all $n \ge 1$. The second line in \eqref{ellregbnd} follows from 
	\[ \| T^i \Omega^j S_r^k \nabla^2 (\chi_{\approx m} q) \|_{L^2(\R^3)} \approx \| T^{\le i} \Omega^{\le j} S_r^{\le k} \Delta (\chi_{\approx m} q)\|_{L^2(\R^3)}, \]
which is obtained via integration by parts using the commutators 
	\[ [\partial_i, \Omega],  [ \partial_i, S_r] \in \hbox{ span } \{ T \}, \quad [ \Delta, \Omega ] = [\Delta, T] = 0, \quad \hbox{and} \quad [\Delta, S_r] = 2 \Delta . \]

    We claim
    \begin{equation}
		\|\Delta q\|\znl{n}{\l} \lesssim \|g\|\znl{n}{\l}. \label{deltaq}
	\end{equation}
Once \eqref{deltaq} is established, \eqref{qgoal} follows by induction in $n$. Indeed, assume \eqref{deltaq} holds. When $n=1$, the RHS of \eqref{ellregbnd} is controlled by $\|g\|\znl{1}{\l}$ using \eqref{deltaq} to bound the first term and \eqref{qprebnd} to bound the last two terms (along with the fact $ \|\nabla \phi \|\znl{n}{\l} \lesssim \|\phi\|\znl{n+1}{\l-1}$
for $\phi \in Z^{n+1,\l-1}$). Thus $\|\nabla^2q\|\znl{1}{\lambda} \lesssim \|g\|\znl{1}{\l}$, so by \eqref{qprebnd} and \eqref{goalredux} we obtain \eqref{qgoal} in the case $n=1$. The inductive step is proved in the same manner, with the inductive hypothesis being used to bound the last two terms on the RHS of \eqref{ellregbnd}.

    To prove \eqref{deltaq} we write
    \begin{equation} 
			-\Delta q = g + \Delta(v-q). \label{delqeq}
		\end{equation}
By \eqref{freeexp} we have
    \[ -\Delta(v-q) = -\Delta \left( \sum_{j=0}^{\lambda - 2}\Big( {c}_j \cdot \nabla^{j}\la r \ra^{-1} + {e}_j(r)\cdot (\nabla^j \la r \ra^{-1}) \la r \ra^{j-\lambda+1} \Big) + {d}(r) \cdot \nabla^{\lambda-1} \la r \ra^{-1}   \right).  \]
Direct calculation using \eqref{freebnd} yields
    \begin{align*}
		\|\Delta (c_j\nabla^j\rr^{-1})\|\znl{n}{\l} &\lesssim \| \rr^\l	g\|\les \\
		\|\Delta (e_j(\nabla^j\rr^{-1})\rr^{j-\l+1})\|\znl{n}{\l} &\lesssim \|\rr^\l	g\|\les \\
		\|\Delta (d\nabla^{\lambda-1}\rr^{-1})\|\znl{n}{\l} &\lesssim \|\rr^\l	g\|\les, 
	\end{align*}
so that $\|\Delta(v-q)\|\znl{n}{\l} \lesssim \|g\|\znl{0}{\l}$, which combined with \eqref{delqeq} gives \eqref{deltaq}. This concludes the proof of \eqref{qgoal} and thus the proof of the proposition.
\end{proof}

In Propositions \ref{propexp} and \ref{lowferr} we will need expressions for $(-\Delta)^{-1} g$ where $g \in S_{rad}(r^{-q})$ with $ q \ge 2$, which are stated in the following lemma.

\begin{lemma} \label{delrad}
	Let $g \in S_{rad}(r^{-q})$ and set $v = (-\Delta)^{-1} g$
	\begin{enumerate}
		\item If $q \ge 4$ then there exist a constant $c$ and an $e \in S_{rad}(1)$ such that
			\[ v = c \rr^{-1} + e(r) \rr^{-(q-2)}. \]
		\item If $q =3$ then there exists an $\epsilon \in S_{rad}(\ln r)$ such that
			\[ v =  \epsilon(r)\rr^{-1}. \]
		\item If $q = 2$, then $v \in S_{rad}(\ln r)$.		
	\end{enumerate} 
\end{lemma}

    A full proof of Lemma \ref{delrad} is omitted. To prove parts $1$ and $2$, write $\partial_r^2(rv) = -rg$ so \[ rv = -\int_0^r \int_s^\infty \rho g(\rho) \dd \rho ds. \]
    Similarly, to prove part $3$, write
        \[ v= -\int_0^r s^{-2} \int_0^2 \rho^2g(\rho) \dd \rho ds.  \]

We now prove an analogue of Proposition \ref{freeres} with $\Delta^{-1}$ replaced by $R_0 = (\Delta + P^2)^{-1}$. Here we first see the role of long range metric behavior - the faster the metric tends to flat, the further we can go down in the expansion for $R_0g$.

\begin{proposition} \label{propexp}
	Let $g \in Z^{n+4,\lambda}$ with $\lambda, n \in \N$. Take $v = R_0 g$.  
	
	If $1 \le \lambda \le \kappa $, then for large $r$, $v$ can be written as in~\eqref{freeexp} where the following estimate holds
		\begin{equation} 
			\sum_{j=0}^{\l-2} \Big( |c_j| + \|e_j\|_{\ell^1S(1)} \Big) + \|d\|_{L^{\infty}} + \|S_rd\|_{\ell^1S(1)} +  \| q\|_{Z^{n+2,\lambda-2}} \lesssim \|g\|_{Z^{n+4,\lambda}}.  \label{zeroexpest}
		\end{equation} 
	
	If $\lambda = \kappa+1$, then for large $r$, $v$ can be written as in~\eqref{freeexp} where the following estimate holds
		\begin{equation} 
			  \sum_{j=0}^{\l-2} |c_j| + \|e_0\|_{S(1)} +  \sum_{j=1}^{\l-2} \|e_j\|_{\ell^1S(1)} + \|d\|_{L^{\infty}} + \|S_rd\|_{\ell^1S(1)} +  \| q\|_{Z^{n+2,\lambda-2}} \lesssim \|g\|_{Z^{n+4,\lambda}}.  \label{zeroexpest1}
		\end{equation} 
\end{proposition}

	Before proving the proposition we provide a brief summary of the argument. Since we are concerned only with large $r$, we consider $\chi_{>R} v=:w$ and show
	\[ -\Delta w = h + \chi_{>R/2} P^2 w \]
	where $\| h\|_{Z^{n,\lambda}} \lesssim \| g \|_{Z^{n+4,\lambda}}$.
	We then use use Lemma \ref{freeres} to obtain the desired form for $w$. For $\lambda \le \kappa $, this works by showing that the above equation for $w$ is perturbative with respect to \eqref{freebndfin}. The $\l = \kappa+1$ case is similar to the case $\lambda \le \kappa$, but there is one term which fails to be perturbative. This non-perturbative term has the benefit of being radial and will be handled using Lemma \ref{delrad}.
\begin{proof}
	Set $w = \chi_{>R} v $ and write 
	\begin{equation} 
		P_0 w = \chi_{>R}g + [P_0, \chi_{>R}]v =: h. \label{p0wh}
	\end{equation}
	Recall that $P_0 = \Delta + P^2$ with $P^2$ as in \eqref{p2form}. Since $\chi_{>R} \equiv 0$ for $r<R$ and  $\chi_{>R} \equiv 1$ for $r>2R$, the commutator $[P_0,\chi_{>R}]$ is supported on $[R, 2R]$. Direct calculation yields  
		\begin{equation}
		\begin{split} 
			\| \la r \ra^\lambda [P_0,\chi_{>R}] v \|_{\LE^*} &\lesssim \sum_{m=\log R}^{\log 2R} \Big[ \| \la r \ra^{\frac{1}{2}+\lambda} (R^{-2} +R^{-1} \rho_\ell^{-\kappa-1} + R^{-2}\rho_\ell^{-\kappa}) v\|_{L^2(A_m)} \\
				&\qquad + \| \la r \ra^{\frac{1}{2}+\lambda}( R^{-1}+ R^{-1} \rho_\ell^{-\kappa}) \nabla v\|_{L^2(A_m)} \Big] \\
			&\lesssim_R \| \la r \ra^{-1} v\|\len{} + \| \nabla v\|\len{} \\
			&\lesssim \|v\|_{\mathcal{LE}_0}
		\end{split} \label{vle0}
		\end{equation}
where 
	\[ \|v\|_{\mathcal{LE}_0} = \|\rr^{-1}v\|\len{} + \| \nabla v\|\len{} + \|\rr \nabla^2 v\|\len{} \]
by the definition of the $\mathcal{LE}_\tau$ norm in \eqref{letnorm}. The constants in the inequalities in \eqref{vle0} depend on $R$, but this is not an issue since $R$ is fixed. Now \eqref{p0wh} and \eqref{vle0} yield
		\[ \| \la r \ra^\lambda h\|_{\mathcal{LE}^*} \lesssim \| \la r \ra^\lambda g\|\lesn{} + \| v \|_{\mathcal{LE}_0} \lesssim \| \la r \ra^\lambda g \|\lesn{4} \]
		since $\| v\|_{\mathcal{LE}_0} \lesssim \| \la r \ra g\|_{\mathcal{LE}^{*,4}}$ by Proposition \ref{resbnd}. 
		
		Similarly we find
		\begin{align*}
			\|&h\|\znl{n}{\l} \lesssim \|g\|\znl{n}{\l} + \supn{n} \|\rr^\l \vf [P_0,\chi_{>R}] v\|\les \\
				&\lesssim \|g\|\znl{n}{\l} + \supn{n} \|\rr^\l [P_0,\chi_{>R}] \vf  v\|\les + \|\rr^\l \big[\vf, \rr^\l [P_0,\chi_{>R}]\big]  v\|\les  \\
				&\lesssim_R \|g\|\znl{n}{\l} + \sup_{i \le n} \|\rr^\l T^i  v\|_{\mathcal{LE}^*([R,2R]) }\\
				&\lesssim_R \|g\|\znl{n}{\l} + \| v \|_{\mathcal{LE}_0^n}
		\end{align*}
where $\mathcal{LE}^*([R,2R])$ indicates the $\mathcal{LE}^*$ norm restricted to $R \le r \le 2R$. The last inequality follows due to the commutator $[P_0,\chi_{>R}]$ being supported on $[R,2R]$ and \eqref{vftot}. Thus by Proposition \ref{resbnd} we have
		\begin{equation} \label{hbndg}
			\| h\|_{Z^{n,\lambda}} \lesssim_R \| g\|_{Z^{n+4,\lambda}}. 
		\end{equation}

		Rewriting the equation $P_0 w = h$ using $P_0 = \Delta + P^2$, we obtain
	\begin{equation}
		-\Delta w = -h + \chi_{>R/2} P^2 w. \label{nearinfty}
	\end{equation}
Note $\chi_{>R/2}P^2 w = P^2 w$  because the support of $P^2 w$ is contained in the region where $\chi_{>R/2}=1$.
	
	By Lemma \ref{freeres}, $w$ can be written as in \eqref{freeexp} where \eqref{freebndfin} holds so we have
		\begin{equation} \label{uselem}
		\begin{split}
			&\sum_{j=0}^{\lambda-2} |c_j| + \|e_j\|_{\ell^1S(1)} + \|d(r)\|_{L^{\infty}} + \|S_rd\|_{\ell^1S(1)} + \| q \|_{Z^{n+2,\lambda-2}} \\
				&\qquad \le C \left\|  \chi_{>{\frac{R}{2}}}P^2 \left( \sum_{j=0}^{\lambda-2} \left( {c}_j \cdot \nabla^{j}\la r \ra^{-1} + {e}_j\cdot (\nabla^j \la r \ra^{-1}) \la r \ra^{j-\lambda+1} \right) + {d} \cdot \nabla^{\lambda-1} \la r \ra^{-1} + q \right) \right\|_{Z^{n,\lambda}}\\
				&\qquad \qquad + C_R \| h \|_{Z^{n,\lambda}}.
		\end{split} 		
		\end{equation}
If we can show
	\begin{equation} 
	\begin{split}
		&\left\| \chi_{>{\frac{R}{2}}}P^2 \left( \sum_{j=0}^{\lambda-2} \left( {c}_j \cdot \nabla^{j}\la r \ra^{-1} + {e}_j\cdot (\nabla^j \la r \ra^{-1}) \la r \ra^{j-\lambda+1} \right) + {d} \cdot \nabla^{\lambda-1} \la r \ra^{-1} + q \right) \right\|_{Z^{n,\lambda}}\\
		&\qquad \lesssim R^{-1}\left( \sum_{j=0}^{\lambda-2} |c_j| + \|e_j\|_{\ell^1S(1)} + \|d\|_{L^{\infty}} + \|S_rd\|_{\ell^1S(1)} + \| q \|_{Z^{n+2,\lambda-2}} \right) 
	\end{split} \label{pertest}
	\end{equation}
	then choosing $R$ sufficiently large allows us to bootstrap the perturbative term in \eqref{pertest}, and we can use \eqref{hbndg} to obtain
	\[ \begin{split}
		\sum_{j=0}^{\l-2} (|c_j| + \|e_j\|_{\ell^1S(1)}) + \|d(r)\|_{L^{\infty}} + \|S_rd\|_{\ell^1S(1)} + \| q \|_{Z^{n+2,\lambda-2}} \le (1-cR^{-1})^{-1} C_R  \| g \|_{Z^{n+4,\lambda}}
	 \end{split} \]
as desired. Thus we wish to show
	\begin{align}
		\| \chi_{>\frac{R}{2}} P^2 c_j\cdot \nabla^j \rr^{-1} \|\znl{n}{\l} &\lesssim R^{-1}|c_j|, \qquad 0 \le j \le \l -2 \label{cpert} \\
		\| \chi_{>\frac{R}{2}} P^2 e_j \cdot (\nabla^j \la r \ra^{-1}) \la r \ra^{j-\lambda+1}\|\znl{n}{\l} &\lesssim R^{-1} \|e_j\|_{\ell^1S(1)}, \qquad 1 \le j \le \l -2 \label{epert} \\
		\| \chi_{>\frac{R}{2}} P^2d \nabla^{\l-1}\rr^{-1}\|\znl{n}{\l} &\lesssim R^{-1}(\|d\|_{L^{\infty}} + \|S_rd\|_{\ell^1S(1)}) \label{dpert} \\
		\| \chi_{>\frac{R}{2}} P^2 q \|\znl{n}{\l} &\lesssim R^{-1} \|q\|\znl{n+2}{\l-2} \label{qpert}.
	\end{align}
We will see \eqref{epert}, \eqref{dpert}, and \eqref{qpert} hold for $\l \le \kappa+1$. And we will see \eqref{cpert}  holds for $\l \le \kappa+1$ when $1 \le j \le \l -2$. When $j = 0$, \eqref{cpert} holds only for $\l \le \kappa$. Our argument handling $c_0 \rr^{-1}$ when $\l = \kappa + 1$ will change the space we can assume $e_0$ is in, causing the difference in the result for $\l = \kappa +1$.
	
	First we prove \eqref{qpert}. Direct calculation yields
	\begin{equation}
	\begin{split}
		\|\la r \ra^\lambda \chi_{>R/2}  P^2 q\|\les &\lesssim \sum_{a \le 2} \| \chi_{>R/2} \la r \ra^\lambda \rho_\ell^{-\kappa -2 + a} \nabla^a q\|\les + \|\chi_{>R/2} \la r \ra^\lambda \rho_r^{-\kappa -2} \Delta_\omega q\|\les\\
				&\qquad + \|\chi_{>R/2} \la r \ra^\lambda \rho_r^{-\kappa -2}q \|\les 	\\
			&\lesssim  \sum_{a \le 2} \|\chi_{>R/2} \la r \ra^{\lambda-\kappa-2+a} \nabla^a q\|\les, \label{pertles}
	\end{split}
	\end{equation}
and we note \eqref{pertles} holds for any $q$ with enough regularity. Now we have
	\begin{equation}
	\begin{split}
		\| \rr&^\l \vf \chi_{>R/2}P^2 q\|\les \\
			&\lesssim \sum_{a \le 2} \| \rr^{\l-\kappa-2+a}\nabla^a \vf q\|_{\mathcal{LE}^*([\frac{R}{2},\infty))} + \| \rr^\l[\vf,\chi_{>R/2}P^2] q\|\les. 
	\end{split} \label{pertvf}
	\end{equation}
For the commutator $[\vf,\chi_{>R/2}P^2]$ we find
	\[ [\vf,\chi_{>R/2}P^2] = [T^i,\chi_{>R/2}]\Omega^jS_r^kP^2 + T^i\Omega^j[S_r^k,\chi_{>R/2}]P^2 + \chi_{>R/2}[\vf,P^2]. \]
The commutators $[ S_r^k, \chi_{>R/2}]$ and $[T^i,\chi_{>R/2}]$ are compactly supported and uniformly bounded in $R$. We note if $Q^2$ is an operator of the form
	\[
	\begin{split}
	    &Q^2 = \partial_i h^{ij} \partial_j + h^\omega \Delta_\omega + h_r + h_\ell,\\ 
	    &h^{ij} \in \ls{-\kappa},\quad  h^\omega, h_r \in S_{rad}(r^{-\kappa-2}), \quad \hbox{ and } \quad h_\ell \in \ls{-\kappa-2} ,
	\end{split}
	\]
then $ [\Gamma, Q^2] = Q^2 $ where we allow the precise form of $Q^2$ to change each time it appears. Since $P^2$ is an operator of the above form, it follows that $[\Gamma^n, P^2] = Q^2 \Gamma^{<n}$. Now we can use \eqref{pertles} to obtain
	\begin{align*}
		\sup_{i+j+k \le n} \| \rr^\l[\vf,\chi_{>R/2}&P^2] q\|\les\\ 
		    &\lesssim \sup_{i+j+k \le n-1} \sum_{a \le 2} \| \rr^{\l-\kappa-2+a}\nabla^a\vf q\|_{\mathcal{LE}^*([\frac{R}{2},\infty))},
	\end{align*}
which combined with \eqref{pertvf} yields
	\begin{equation} 
	\begin{split}
		\|\chi_{>R/2}P^2 q\|\znl{n}{\l} &\lesssim \|q\|_{Z^{n+2,\lambda-\kappa-2}([\frac{R}{2},\infty))} \lesssim R^{-1}\|q\|\znl{n+2}{\l-2},
	\end{split} \label{qend}
	\end{equation}
using $|\nabla q| \le  r^{-1} (|S_rq| + |\Omega q|)$ (which holds for general $q$), as desired.

    In the following calculations we will use the fact that if $\phi \in \ell^1S(1)$, then
    \begin{equation} \label{lemmaA3}
        \|(\partial^\alpha \phi)\rho^q\|_{Z^{n,\l}} \lesssim \sum_{|B|=|\alpha|}^{n+|\alpha|} \|\la r \ra^{|B|}\partial^B\phi\|_{L^\infty(A_m)}
    \end{equation}
    for $q \le -\l-2+|\alpha|$, which is established by direct calculation.

	To prove \eqref{dpert} we use \eqref{qend} to find
	\begin{align*}
		\|\chi_{>R/2}P^2 d(r)\nabla^{\l-1}&\rr^{-1}\|\znl{n}{\l}\\ 
		    &\lesssim \|\chi_{>R/2}P^2 d(r) \rho^{-\l} \|\znl{n}{\l}\\
			&= \|d(r) \rho^{-\l} \|_{Z^{n+2,\lambda-\kappa-2}([\frac{R}{2},\infty))}\\
			&\lesssim \| S_r (d \rho^{-\l}) \|_{Z^{n+1,\lambda-\kappa-2}([\frac{R}{2},\infty))} + \| \Omega (d \rho^{-\l}) \|_{Z^{n+1,\lambda-\kappa-2}([\frac{R}{2},\infty))} \\
				&\qquad + \| T (d \rho^{-\l}) \|_{Z^{n+1,\lambda-\kappa-2}([\frac{R}{2},\infty))} + \|d \rho^{-\l}\|_{Z^{0,\l-\kappa-2}([\frac{R}{2},\infty))} \\
			&\lesssim \| (S_r d) \rho^{-\l}\|_{Z^{n+1,\lambda-\kappa-2}([\frac{R}{2},\infty))} + \| d \rho^{-\l}\|_{Z^{n+1,\lambda-\kappa-2}([\frac{R}{2},\infty))}\\
			&\lesssim \sum_{b=0}^{n+1} \| (S_r d) \rho^{-\l}\|_{Z^{b,\lambda-\kappa-2}([\frac{R}{2},\infty))} + \| d \rho^{-\l}\|_{Z^{0,\lambda-\kappa-2}([\frac{R}{2},\infty))}\\
			&\lesssim R^{-1} \|S_r d\|_{\ell^1S(1)} + R^{-1}\|d\|_{L^{\infty}}.
	\end{align*}
The last inequality follows from \eqref{lemmaA3}.

	To prove \eqref{epert} we use \eqref{qpert} and \eqref{lemmaA3} to find
	\begin{equation} 
	\begin{split}
		\|\chi_{>R/2}P^2 e_j (\nabla^{j}\rr^{-1})\rr^{j-\l+1}\|\znl{n}{\l} &\lesssim R^{-1}\|e_j (\nabla^{j}\rr^{-1})\rr^{j-\l+1}\|\znl{n+2}{\l-2}\\ &\lesssim R^{-1}\|e_j\|_{\ell^1S(1)} 
	\end{split} \label{ereuse}.
	\end{equation}

	Finally we consider \eqref{cpert}. Using \eqref{qend} we find
	\begin{align*}
		\|\chi_{>R/2}P^2 c_j \nabla^j \rr^{-1}\|\znl{n}{\l} &\lesssim \|c_j \nabla^j \rr^{-1}\|_{Z^{n+2,\l-\kappa-2}([R/2,\infty))}\\
			&\lesssim |c_j| \|\rr^{\l-\kappa-3-j}\|_{\mathcal{LE}^*([R/2,\infty))}\\
			&\lesssim |c_j| \sum_{m>\log\frac{R}{2}} 2^{m(\l-\kappa-1-j)}.
	\end{align*}
If $\lambda < \kappa+1+j$, then this yields
	\[ \| \chi_{>R/2}P^2 c_j \nabla^j \rr^{-1} \|\znl{n}{\l} \lesssim R^{-1}|c_j| \]
as desired. Thus the $c_0$ term fails to be perturbative when $\l = \kappa+1$.
	
	A straightforward calculation yields
	$ |P^2 (c_0 \la r \ra^{-1})| \lesssim c_0 \rho_r^{-\kappa-3} + c_0 \rho_\ell^{-\kappa-3} $. Note that $\rho_\ell^{- \kappa-3} \in Z^{\nu,\kappa+1}$ for all $\nu$. We obtain decay as $R \to \infty$ so that
	$ \| \chi_{>R/2} c_0 \rho_\ell^{-\kappa-3} \|_{Z^{n,\lambda}} \lesssim o_R(1) |c_0|. $
Thus only the radial term $c_0 \rho_r^{-\kappa-3}$, which arises when the radial scalar term in $P^2$ lands on $c_0 r^{-1}$, fails to be perturbative. 

	To handle this piece we consider
	\[ -\Delta w = h \in S_{rad}(r^{-\kappa-3}), \qquad \hbox{supp }h \subseteq \{ r \ge \frac{R}{4}\}. \]
By Lemma \ref{delrad} we have that $w$ can be written as
	$ w = c \rr^{-1} + e(r) \rr^{-\kappa-1}. $
Using the fact $\partial_r^2(rw) = rh$, we see 
	$  \|e\|_{S(1)} \lesssim \| h\|_{S_{rad}(r^{-\kappa-3})}. $
Furthermore, since $\hbox{supp }h \subseteq \{ r \ge \frac{R}{4}\}$ we find for $s < \frac{R}{4}$
	\begin{align*}
		|Rw(R)| &= \left| \int_0^R \int_s^\infty \partial_r^2(rw) \dd r \dd s \right|\\
		 	&= \left| \int_0^R \int_s^{\infty} rh \dd r \dd s \right| \\
		 	&\le \int_0^R \int_{R/4}^{\infty} \rr^{-\kappa-2} \dd r \dd s \|h\|_{S_{rad}(r^{-\kappa-3})} \\
		 	&\lesssim R^{-\kappa}\|h\|_{S_{rad}(r^{-\kappa-3})}.
	\end{align*}
Therefore we have
	\[ |c| = |Rw(R) - e(R)R^{-\kappa}| \lesssim R^{-\kappa} \|h\|_{S_{rad}(r^{-\kappa-3})}. \]
It follows that $w = c \rr^{-1} + e(r) \rr^{-\kappa-1}$ satisfies the estimate
	\[ |c| + R^{-\frac{1}{2}}\|e\|_{S(1)} \lesssim R^{-\frac{1}{2}} \|h\|_{S_{rad}(r^{-\kappa-3})}. \]
	
	Using Lemma \ref{freeres} and the above calculations, we have if 
	\[ -\Delta w = h_1 + h_2, \qquad h_1 \in Z^{n,\l}, \quad h_2 \in \sr{-q} \hbox{ for } q \ge 4 \] 
where $\hbox{supp }(h_2) \subseteq \{ r \ge \frac{R}{4} \}$, then
$w$ can be written as in \eqref{freeexp} with
	\[ 
	\begin{split}
		|c_0| + R^{-\frac{1}{2}}\|e_0\|_{S(1)} + \sum_{j=1}^{\lambda-2} |c_j| +\|e_j\|_{\ell^1S(1)} + \|d(r)\|_{L^{\infty}} + &\|S_rd\|_{\ell^1S(1)} + \| q \|_{Z^{n+2,\lambda-2}}\\
		 	&\lesssim \| h_1 \|_{Z^{n,\lambda}}+ \| h_2 \|_{R^{\frac{1}{2}}S_{rad}(r^{-\lambda-2})}. 
	\end{split}	
	\]
	Note that $\ell^1S(1) \subset S(1)$ so that the change of space for $e_0$ as compared with Lemma \ref{freeres} causes no problem.
	We now have \eqref{nearinfty} is perturbative with respect to this estimate since
	\[ R^{-\frac{1}{2}}\| \chi_{<R/2}c_0 \rho_r^{-\kappa-3}\|_{S_{rad}(r^{-\kappa-3})} \lesssim R^{-\frac{1}{2}} |c_0|. \]
Furthermore, the $e_0$ terms remain perturbative since our estimate \eqref{ereuse} came with extra powers of $R^{-1}$ and $\| \chi_{>\frac{R}{2}} e_0 \|_{\ell^1S(1)} \lesssim \log R \| e_0\|_{S(1)}$. This concludes the proof of the proposition.
\end{proof}
\end{section}

\begin{section}{Low Frequency Analysis} \label{lowfreqan}
In this section we use Proposition \ref{propexp} to calculate the error in the estimate $R_\tau g \approx (R_0g)\enitr.$ Direct calculation yields for any function $\phi$
	\begin{equation} \label{ptauphie}
	\begin{split}
		P_\tau(\phi e^{-i\tau\rr}) &= [(\Delta +P^2)\phi]e^{-i\tau\rr} - 2i\tau\big[\drfracr \phi\big]e^{-i\tau\rr} \\
			&\qquad + \Big[\big(\tau(\rho_\ell^{-\kappa-1} + \rho_\ell^{-\kappa}\nabla) + \tau^2 \rho_\ell^{-\kappa} \big)\phi\Big]e^{-i\tau\rr}. 
	\end{split} 
	\end{equation}
Since $P_0 = (\Delta + P^2)$, \eqref{ptauphie} gives
\begin{equation} \label{ptauerr1} 
	\begin{split} 
		P_\tau \left( R_\tau g - (R_0 g) e^{-i\tau \la r \ra} \right) &= g(1 - e^{-i\tau \la r \ra}) - 2i\tau \big[ (\partial_r + \frac{1}{r}) R_0 g \big] e^{-i\tau \la r \ra} \\
		    & \quad +  \Big[ \big( \tau(\rho_\ell^{-\kappa-1} + \rho_\ell^{-\kappa}\nabla) + \tau^2 \rho_\ell^{-\kappa} \big)R_0g \Big] e^{-i\tau\rr}.
	\end{split}	
\end{equation}

Analyzing the right hand side of \eqref{ptauerr1} with $R_0g$ represented by the expansion found in Proposition \ref{propexp} will yield a useful form of the error in the low frequency estimate. Lemmas \ref{lmadrfracr} - \ref{lmaenaught} provide preliminary calculations that help us handle terms arising in the $-2i\tau \big[ \drfracr R_0g \big] \enitr$ piece in \eqref{ptauerr1}. In Proposition \ref{lowferr} we will use these lemmas to establish the error in the low frequency estimate.

\begin{lemma} \label{lmadrfracr}
    Let $n$ be a positive integer. Let $\varphi_0$ be given and assume $\varphi_a$ satisfies
        \begin{equation} \label{lmaass}
            -\frac{1}{2} \Delta \varphi_{a} = \drfracr\varphi_{a-1}
        \end{equation}
    and $\varphi_a \in S(r^{-1})$ for $1 \le a \le n$.
    Then
     \begin{equation} \label{help}
	  \begin{split}
		-2i\tau \big[ \drfracr\varphi_0& \big] e^{-i\tau\rr}\\ 
		    &= \sum_{a=1}^{n} \Big( -P_\tau((-i\tau)^a \varphi_a e^{-i\tau\rr}) + (\tau^a \zln{\kappa}{\nu} + \tau^{a+1}\zln{\kappa}{\nu}+ \tau^{a+2}\zln{\kappa-1}{\nu})\enitr \Big)\\
			&\qquad  - 2(-i\tau)^{n+1} \drfracr(\varphi_n)e^{-i\tau\rr}  
	    \end{split}
	    \end{equation}
    \end{lemma}

\begin{proof}
    By \eqref{ptauphie} and \eqref{lmaass}, we have
	\begin{equation}
	\begin{split}	
		-2i\tau \big[ \drfracr \varphi_0 \big] e^{-i\tau\rr} &=i\tau(\Delta\varphi_1)\enitr \\
		    &= P_\tau(i\tau \varphi_1 e^{-i\tau\rr}) - 2\tau^2 \big[ \drfracr \varphi_1 \big]e^{-i\tau \rr}-i\tau (P^2\varphi_1)e^{-i\tau \rr}\\
			    &\qquad  + \tau \Big[ \big( \tau(\rho_\ell^{-\kappa-1} + \rho_\ell^{-\kappa}\nabla) + \tau^2 \rho_\ell^{-\kappa} \big)\varphi_1 \Big] e^{-i\tau\rr} \\
			&= i\tau P_\tau(\varphi_1 e^{-i\tau\rr}) - 2\tau^2 \big[ \drfracr \varphi_1 \big]e^{-i\tau \rr}\\
			&\qquad  + \Big( \tau \zln{\kappa}{\nu} + \tau^2 \zln{\kappa}{\nu} + \tau^3 \zln{\kappa-1}{\nu} \Big) \enitr.
	\end{split} \label{cpre}
	\end{equation}
The last inequality uses the fact that $\varphi_1 \in S(r^{-1})$ by assumption, and the term $\tau \zln{\kappa}{\nu}$ is limited in its decay by the $\rho_r^{-\kappa-2}$ term in $P^2$.

Repeating this argument, now for $2i\tau^2\big[ \drfracr i\varphi_1\big] e^{-i\tau\rr}$ with $\varphi_2$ satisfying $ -\frac{1}{2}\Delta\varphi_2 = \drfracr\varphi_1 $
and plugging the resulting expression into \eqref{cpre} we find
	\[ 
	\begin{split}
		-2i\tau \drfracr(\varphi_0) e^{-i\tau\rr} &= P_\tau\Big(i\tau \varphi_1 e^{-i\tau\rr} + \tau^2 \varphi_2 e^{-i\tau\rr}\Big) + 2i\tau^3 \big[ \drfracr \varphi_2 \big] e^{-i\tau\rr}\\
			&\qquad  + \Big(\tau\zln{\kappa}{\nu} + \tau^2 \zln{\kappa}{\nu} + \tau^3 (\zln{\kappa}{\nu} + \zln{\kappa-1}{\nu}) +\tau^4 \zln{\kappa}{\nu} \Big) \enitr
	\end{split}
	\]
as long as $\varphi_1, \varphi_2 \in S(r^{-1})$. 

Repeating this process yields \eqref{help} since $\varphi_a \in S(r^{-1})$ for $a \le n$ by assumption.
    \end{proof}

\begin{lemma} \label{keylapinv}
If $r \ge 2$, then for $j \in \N_+$
	\begin{equation}
		\drfracr(\nabla^j \rr^{-1}) = -\frac{1}{2} \Delta \Big( \sum_{k=0}^{j-1} \nabla^{j-1-k}\frac{x}{r} \nabla^k \rr^{-1} \Big). \label{drfrdel} 
	\end{equation}
\end{lemma}

\begin{proof}
Direct calculation yields $[\drfracr,\nabla]  = -\frac{1}{2} [\Delta, \frac{x}{r} ]$, and we use this identity to handle the $j=1$
case:
	\begin{align*}
		\drfracr \nabla \rr^{-1} = \nabla\drfracr r^{-1} + \frac{1}{2}\frac{x}{r} \Delta \rr^{-1} - \frac{1}{2} \Delta \Big(\frac{x}{r} r^{-1}\Big) = -\frac{1}{2} \Delta \Big(\frac{x}{r}\rr^{-1}\Big).
	\end{align*}

Now fix $J$ and assume \eqref{drfrdel} holds for $j = J-1$. We calculate
	\begin{align*}
		\drfracr \nabla^J \rr^{-1} &= \nabla \drfracr \nabla^{J-1} \rr^{-1} + \frac{1}{2} \frac{x}{r} \Delta (\nabla^{J-1}\rr^{-1}) - \frac{1}{2} \Delta \Big(\frac{x}{r} \nabla^{J-1}\rr^{-1} \Big) \\
			&= \nabla\Big(-\frac{1}{2} \Delta \sum_{k=0}^{J-2} \nabla^{J-2-k}\frac{x}{r}\nabla^k \rr^{-1} \Big) + 0 - \frac{1}{2}\Delta \Big(\frac{x}{r}  \nabla^{J-1} \rr^{-1}\Big) \\
			&= -\frac{1}{2} \Delta \sum_{k=0}^{J-1}\nabla^{J-1-k}\frac{x}{r} \nabla^k \rr^{-1}
	\end{align*}
for $r \ge 2$, as desired.
    \end{proof}

\begin{lemma} \label{drrvarphi}
	Let $a \in \N$ be fixed and let $\varphi^j$ be a family of functions indexed by $j$ which satisfy
	\begin{equation} \label{phijass}
		\drfracr \varphi^j = -\frac{1}{2} \Delta \sum_{k=a}^{j-1} \nabla^{j-1-k} \frac{x}{r} \varphi^k, \qquad j \ge a+1 
	\end{equation}
and $\drfracr \varphi^a = 0$. Then
	\begin{equation}
	\begin{split}
		\drfracr  \sum_{k=a}^{j-1} \nabla^{j-1-k} \frac{x}{r} \varphi^k
			&= -\frac{1}{2} \Delta \sum_{k=a+1}^{j-1}\nabla^{j-1-k} \frac{x}{r} \Big( \sum_{\ell=a}^{k-1} \nabla^{k-1-\ell}\frac{x}{r}  \varphi^\ell  \Big).
	\end{split} \label{dpre}
	\end{equation} 
    \end{lemma}

Note that if we define \[ \varphi_1^j := \sum_{k=a}^{j-1} \nabla^{j-1-k}\frac{x}{r} \varphi^k \]
then the lemma shows that $\varphi_1^j$ is a family of  functions indexed by $j$ which satisfy the assumptions of the lemma with $a$ replaced by $a+1$.

\begin{proof}
We calculate
	\begin{equation} \label{lemcal1}
	\begin{split}
	    \drfracr \sum_{k=a}^{j-1} \nabla^{j-1-k} \frac{x}{r}  \varphi^k
			&= \sum_{k=a+1}^{j-1} \nabla^{j-1-k} \frac{x}{r} \drfracr \varphi^k + \sum_{k=a}^{j-2} [\drfracr,\nabla^{j-1-k}] \frac{x}{r} \varphi^k\\
			&= -\frac{1}{2} \sum_{k=a+1}^{j-1}\sum_{\ell=a}^{k-1} \nabla^{j-1-k}\Big( \frac{x}{r} \Delta\Big) \nabla^{k-1-\ell} \frac{x}{r} \varphi^\ell\\
				&\qquad -\frac{1}{2} \sum_{k=a}^{j-2} \sum_{\ell=1}^{j-1-k} \nabla^{j-1-k-\ell} [\Delta, \frac{x}{r}] \nabla^{\ell-1}\frac{x}{r}\varphi^k.
	\end{split}
	\end{equation}
The first equality uses the assumption $\drfracr \varphi^a = 0$. The second equality uses \eqref{phijass} and the identity
	\begin{equation}
	\begin{split} 
		[\drfracr,\nabla^b] = \sum_{\ell=1}^{b} \nabla^{b-\ell}[\drfracr,\nabla]\nabla^{\ell-1} = -\frac{1}{2} \sum_{\ell=1}^b \nabla^{b-\ell}[\Delta, \frac{x}{r} ] \nabla^{\ell-1}.
	\end{split} \label{drfnab}
	\end{equation}

Changing the order of summation in the first term on the right hand side of \eqref{lemcal1} and simply switching the indexing labels in the second term yields
	\begin{align*}
		\drfracr \sum_{k=a}^{j-1} \nabla^{j-1-k} \frac{x}{r} \varphi^k
			&= -\frac{1}{2} \sum_{\ell=a}^{j-2}\sum_{k=\ell+1}^{j-1} \nabla^{j-1-k}\Big( \frac{x}{r} \Delta\Big) \nabla^{k-1-\ell} \frac{x}{r} \varphi^\ell\\
				&\qquad -\frac{1}{2} \sum_{\ell=a}^{j-2} \sum_{k=1}^{j-1-\ell} \nabla^{j-1-\ell-k} [\Delta, \frac{x}{r}] \nabla^{k-1}\frac{x}{r}\varphi^\ell\\
			&= -\frac{1}{2} \sum_{\ell=a}^{j-2}\sum_{k=\ell+1}^{j-1} \nabla^{j-1-k}\Big( \frac{x}{r} \Delta + [\Delta, \frac{x}{r}] \Big) \nabla^{k-1-\ell} \frac{x}{r} \varphi^\ell\\
			&= -\frac{1}{2} \Delta \sum_{k=a+1}^{j-1} \nabla^{j-1-k} \frac{x}{r} \sum_{\ell=a}^{k-1} \nabla^{k-1-\ell} \frac{x}{r} \varphi^\ell.
	\end{align*}
To obtain the second equality we redefine the $k$ index by $k \mapsto k+\ell$. To obtain the third equality we switch the order of summation. This completes the proof of the lemma. 
    \end{proof}

\begin{lemma} \label{lmagradrinv}
    Let $j \ge 0$ and $r \ge 2$. Then
   \begin{equation} \label{drphj}
	\begin{split}
		 -2i\tau &\drfracr (\nabla^j \rr^{-1} )e^{-i\tau\rr}\\
		 	 &= \sum_{a=1}^{j} \Big(- P_\tau((-i\tau)^a F_a^j e^{-i\tau\rr}) + (\tau^a \zln{\kappa}{\nu} + \tau^{a+1}\zln{\kappa}{\nu}+ \tau^{a+2}\zln{\kappa-1}{\nu})\enitr \Big) 
	\end{split}
	\end{equation}
	where $|F_a| \lesssim \rr^{-1}$ for $1 \le a \le j$. 
    \end{lemma}

\begin{proof}
    We define $\varphi_0^j := \nabla^j \rr^{-1}$. By Lemma \ref{keylapinv} $\varphi_0^j$ satisfies \eqref{phijass} with $a=0$. Furthermore $(\partial_r + \frac{1}{r})\varphi_0^0 = 0$, so the assumptions of Lemma \ref{drrvarphi} are satisfied. 
    
    Then defining $\varphi_1^j := \sum_{k=0}^{j-1}\nabla^{j-1-k}\frac{x}{r}\varphi_0^k$, we have
	\[ \drfracr\varphi_1^j = -\frac{1}{2} \Delta \sum_{k=1}^{j-1} \nabla^{j-1-k} \frac{x}{r} \varphi_1^k =: -\frac{1}{2} \Delta \varphi_2^j \]
by Lemma \ref{drrvarphi}. Now $\varphi_1^j$ satisfies \eqref{phijass} with $a=1$ and 
	\[ (\partial_r + \frac{1}{r})\varphi_1^1 = \drfracr\frac{x}{r} \rr^{-1} = 0 \] 
so we can iterate the process again. We define
		\begin{equation} 
			\varphi_\ell^j := \sum_{k=\ell-1}^{j-1} \nabla^{j-1-k} \frac{x}{r} \varphi_{\ell-1}^k \label{nextlemmaref}
		\end{equation}
and see $\drfracr \varphi_j^j = \drfracr \big( \frac{x}{r} \big)^j \rr^{-1} = 0$ so the assumptions of Lemma \ref{drrvarphi} are satisfied at each iteration and we find for $0 \le n \le j-1$
	\begin{align*}
		\drfracr \varphi_n^j &= -\frac{1}{2} \Delta \sum_{k=n}^{j-1} \nabla^{j-1-k} \frac{x}{r} \varphi_n^k =: -\frac{1}{2} \Delta \varphi_{n+1}^j.
	\end{align*}
We note  $\varphi_n^j \in S(r^{-j-1+n})$ so $\varphi_n^j \in S(r^{-1})$ for $n \le j$.  Now \eqref{drphj} follows by \eqref{help} since $\drfracr \varphi_j^j = 0$. 
    \end{proof}

We define $a \wedge b$ to be a smooth function such that $a \wedge b = \min(a,b)$ when $a \gg b$ or $b \ll a$.

\begin{lemma} \label{lmaenaught}
   Let $\varphi_0 \in S_{rad}(r^{-\kappa-1})$. Then
        \begin{equation} \label{3termb}
        \begin{split}
            2\tau&\big[\drfracr\varphi_0\big]\enitr\\
			&= \sum_{a=1}^{\kappa-1} \Big( P_\tau(\tau^a F_a \enitr) + (\tau^a \zln{\kappa}{\nu} + \tau^{a+1}\zln{\kappa}{\nu} + \tau^{a+1}\zln{\kappa-1}{\nu})\enitr \Big) + (\tau^\kappa \zln{1}{\nu} + \tau^{\kappa+1}\zln{0}{\nu})\enitr \\
				&\qquad - \tau^\kappa P_\tau\Big( (\rr^{-1}\epsilon_1(r\wedge|\tau|^{-1}) + \tau(\epsilon_2(r\wedge|\tau|^{-1} -\epsilon_2(|\tau|^{-1}))\enitr\Big).
        \end{split}
        \end{equation}  
    where $|F_a| \lesssim \rr^{-1}$ and $\epsilon_1,\epsilon_2 \in S(\log r)$.
    \end{lemma}

\begin{proof}
    Since $\varphi_0 \in S_{rad}(r^{-\kappa-1})$, we see $\drfracr \varphi_0 \in \sr{-\kappa-2}$. Lemma \ref{delrad} implies
 \begin{align*}
 	\drfracr \varphi_{n} &= -\frac{1}{2}\Delta \varphi_{n+1}, \qquad  \varphi_{n} \in \sr{-\kappa-1+n}
 \end{align*}
for $0 \le n \le \kappa-2$.
Then by Lemma \ref{lmadrfracr} we have
	\begin{equation}
	\begin{split}
		-2i\tau&\big[(\partial_r +\frac{1}{r}) \varphi_0\big]\enitr\\
			&= \sum_{a=1}^{\kappa-1} \Big(- P_\tau((-i\tau)^a F_a \enitr) + (\tau^a \zln{\kappa}{\nu} + \tau^{a+1}\zln{\kappa}{\nu} + \tau^{a+1}\zln{\kappa-1}{\nu})\enitr \Big)\\
				&\qquad - 2(-i\tau)^\kappa \drfracr(\varphi_{\kappa-1}) \enitr. \label{estep1}
	\end{split}
	\end{equation}
Since $\varphi_{\kappa-1} \in \sr{-2}$, we see $ \drfracr \varphi_{\kappa-1} \in \sr{-3}$. By Lemma \ref{delrad} there exists an $\epsilon_1 \in S_{rad}(\log r)$ such that
	\[ -\frac{1}{2}\Delta \rr^{-1}\epsilon_1(r) = \drfracr \varphi_{\kappa-1}. \]
We wish to avoid the logarithmic growth in $r$, so we use the modified function
	\[ \tilde{\varphi}_\kappa := \rr^{-1} \epsilon_1(r\wedge |\tau|^{-1}). \]
Now we have
	\begin{align*}
		-\frac{1}{2}\Delta \tilde{\varphi}_\kappa = \chi_{<|\tau|^{-1}} \drfracr \varphi_{\kappa-1}
			= \drfracr \varphi_{\kappa-1} - \chi_{>|\tau|^{-1}} \drfracr \varphi_{\kappa-1}.
	\end{align*}
Using \eqref{ptauphie} we find
	\begin{equation}
	\begin{split}
		2\tau^\kappa \drfracr &\varphi_{\kappa-1} \enitr\\
			&= P_\tau \Big(\tau^\kappa(-\tilde{\varphi}_{\kappa})\enitr\Big) + 2 i \tau^{\kappa+1} \drfracr(\tilde{\varphi}_{\kappa})\enitr \\
				&\qquad + \tau^{\kappa}\Big(\tau(\rho_\ell^{-\kappa-1}+ \rho_\ell^{-\kappa}\nabla) + \tau^2\rho_\ell^{-\kappa}\Big)(\tilde{\varphi}_{\kappa}) \enitr + \tau^\kappa P^2(\tilde{\varphi}_{\kappa})\enitr \\
				&\qquad +2\tau^\kappa \chi_{>|\tau|^{-1}}\big[ \drfracr \varphi_{\kappa-1}\big] \enitr  \\
			&= P_\tau \Big(\tau^\kappa(-\tilde{\varphi}_{\kappa})\enitr\Big) + 2 \tau^{\kappa+1} \chi_{<|\tau|^{-1}} \rr^{-1}(\partial_r\epsilon_1(r))\enitr \\
				&\qquad + \Big(\tau^\kappa \zln{\kappa}{\nu}+ \tau^{\kappa+1}(\zln{\kappa-1}{\nu} + \zln{0}{\nu})  + \tau^{\kappa+2}\zln{\kappa-2}{\nu} \Big)\enitr.
	\end{split} \label{omg}
	\end{equation}
Here we used the fact $|\tau|^{-1}\chi_{>|\tau|^{-1}}\drfracr \varphi_{\kappa-1} \in \ls{-2}$ because $\drfracr \varphi_{\kappa-1} \in \sr{-3}$, and the cutoff function allows us to pull out a $\tau$ factor when summing in the $\ls{-2}$ norm.

	We still need to handle the term $2\tau^{\kappa+1}\chi_{<|\tau|^{-1}} \rr^{-1}(\partial_r\epsilon_1(r))\enitr$. By Lemma \ref{delrad}, there exists an $\epsilon_2 \in S_{rad}(\log r)$ such that
	\[ -\frac{1}{2}\Delta \epsilon_2 = \rr^{-1}(\partial_r\epsilon_1(r)) \in S_{rad}(r^{-2}). \]
To remove the logarithmic growth in $r$ we use the modified function
	\[ \tilde{\varphi}_{\kappa+1} := \epsilon_2(r\wedge |\tau|^{-1}) - \epsilon_2(|\tau|^{-1}) \]
and find
	\[ -\frac{1}{2}\Delta \tvk = \chi_{<|\tau|^{-1}} \rr^{-1}(\partial_r\epsilon_1(r)). \]
Then by \eqref{ptauphie} we have
	\begin{equation} \label{omg2}
	\begin{split}
		2\tau^{\kappa+1}&\chi_{<|\tau|^{-1}} \rr^{-1}(\partial_r\epsilon_1(r))\enitr\\
			&= P_\tau\Big(\tau^{\kappa+1}(-\tvk)\enitr\Big) + 2\tau^{\kappa+2} \drfracr(\tvk)\enitr \\
				&\qquad + \tau^{\kappa+1}\Big(\tau(\rho_\ell^{-\kappa-1}+ \rho_\ell^{-\kappa}\nabla) + \tau^2\rho_\ell^{-\kappa}\Big)(\tvk) \enitr + (P^2\tvk)\enitr\\
			&= P_\tau\Big(\tau^{\kappa+1}(-\tvk)\enitr\Big) + \tau^{\kappa+1} \zln{0}{\nu}\enitr  + \Big(\tau^{\kappa}\zln{\kappa-1}{\nu} + \tau^{\kappa+1} \zln{\kappa-1}{\nu}\Big)\enitr.
	\end{split} 
	\end{equation}
Combining \eqref{estep1}, \eqref{omg}, and \eqref{omg2} then yields \eqref{3termb}, as desired.
    \end{proof}

We are now ready to calculate the error in the approximation $R_\tau g \approx (R_0g)\enitr$, which is established in Proposition \ref{lowferr}. The precise form of the error depends on the regularity and decay assumed for $g$. 

\begin{proposition} \label{lowferr}
	Let $g_\lambda^\nu \in Z^{\nu,\lambda}$ with $1 \le \lambda \le \kappa +1$ and $\nu > 3\l$. Assume $|\tau| \lesssim 1$ and $\Im \tau \le 0$.
	\begin{enumerate}
	\item If $1 \le \lambda \le \kappa$ then
		\[ R_\tau g_\lambda^\nu = (R_0 g_\lambda^\nu) e^{-i\tau \rr} + R_\tau (\chi_{>|\tau|^{-1}}g_\lambda^\nu) + \sum_{m=1}^{\lambda-1} \tau^m(F_m+R_0\zln{\l-m}{\nu-3\l})e^{-i\tau \rr} + \tau^\lambda (R_\tau h_{\nu-3\l})  \]
		where $|F_m| \lesssim \la r \ra^{-1}$, $\zln{\l-m}{\nu-3\l} \in Z^{\nu-3\l,\l-m}$,  and $h_{\nu-3\l}$ satisfies
		\begin{equation} \label{hbnd}
			\sup_{i+j+k+q\le \nu-3\l} \| \la r \ra^q (\partial_r + i\tau)^q T^i\Omega^jS^k h_{\nu-3\l}\|\les \lesssim 1.  
		\end{equation}
	\item If $\lambda = \kappa+1$ then
	\[ \begin{split}
		R_\tau g_{\kappa+1}^{\nu} &= (R_0 g_{\kappa+1}^\nu) e^{-i\tau \la r \ra}  + R_\tau (\chi_{>|\tau|^{-1}} g_{\kappa+1}^\nu) + \sum_{m=1}^{\kappa} \tau^m (F_m + R_0\zln{\kappa+1-m}{\nu-3m}) e^{-i\tau r}   \\
		&\qquad + \tau^{\kappa} \epsilon(r,\tau) e^{-i\tau r} + \tau^{\kappa+1} (R_\tau h_{\nu-3\kappa-3})
		\end{split}
	\] 
	where $|F_m| \lesssim \la r \ra^{-1}$, $\zln{\kappa+1-m}{\nu-3m} \in Z^{\nu-3m,\kappa+1-m}$,  $\epsilon(r,\tau)$ is of the form
	\[ \epsilon(r,\tau) = \rr^{-1} \epsilon_1(r \wedge |\tau|^{-1}) + \tau\Big(\epsilon_2(r \wedge |\tau|^{-1}) - \epsilon_2(|\tau|^{-1}) \Big) \]
	with $\epsilon_1, \epsilon_2 \in S(\log r)$, and $h_{\nu-3\lambda}$ satisfies \eqref{hbnd} with $\l = \kappa+1$.
	\end{enumerate}
\end{proposition}

The term $\epsilon(r,\tau)$ in the statement of Proposition \ref{lowferr} for $\l = \kappa + 1$ arises due to the $e_0(r)$ term in Proposition \ref{propexp} in the $\l = \kappa + 1$ case. The final decay rate is ultimately determined by this $\epsilon(r,\tau)$ term.

We use $\zeta_\l^\nu$ to represent a function in $Z^{\nu,\l}$ and allow $\zln{\l}{\nu}$ to change from line to line. The purpose of this notation is to keep track of what function spaces each term in our calculations is in while reserving $\gln{\l}{\nu}$ to indicate the arbitrary but fixed function in $Z^{\nu,\l}$ given in the statement of the proposition. In general $Z^{N_1,L_1} \subset Z^{N_2,L_2}$ for $N_2 \le N_1$ and $L_2 \le L_1$. The monotonicity of $Z^{\nu,\l}$ allows us to collect terms with different regularity and decay and write them as one term: 
	\begin{equation} 
		\zln{\l_1}{\nu_1} + \zln{\l_2}{\nu_2} = \zln{\min(\l_1,\l_2)}{\min(\nu_1,\nu_2)}. \label{zmon}
	\end{equation}

\begin{proof}
    Define $\el$ to be the error associated to $R_\tau \gl$: $\el := R_\tau \gl - (R_0 \gl) e^{-i\tau \rr}.$ We begin by establishing a useful expression for $P_\tau(\el)$ using \eqref{ptauerr1}. Our expansion in Proposition \ref{propexp} shows $R_0\gl \in S(r^{-1}) + Z^{\nu-2,\l-2}$. It follows that
	\begin{equation} \label{lotsimp} 
	    \rho_\ell^{-\kappa-1}R_0\gl \in Z^{\nu-2,\kappa}, \quad \rho_\ell^{-\kappa}\nabla R_0\gl \in Z^{\nu-3,\kappa}, \quad \hbox{ and } \quad \rho_\ell^{-\kappa}R_0\gln{\l}{\nu} \in Z^{\nu-2,\kappa-1}.
	 \end{equation}

Furthermore we have
	\begin{equation}
		\gl-\gl\enitr = \chi_{>|\tau|^{-1}}(r)\gl + \tau \chi_{<|\tau|^{-1}}(r)\rr\gl\frac{1-\enitr}{\tau \rr} - \tau^\l \chi_{>|\tau|^{-1}}(r)\tau^{-\l}\gl\enitr. \label{gsimp}
	\end{equation}
We use \eqref{ptauerr1}, \eqref{lotsimp}, and \eqref{gsimp} to find        \begin{equation} \label{ptauerr} 
		\begin{split} 
				P_\tau(\el) &= \chi_{>|\tau|^{-1}}(r)\gl 
				    + \overbrace{\tau \Big(\chi_{<|\tau|^{-1}}(r)\rr\gl\frac{1-\enitr}{\tau r}\Big)}^{(I)} 
				    - \overbrace{\tau^\l \Big(\chi_{>|\tau|^{-1}}(r)\tau^{-\l}\gl\enitr\Big)}^{(II)}\\
					&\qquad - \underbrace{2i\tau\drfracr(R_0\gl)\enitr}_{(III)} 
					+ \underbrace{(\tau \zeta_\kappa^{\nu-3} + \tau^2 \zeta_{\kappa-1}^{\nu-2})\enitr}_{(IV)}.
			\end{split}	
		\end{equation}

    Next we claim that 
    \begin{equation} \label{hreq}
        \h_{\nu-3\l}=\zln{0}{\nu-3\l}\enitr
    \end{equation} 
    satisfies \eqref{hbnd}. Indeed, when $r \ge 2$, we find for any $\phi(x)$
	\begin{align*}
			S^k(\phi \enitr) &= (S_r^k \phi)\enitr, \qquad \Omega^j(\phi \enitr) = (\Omega^j \phi) \enitr\\
			|T^i (\phi \enitr)| &\lesssim \sum_{a=0}^i |(T^a \phi)\enitr|, \qquad (\partial_r +i\tau)^q (\phi \enitr) = (\partial_r^q \phi)\enitr.
		\end{align*}
Thus when $i + j + k + l \le \nu-3\l$ we have
	\[ \|r^l(\partial_r +i\tau)^l \vf (\zln{0}{\nu-3\l} \enitr)\|\les \lesssim \|T^i\Omega^jS_r^{k+l} \zln{0}{\nu-3\l}\|\les \lesssim 1. \]

Our formula for $\el$ will be recursive in $\l$, so we begin by calculating $\eln{1}{\nu}$ directly.
	
\noindent \textbf{Case 1: $\lambda =1$} 

	We calculate $\eln{1}{\nu}$ using \eqref{ptauerr} and the expansion for $R_0\gln{1}{\nu}$ given by Proposition \ref{propexp}:
		\begin{equation} 
			R_0g_1^\nu = d(r)\rr^{-1} + q \label{exp1}
		\end{equation}
where $d \in L^{\infty}$, $S_r d \in \ell^1S(1)$, and $q \in Z^{\nu-2,\l-2}$. Terms $II$ and $IV$ in \eqref{ptauerr} are readily seen to be of the form $\tau h_{\nu-3}$ using \eqref{hreq}. Term $I$ in \eqref{ptauerr} can also be included in $\tau h_{\nu-3}$. The cutoff function restricts $I$ to the region where $r|\tau|\lesssim 1$, so \eqref{hbnd} reduces to $h_{\nu-3} \in Z^{\nu-3,0}$. Since $r \gln{1}{\nu} \in Z^{\nu,0}$,	
	\[ \partial_r\Big( \frac{1-\enitr}{\tau r^n} \Big) = \frac{i\enitr}{r^n} -n\frac{1-\enitr}{\tau r^{n+1}}, \qquad \partial_r\Big(\frac{i\enitr}{r^n}\Big) = -n\frac{i\enitr}{r^{n+1}} + \frac{\tau\enitr}{r^n}, \]
and
	\[ S \Big( \frac{1-\enitr}{\tau r} \Big) = 0, \]
we have 
	\[ \rr \gln{1}{\nu} \frac{1-\enitr}{\tau \rr} \in Z^{\nu,0} \subset Z^{\nu-3,0}, \] 
as desired. Here and throughout we harmlessly assume $r \ge 2$.  We note the above calculations and \eqref{gsimp} imply
	\begin{equation}
		R_\tau(\gln{1}{\nu}(1-\enitr)) = R_\tau(\chi_{>|\tau|^{-1}} \gln{1}{\nu}) + \tau (R_\tau h_{\nu}). \label{itl1}
	\end{equation}
This equation will help us handle terms that will arise when $\l \ge 2$ by providing a base case for an inductive argument. 

	For term $III$  in \eqref{ptauerr}, we use \eqref{exp1} and write $\partial_r = r^{-1} S_r$ to find $ \drfracr R_0\gln{1}{\nu} \in Z^{\nu-3,0}$, so this term also satisfies \eqref{hbnd}.

    Combining our calculations for terms $I, II, III,$ and $IV$ and applying $R_\tau$ to both sides of \eqref{ptauerr} yields
	\begin{equation} 
		E_1^\nu = R_\tau(\chi_{>|\tau|^{-1}}g_1^\nu) + \tau R_\tau h_{\nu-3} \label{l1res}
	\end{equation}
as desired.

\noindent \textbf{Case 2: $2 \le \lambda \le \kappa+1$}
	We proceed as in the $\l =1$ case. For term $I$ in \eqref{ptauerr}, we see  $\rr\gln{\l}{\nu} \frac{1-\enitr}{\tau \rr} \in Z^{\nu,\l-1} $ uniformly in $\tau$ as $\tau \to 0$. So we can write
	\begin{equation}
		I = \tau \chi_{<|\tau|^{-1}}(r) \rr \gln{\l}{\nu} \frac{1-\enitr}{\tau \rr} = \tau \chi_{<|\tau|^{-1}}(r) \zeta_{\l-1}^{\nu}. \label{bterm}
	\end{equation}
	
	For term $II$ in $\eqref{ptauerr}$, we note $\chi_{>|\tau|^{-1}}(r) \tau^{-\l}\gln{\l}{\nu} \in Z^{\nu,0}$ uniformly in $\tau$ as $\tau \to 0$. So we can write
	\begin{equation}
		 II = \tau^\l \chi_{>|\tau|^{-1}} \tau^{-\l}\gln{\l}{\nu} \enitr = \tau^\l \zeta_{0}^{\nu} \enitr. \label{cterm}
	\end{equation}
Substituting \eqref{bterm} and \eqref{cterm} into \eqref{ptauerr} we have
	\begin{equation}
	\begin{split}
		P_\tau(E_\l^\nu) &= \chi_{>|\tau|^{-1}}\gln{\l}{\nu} + \tau \chi_{<|\tau|^{-1}} \zeta_{\l-1}^{\nu}  + \Big(\tau \zeta^{\nu-3}_\kappa + \tau^2 \zeta^{\nu-2}_{\kappa-1} + \tau^\l \zeta_{0}^{\nu}\Big) \enitr \\
		&\qquad - 2i\tau\big[\drfracr R_0\gln{\l}{\nu}\big]\enitr. \label{step1}
	\end{split}
	\end{equation}

	We claim that terms of the right hand side of \eqref{step1} of the form
		\begin{equation} 
			\tau^m \zln{\l-m}{\nu}\enitr, \qquad 1 \le m \le \l-1 \label{itterm1}
		\end{equation}
produce error terms which can be handled inductively. First consider the case $m = \l-1$. By \eqref{itl1} we see that after applying $R_\tau$ to both sides of \eqref{step1}, terms on the right hand side that are of the form \eqref{itterm1} with $m=\l-1$ become
	\begin{equation} 
		\tau^{\l-1}R_\tau(\zln{1}{\nu}\enitr) = \tau^{\l-1} R_\tau(\chi_{<|\tau|^{-1}}\zln{1}{\nu})+\tau^\l(R_\tau h_\nu), \label{iemain1}
	\end{equation}
and we can appeal to case 1 of the proposition to handle the first term, while the second term is expected to appear in $\el$ (In fact this term has more regularity than the $h$ term in the statement of the proposition. We will see the last term on the right hand side of \eqref{step1} limits the amount of regularity we can get for $h_{\nu-3\l}$.)

	Now consider \eqref{itterm1} for $1 \le m \le \l -2$. Substituting \eqref{bterm} and \eqref{cterm} into \eqref{gsimp} gives
	\[
		R_\tau(\gl \enitr) = R_\tau(\chi_{<|\tau|^{-1}}\gl) + \tau R_\tau(\chi_{<|\tau|^{-1}}\zln{\l-1}{\nu}) + \tau^\l(R_\tau h_\nu). 
\]
	 Therefore after applying $R_\tau$ to both sides of \eqref{step1}, we see terms on the right hand side that are of the form \eqref{itterm1} with $1 \le m \le \l-2$ become
	 \begin{equation}
	 	\tau^mR_{\tau}(\zln{\l-m}{\nu}\enitr) = \tau^mR_\tau(\chi_{<|\tau|^{-1}}\zln{\l-m}{\nu}) + \tau^{m+1}R_\tau(\chi_{<|\tau|^{-1}}\zln{\l-m-1}{\nu}) + \tau^\l(R_\tau h_\nu), \label{iemain2}
	 \end{equation}
and we can proceed inductively for the first 2 terms, while the last term is expected to appear in $\el$. 

Next we consider the term $-2i\tau \big[ \drfracr R_0\gln{\l}{\nu}\big]\enitr$ in \eqref{step1} using Proposition \ref{propexp}. We note $\drfracr\rr^{-1} = 0$ so the term $2i\tau\drfracr c_0 \rr^{-1}$ vanishes. It is left to consider the terms
	\begin{enumerate}
		\item[(A)] $-2i\tau\big[ \drfracr c_j\nabla^j \rr^{-1} \big] \enitr$ for $1 \le j \le \l-2$
		\item[(B)] $-2i\tau\big[\drfracr e_j(r) (\nabla^j \rr^{-1} ) \rr^{j-\l+1}\big] \enitr$ for $1 \le j \le \l-2$
		\item[(C)] $-2i\tau \big[ \drfracr e_0 \rr^{-\l} \big] \enitr$ (Here the cases $\l \le \kappa$ and $\l = \kappa+1$ must be considered separately since $e_0 \in \ell^1S(1)$ for $\l \le \kappa$ and $e_0 \in S(1)$ for $\l = \kappa+1$ by Proposition \ref{propexp}.)
		\item[(D)] $-2i\tau \big[ \drfracr d(r)\nabla^{\l-1} \rr^{-1} \big] \enitr$
		\item[(E)] $-2i\tau \big[ \drfracr q(x) \big] \enitr$.
	\end{enumerate}

	For term E we have $q \in Z^{\nu-2,\l-2}$, so writing $\partial_r = r^{-1}S_r$ we see $\drfracr q \in Z^{\nu-3,\l-1}$. Thus we can write
	$
		E = \tau \zln{\l-1}{\nu-3} \enitr.
	$

	For term B we have $e_j \in \ell^1S(1)$ so $\drfracr e_j (\nabla^j \rr^{-1}) \rr^{j-\l+1} \in \ls{-\l-1}$. Since $\ls{-2} \subseteq Z^{N,0}$ for any $N$, we can write
	$
		B =  \tau \zln{\l-1}{\nu-3} \enitr.
	$

    To handle term A we multiply both sides of \eqref{drphj} by the constant $c_j$ to find 
	\[
		A = \sum_{a=1}^j \Big( P_\tau(\tau^a F_a e^{-i\tau\rr}) + (\tau^a \zln{\kappa}{\nu} + \tau^{a+1}\zln{\kappa}{\nu}+ \tau^{a+2}\zln{\kappa-1}{\nu})\enitr \Big)
	\]
where $|F_a| \lesssim \rr^{-1}$. Here we absorbed the constant into the functions $\zln{\kappa}{\nu}$ and $\zln{\kappa-1}{\nu}$.

	Next we consider term D. Using Lemma \ref{lmagradrinv} and the fact that $S_rd \in \ell^1S(1)$, we calculate
	\begin{align*}
		D &= -2i\tau(\partial_r d)(\nabla^{\l-1} \rr^{-1})\enitr - d2i\tau\drfracr(\nabla^{\l-1} \rr^{-1}) \\
			&=\tau \zln{\l-1}{\nu}\enitr + \sum_{a=1}^{\l-1} d P_\tau ((-i\tau)^a F_a e^{-i\tau\rr}) + (\tau^a \zln{\kappa}{\nu} + \tau^{a+1}\zln{\kappa}{\nu}+ \tau^{a+2}\zln{\kappa-1}{\nu})\enitr \\
			&=\tau \zln{\l-1}{\nu}\enitr + \sum_{a=1}^{\l-1} P_\tau( d (-i\tau)^a F_a e^{-i\tau\rr}) - [P_\tau, d] (-i\tau)^a F_a e^{-i\tau\rr}\\
				&\qquad + \Big(\tau^a \zln{\kappa}{\nu} + \tau^{a+1}\zln{\kappa}{\nu}+ \tau^{a+2}\zln{\kappa-1}{\nu}\Big)\enitr.
	\end{align*}
We find by direct calculation
	$
		[P_\tau, d]F_a\enitr = \Big(\rho_\ell^{-\l+a-2} + \tau \rho_\ell^{-\l+a-1}\Big)\enitr. 
	$
so that
    \[
        D = \sum_{a=1}^{\l-1} P_\tau \Big( \tau^a F_a e^{-i\tau\rr}\Big) + (\tau^a (\zln{\kappa}{\nu}+\zln{\l-a}{\nu}) + \tau^{a+1}(\zln{\kappa}{\nu} + \zln{\l-a-1}{\nu})+ \tau^{a+2}\zln{\kappa-1}{\nu})\enitr
    \]
since $\ls{-2} \subset Z^{N,0}$ for any $N$.

	Substituting our expressions for terms $A, B, D,$ and $E$ into \eqref{step1} and simplifying yields
	\begin{equation} \label{step2}
	\begin{split}
		P_\tau(\el) &= \chi_{>|\tau|^{-1}}\gln{\l}{\nu} + \sum_{a=1}^{\l-1}\Big( P_\tau(\tau^a F_a\enitr)\Big) + \tau \chi_{<|\tau|^{-1}} \zeta_{\l-1}^{\nu}  + \Big( \tau \zln{\l-1}{\nu-3} + \tau^2\zln{\l-2}{\nu-2} \Big)\enitr \\
			&\qquad + \sum_{m=3}^{\l} \tau^m \zln{\l-m}{\nu} \enitr + \underbrace{2\tau\drfracr(e_0\rr^{-\l})\enitr}_{(C)}
	\end{split} 
	\end{equation}
for $1 \le \l \le \kappa+1$. 

For term $C$ we consider the cases $2 \le \l \le \kappa$ and $\l = \kappa+1$ separately.

\noindent\textbf{Case 2(a): $2\le \l \le \kappa$}

Here we have $e_0 \in \ell^1S(1)$ so $e_0\rr^{-\l} \in \ls{-\l}$ and we find
	\begin{equation}
		C = 2\tau(\partial_r +\frac{1}{r}) (e_0(r) \rr^{-\lambda})\enitr = \tau \zln{\l-1}{\nu} \enitr. \label{3terma}
	\end{equation}
We can absorb $C$ into the term $\tau \zln{\l-1}{\nu-3}\enitr$ in \eqref{step2}. Then applying $R_\tau$ to both sides of \eqref{step2} and using \eqref{iemain1} and \eqref{iemain2} yields
	\begin{equation}
	\begin{split}
		\el &= R_\tau(\chi_{>|\tau|^{-1}}\gl) + \sum_{a=1}^{\l-1} (\tau^a F_a \enitr) + \tau R_\tau(\chi_{<|\tau|^{-1}}\zln{\l-1}{\nu-3}) + \tau^2R_\tau(\chi_{<|\tau|^{-1}}\zln{\l-2}{\nu-3})\\
			&\qquad + \sum_{m=3}^{\l-1} \tau^mR_\tau(\chi_{<|\tau|^{-1}}\zln{\l-m}{\nu}) + \tau^\l(R_\tau h_{\nu-3}).
	\end{split}
	\end{equation}
Part 1 of the proposition then follows by induction in $\l$ and the established base case for $\l =1$. We note the term $\tau R_\tau(\chi_{<|\tau|^{-1}}\zln{\l-1}{\nu-3})$ leads to the loss of regularity for $h$.

    \noindent\textbf{Case 2(b): $\l = \kappa+1$}

    Since $e_0\rr^{-\kappa-1} \in S_{rad}(r^{-\kappa-1})$, we can use Lemma \ref{lmaenaught}. Combining \eqref{3termb} and \eqref{step2} then applying $R_\tau$ yields
	\begin{equation}
	\begin{split}
		\eln{\kappa+1}{\nu} &= R_\tau(\chi_{>|\tau|^{-1}}\gln{\kappa+1}{\nu}) + \sum_{a=1}^{\kappa} (\tau^a F_a \enitr) + \tau R_\tau(\chi_{<|\tau|^{-1}}\zln{\kappa}{\nu}+\zln{\kappa}{\nu-3}\enitr) \\
			&\quad + \tau^2 R_\tau(\zln{\kappa-1}{\nu-2}\enitr) + \sum_{m=3}^{\kappa} \tau^m R_\tau (\zln{\kappa+1-m}{\nu}\enitr) + \tau^\kappa \epsilon(r,\tau)\enitr + \tau^{\kappa+1}R_\tau h_\nu.
	\end{split}
	\end{equation}
Part 1 of the proposition, \eqref{iemain1}, and \eqref{iemain2} then give
	\begin{equation}
	\begin{split}
		\eln{\kappa+1}{\nu} &= R_\tau(\chi_{>|\tau|^{-1}}\gln{\kappa+1}{\nu}) + \sum_{m=1}^{\kappa}\Big(\tau^m (F_m + R_0\zln{\l-m}{\nu-3m})\enitr \Big) + \tau^\kappa\epsilon(r,\tau)\enitr\\
		&\qquad + \tau^{\kappa+1}(R_\tau h_{\nu-3\kappa-3})
	\end{split}
	\end{equation}
as desired.  This concludes the proof of the proposition.
    \end{proof}

\end{section}

\begin{section}{Pointwise Resolvent Bounds}
In this section we establish the pointwise resolvent bounds that will be used in the proof of the main theorem (see Propositions \ref{proplargetau} and \ref{propsmalltau}). The results of this section do not improve on the results in \cite{tat2013}, but we do track the required regularity more precisely. 

Our argument uses the Sobolev embedding
		\begin{equation} \label{sobemb}
			\| \phi \|_{L^\infty(\mathbb{S}^2)} \lesssim \| \phi \|_{L^2(\mathbb{S}^2)} + \| \Omega^2 \phi \|_{L^2(\mathbb{S}^2)} 
		\end{equation}
to obtain useful $L_r^2L_\omega^\infty(A_m)$ bounds on $g$ and $R_\tau g$. For reference we begin with two preliminary lemmas resulting from a straightforward application of \eqref{sobemb}.

\begin{lemma} \label{lotosupo}
	If $\phi$ satisfies 
	    \[ \| \la r \ra^p \phi \|_{L^2(A_m)} + \sum_{|\alpha| = 2} \| \la r \ra^p \Omega^\alpha \phi \|\ltwoam \lesssim 1 \] 
	then $ 2^{m(1+p)}\| \phi \|\lrsupo \lesssim 1 $.
\end{lemma}

\begin{proof}
	A change of coordinates yields $\|\la r \ra^p \phi \|_{L^2(A_m)}^2 \approx 2^{2m(1+p)} \| \phi \|\lrlo^2$. Then from the Sobolev embedding \eqref{sobemb} we obtain
	\begin{equation} \label{techfix}
	\begin{split}
	 	2^{m(1+p)} \| \phi \|\lrsupo \lesssim \| \la r \ra^p \phi \|_{L^2(A_m)} +  \| \la r \ra^p \Omega^2 \phi \|_{L^2(A_m)} \lesssim 1. 
	 \end{split}
	 \end{equation}
\end{proof}

In the following proofs we use the notation $Q_\ell$ to denote any operator of the form
		\begin{equation} 
			\tau (\partial_i h^i + h^i \partial_i) + \partial_i h^{ij}\partial_j + h_\ell, \qquad h^{i}, h^{ij} \in \ell^1S(r^{-\kappa}), \quad h_\ell \in \ell^1S(r^{-\kappa-2})  \label{Qellform}
		\end{equation}
and use $Q_r$ to denote an operator of the form
		\begin{equation} 
			h^\omega \Delta_\omega + h_r, \quad h^\omega, h_r \in S_{rad}(r^{-\kappa-2}). \label{Qrform}
		\end{equation}

\begin{lemma} \label{znqlessr2}
	If $\phi,$  $S_r \phi,$  $\Omega^2\phi$, $\Omega^2 S_r \phi \in \mathcal{LE}^*$, then
		$ |\phi| \lesssim \la r \ra^{-2}$. Furthermore, if $\phi \in Z^{n,q}$, then $ |\partial_r^p \phi | \lesssim \la r \ra^{-2-p-q}$ for $p \le n-3$.
\end{lemma}

\begin{proof}
	Fix $m$. By assumption, $\la r \ra^{\frac{1}{2}} \phi, 	\la r \ra^{\frac{1}{2}} \Omega^2 \phi, \la r \ra^{\frac{3}{2}} \partial_r \phi, \la r \ra^{\frac{3}{2}} \Omega^2 \partial_r \phi \in L^2(A_m), $
	so by \eqref{lotosupo}, we have $ 2^{\frac{3m}{2}} \| \phi \|\lrsupo \lesssim 1$ and $	2^{\frac{5m}{2}} \| \partial_r \phi\|\lrsupo \lesssim 1$. Using the Fundamental Theorem of Calculus and Cauchy-Schwarz, we find pointwise bounds on $\phi$:
	\begin{align*}
		\| \phi\|_{L^\infty({A_m})} &\lesssim 2^{-\frac{m}{2}} \|\phi\|\lrsupo + 2^{\frac{m}{2}} \| \partial_r \phi \|\lrsupo \lesssim 2^{-2m}.
	\end{align*}
Thus $| \phi| \lesssim \la r \ra^{-2}$ since $m$ was arbitrary.

	Now take $\phi \in Z^{n,q}$. We can write $\rr^p \partial_r^p S_r^p$ as a linear combination of $S_r^k$ for $k \le p$ so $ |\rr^{q+p} \partial_r^p \phi| \lesssim \sum_{k \le p} | \rr^q S_r^k \phi|$. Since $[\Omega, r] = 0$ and $[S_r, r] =r$, it follows by the definition of $Z^{n,q}$ that $\Omega^j S_r^k \rr^{q+p} \partial_r^p \phi \in \mathcal{LE}^*$ when $j + k \le 3 $
if $p \le n-3$. Therefore by the first part of the proposition we have $|\rr^{q+p} \partial_r^p \phi| \lesssim \rr^{-2}$.
\end{proof}

The following calculation will be useful for the remaining lemmas and propositions in this section. Writing
    \[ (\partial_r^2 +\tau^2) =  P_\tau - ( 2r^{-1}\partial_r + r^{-2}\Delta_\omega + Q_\ell + Q_r), \] 
where $Q_\ell, Q_r$ are as in \eqref{Qellform} and \eqref{Qrform}, we obtain
	\begin{equation} \label{precalcs}
	\begin{split}
	    	(\partial_r^2 +\tau^2)rv_{ijk} &= r(\partial_r^2 +\tau^2)v_{ijk} + 2\partial_r v_{ijk} \\
		&= rP_\tau v_{ijk} -r^{-1}\Delta_\omega v_{ijk} -r(Q_\ell +Q_r) v_{ijk}.
	\end{split}
	\end{equation}
Commuting $P_\tau$ with $T^i\Omega^jS^k$ yields 
	\begin{equation} \label{pveccomm} 
			P_\tau v_{ijk} = g_{ij \le k} + Q_\ell(v_{<ijk} + v_{\le i < j k} + v_{\le i \le j < k} ) + Q_r ( v_{<ijk} + v_{\le i j < k} ).
	\end{equation}
Then we rewrite the first term of \eqref{precalcs} using \eqref{pveccomm} to find
    \begin{equation} \label{drtaubnd}
				(\partial_r^2 + \tau^2)rv_{ijk} = -r^{-1}\Delta_\omega v_{ijk} + r(Q_\ell+ Q_r) v_{\le i \le j \le k} + rg_{\le i \le j \le k}.
			\end{equation}

In Proposition \ref{firstptbnd} we state the same pointwise bounds and outgoing radiation condition established in \cite[Proposition 16]{tat2013}. However, we obtain different numerology for the number of vector fields that can be applied to $v = R_\tau g$ so that the results hold. Thus we offer a concrete justification for the change in the vector field numerology but provide only a brief outline of the argument. We will use notation as in Proposition \ref{resbnd2} so that $M$ indicates the regularity assumed for $g$. We take $v_{ijk} = T^i\Omega^jS^k v$ and $g_{ijk} = T^i\Omega^jS^k g$. Similarly, we write $v_{<i<j<k} = T^{<i}\Omega^{<i}S^{<k}v$ and use analogous notation for $g$.

\begin{proposition} \label{firstptbnd}
	Assume $\Im \tau \le 0$. Let $g \in \mathcal{LE}^*$ satisfy~\eqref{gijkbnd} and possibly depend on $\tau$. Set $v = R_\tau g$.
	\begin{enumerate}
		\item[(i)] If $|\tau| \gtrsim 1$, then
			\begin{equation}
				|T^i\Omega^jS^k v(\tau)| \lesssim (|\tau|\la r \ra)^{-1}, \quad i+4j+16k \le M-20. \label{largebnd1}
			\end{equation}
		
		\item[(ii)] If $|\tau| \lesssim 1$, then
			\begin{equation}
				|T^i\Omega^jS^k v(\tau)| \lesssim 
					\begin{cases} \min \{ 1, (|\tau|\la r \ra)^{-1} \} &\quad i=0 \\ 
						\la r \ra^{-1} &\quad i \ge 1  \end{cases}
				  \qquad i+4j+16k \le M-20. \label{smallbnd1}
			\end{equation}
		
		\item[(iii)] If $\tau \in \R \setminus \{ 0 \}$, then we have the outgoing radiation condition:
			\begin{equation} \label{randcond2}
				\lim_{|x| \to \infty} r(\partial_r+i\tau)T^i\Omega^jS^k v(\tau) = 0, \quad i+4j+16k \le M-20.
			\end{equation}
	\end{enumerate}
\end{proposition}

\begin{proof}[proof summary]
    \textbf{(i) and (ii)}
    The estimate 
        \begin{equation} \label{keyest}
		 		\sum_m 2^{\frac{m}{2}}\|(\partial_r^2 + \tau^2)( r v_{ijk})\|\lrsupo \lesssim 1 
			\end{equation}
	implies
		\begin{equation} \label{aresult61}
				|v_{ijk}| \lesssim (\la r \ra |\tau|)^{-1} \qquad \hbox{and} \qquad |\partial_r v_{ijk}| \lesssim \la r \ra^{-1}
			\end{equation}
	using the fundamental solution for $(\partial_r^2+\tau^2)$, which is given by $\tau^{-1}e^{-i\tau|s|}$. 
	
	The pointwise bounds for $|\tau| \gtrsim 1$ and for the case $\rr \gtrsim |\tau|^{-1}$ when $|\tau| \lesssim 1$ are obtained using \eqref{keyest}. To show \eqref{keyest} holds, we bound each term on the right hand side of \eqref{drtaubnd} by applying Lemma \ref{lotosupo} to the assumption $\|g_{ijk}\|\les \lesssim 1$ for $i+4j+16k \le M$ and the resulting fact $\|v_{ijk}\|\letn{} \lesssim 1$ for $i+4j+16 \le M-4$ (which holds by Proposition \ref{resbnd2}).
	
	To see where the vector field loss occurs, consider the term $r^{-1}\Delta_\omega v_{ijk}$ on the right hand side of \eqref{drtaubnd}. We wish to show this term satisfies \eqref{keyest}. Using \eqref{techfix} in Lemma \ref{lotosupo}, it suffices to show
	\begin{equation} \label{goal611} 
		\sum_m  \| \la r \ra^{-\frac{3}{2}} \Delta_\omega v_{ijk} \|\ltwoam + \| \la r \ra^{-\frac{3}{2}}  \Omega^2(\Delta_\omega v_{ijk}) \|\ltwoam \lesssim 1 .
	\end{equation}
	When $|\tau| \gtrsim 1$, we have 
		\begin{equation} \label{largeletau}
			\| \la r \ra^{-\frac{1}{2}} v_{ijk} \|_{L^2(A_m)} \lesssim |\tau| \| \la r \ra^{-\frac{1}{2}} v_{ijk} \|\ltwoam \lesssim \|v_{ijk}\|\letn{} \lesssim 1 \qquad i + 4j + 16k < M-4. 
		\end{equation}
 	Replacing $\Delta_\omega$ by $\sum_{|\alpha| = 2} \Omega^\alpha$ and using \eqref{largeletau} yields
 	\begin{align*} 
 		\sum_m  \| &\la r \ra^{-\frac{3}{2}} \Delta_\omega v_{ijk} \|\ltwoam + \| \la r \ra^{-\frac{3}{2}}  \Omega^2(\Delta_\omega v_{ijk}) \|\ltwoam\\
 		    &= \sum_m \| \la r \ra^{-\frac{3}{2}} v_{i(j+2)k} \|\ltwoam + \| \la r \ra^{-\frac{3}{2}} v_{i(j+4)k} \|\ltwoam \\
 			&\lesssim 1, \qquad i + 4j + 16k \le M - 20.
 	\end{align*}
 	This shows the $M-20$ vector field loss (as compared to the erroneously stated $M-12$ loss in \cite{tat2013}).
	
	When $|\tau| \lesssim 1$ and $|\tau|^{-1} \gtrsim \rr$, then \eqref{aresult61} is insufficient since $(|\tau|\la r \ra)^{-1}$ is unbounded. The advantage in this case is that the $(\la r \ra^{-1} + |\tau|)^{-1}$ weight in the second order term of $\|v_{ijk}\|\letn{}$ is bounded below by $\la r \ra$ when $\la r \ra \lesssim |\tau|^{-1}$.
\end{proof}

We will use Proposition \ref{firstptbnd} to establish pointwise bounds on $(\tau \partial_\tau)^p (ve^{ir\tau})$. Note $ \tau \partial_\tau (ve^{ir\tau}) = [(-S + r(\partial_r + i\tau))v] e^{ir\tau}.$ This motivates the following lemma, which will be used to prove the subsequent proposition stating the pointwise bounds on $(\tau \partial_\tau)^p (ve^{ir\tau})$ for $|\tau| \gtrsim 1$. We remark that while the above calculation shows we are primarily concerned with $(\partial_r + i\tau)^pv_{00k}$, our methods will generate $T$ and $\Omega$ vector fields as we induct in $k$, so we handle $(\partial_r + i\tau)^p v_{ijk}$.

\begin{lemma} 
	Let $g \in Z^{n,q}$. If $\tau \in \R$ and $|\tau| \gtrsim 1$, then $v = R_\tau g$ satisfies the pointwise bounds
	\begin{equation} \label{dritaulargebnd} 
		|(\partial_r +i\tau)^{p} v_{ijk}| \lesssim |\tau|^{p-1} \la r \ra^{-p - 1}, \qquad p \le q, \quad p\le n-3, \quad \hbox{and} \quad i + 4j + 16k \le n - 20 -8p.  
	\end{equation}
\end{lemma}

\begin{proof}
	Note $i+j+k < i + 4j + 16k <n$ so $g \in Z^{n,q}$ implies $g$ satisfies \eqref{gijkbnd} with $M =n$, and the results of Proposition \ref{firstptbnd} apply with $M=n$.
	
	When $p = 0$, \eqref{dritaulargebnd} follows from \eqref{largebnd1}.
	
	Now let $p=1$. All but the last term on the right hand side of \eqref{drtaubnd} are pointwise bounded by $\la r \ra^{-2}$ using \eqref{largebnd1} and the fact that $\kappa \ge 2$. We replace $\Delta_\omega v_{ijk}$ by $v_{i(j+2)k}$ so the bounds hold when $i + 4j + 16k \le n - 28$. For the final term, by Lemma \ref{znqlessr2} we have $ |\partial_r^k g| \lesssim \la r \ra^{-2-k-q}$ when $k \le n-3$, so $|rg| \lesssim \la r \ra^{-2}$ since $1=p\le q$. Thus we have $ |(\partial_r^2 + \tau^2)(rv_{ijk})| \lesssim \rr^{-2}$.
	
	We rewrite $\partial_r^2+\tau^2 = (\partial_r-i\tau)(\partial_r+i\tau)$ and use an integrating factor to write
	    \[ \partial_r\big[\big((\partial_r+i\tau)rv_{ijk}\big)e^{-i\tau r}\big]  =\big((\partial_r^2 +\tau^2)rv_{ijk}\big) e^{-i\tau r} \]
By \eqref{largebnd1} and \eqref{randcond2}, we have $\lim_{r \to \infty} (\partial_r + i\tau)(rv_{ijk}) = 0$, so we can integrate from infinity to find
	\[ |(\partial_r + i\tau)(rv_{ijk})| \lesssim \int_{r}^{\infty} | \la s \ra^{-2}| ds = \la r \ra^{-1}. \]
	It follows that $|(\partial_r + i\tau)v_{ijk}| \lesssim  \la r \ra^{-2}$, as desired.
	
	We proceed by induction. Fix $p$ and assume $|(\partial_r+i\tau)^a\vijk|\lesssim |\tau|^{a-1}\rr^{-a-1} $ for $a <p$  when $i + 4j + 16 \le n - 20 -8a$. 
Applying $(\partial_r+i\tau)^{p-1}$ to \eqref{drtaubnd} we find
	\begin{equation} \label{gencase}
	\begin{split} 
	(\partial_r &- i\tau)(\partial_r + i\tau)^p(rv_{ijk})\\
		&= \sum_{m=0}^{p-1}\left((-1)^{p-m+1}c_m r^{-(p-m)} (\partial_r + i\tau)^m \Delta_\omega  \vijk\right) + r(\partial_r + i\tau)^{p-1} (Q_\ell+ Q_r) \vleijk  \\
			&\quad + C(\partial_r + i\tau)^{p-2} (Q_\ell+Q_r) \vleijk + \big(r(\partial_r + i\tau)^{p-1} + C(\partial_r + i\tau)^{p-2}\big) g_{\le i\le j\le k}.  
	\end{split}	
	\end{equation}
Each term on the right hand side of \eqref{gencase} is bounded in magnitude by $|\tau|^{p-1} \la r \ra^{-p-1}$. The first term is bounded, by the inductive hypothesis, when $i + 4j + 16k \le n - 20 -8p$. For the $Q_\ell$ and $Q_r$ terms, we commute $(\partial_r +i\tau)$ with the coefficients of the operators and view the derivatives as vector fields. The bounds then follow using our assumption $\kappa \ge 2$ once we note
	\begin{equation} \label{drtqcomm}
	\begin{split} 
		[(\partial_r+i\tau)^{p-1}, \rho_{\bullet}^{-\kappa}] &= \sum_{m=1}^{p-1} c_m \rho_\bullet^{-\kappa-m}(\partial_r+i\tau)^{p-1-m}
	\end{split}
	\end{equation}
for $\bullet \in \{ \ell, r\}$. For the last $g_{ijk}$ terms we use Lemma \ref{znqlessr2} (which requires our assumption $p \le n-3$) to find
	\begin{align*}
		|r(\partial_r+i\tau)^{p-1}&g_{\le i \le j \le k}| + |(\partial_r+i\tau)^{p-2}g_{\le i \le j \le k}|\\ 
			&= \big|r \sum_{m = 0}^{p-1} c_m (i\tau)^{m}\partial_r^{p-1-m}g_{\le i \le j \le k}\big| + \big|\sum_{m = 0}^{p-2}(i\tau)^{m}c_m \partial_r^{p-2-m}g_{\le i \le j \le k}\big| \\
			&\lesssim |\tau|^{p-1} \la r \ra^{-1-p}.
	\end{align*}
	The last inequality holds since we assume $p \le q$.

	Now we have
	\[ |(\partial_r - i\tau)(\partial_r + i\tau)^p(r\vijk)| \lesssim |\tau|^{p-1} \la r \ra^{-p -1} \]
	Integrating as before then yields \eqref{dritaulargebnd}. 
\end{proof}

Proposition \ref{proplargetau} establishes the pointwise bounds we will use in the proof of the main theorem. This corresponds to Proposition 17 in \cite{tat2013}. Note we use Tataru's method of proof and correct an error in the proposition statement.

\begin{proposition} \label{proplargetau} 
	Let $g \in Z^{n,q}$. If $\tau \in \R$ and $|\tau| \gtrsim 1$, then $v=R_\tau g$ satisfies the pointwise bounds
	\begin{equation}
		|(\tau \partial_\tau)^p \left(ve^{i\tau\la r \ra}\right)| \lesssim |\tau|^{p-1} \la r \ra^{-1}, \qquad p \le q \quad \hbox{ and } \quad 16p \le n-20. \label{largetauresbnd}
	\end{equation}
\end{proposition}

\begin{proof}
	If $p=0$, then \eqref{largetauresbnd} follows from \eqref{largebnd1}.
	
	To handle $p =1$, we write $ (\tau \partial_\tau)\left(ve^{i \rr \tau}\right) = (-Sv + r(\partial_r + i\tau)v)e^{i \rr  \tau}$. Then $|Sv| \lesssim |\tau|^{-1}\rr^{-1}$ using \eqref{largebnd1},  and $|r(\partial_r+i\tau)v| \lesssim \rr^{-1}$ using \eqref{dritaulargebnd}. Both \eqref{largebnd1} and \eqref{dritaulargebnd} hold under our assumptions $p\leq$ $16p \le n-20$, which implies $p \le n-3$ since $p$ is nonnegative).
	
	For general $p$, we write
		\[ (\tau \partial_\tau)^p\left(ve^{i \rr  \tau}\right) = \left( \sum_{j = 0}^p \sum_{\ell = 0}^j c_{j\ell} r^\ell(\partial_r+i\tau)^\ell(-S)^{p-j}v \right)e^{i  \rr  \tau}. \]
	Each term on the right hand side is bounded by $|\tau|^{\ell-1} \la r \ra^{-1} \lesssim |\tau|^{p-1} \la r \ra^{-1}$ using \eqref{dritaulargebnd} and our assumption $16p \le n-20$.
\end{proof}

Next we find pointwise resolvent bounds for $|\tau| \lesssim 1$. In this case we are interested in the term $R_\tau h_{\nu - 3\kappa -3}$ in our expression for $R_\tau g$ in Proposition \ref{lowferr}. The terms included in $h$ in the proof of Proposition \ref{lowferr} depend on $\tau$, so we consider a $\tau$ dependent function $g$.

\begin{lemma} \label{specialgbnd}
	Let $g \in \mathcal{LE}^*$, possibly depending on $\tau$, satisfy
	\begin{equation}
		\| \la r \ra^q (\partial_r+i\tau)^q T^i\Omega^jS^k g \|_{\mathcal{LE}^*} \lesssim 1, \quad q+i+4j+16k \le n. \label{gassump}
	\end{equation}	
	If $\tau \in \R$, $|\tau| \lesssim 1$, and $p <n-3$ then 
	\begin{align}
		|\partial_r^p(ge^{i\tau r})| \lesssim r^{-p-2}, \label{specbnd1}\\
		|(\partial_r+i\tau)^p g| \lesssim r^{-p-2}, \label{specbnd2}
	\end{align}
and
		\begin{equation}
			|\partial_r^p g| \lesssim r^{-2}. \label{specbnd3}
		\end{equation}

Furthermore, if $|\tau r| \lesssim 1$ then
	\begin{equation}
		|\partial_r^p g| \lesssim r^{-p-2}. \label{specbndsmallr}
	\end{equation}
\end{lemma}

	\begin{proof}
		To prove \eqref{specbnd1} we calculate $S_r (ge^{i\tau r}) = (r(\partial_r +i\tau)g)e^{i\tau r} $ so that
		\begin{equation} 
			S_r^k (ge^{i\tau r}) = \sum_{m=1}^k c_m (r^m (\partial_r + i\tau)^m g)e^{i\tau r}. 
		\end{equation}
We also have $\Omega^j g e^{i\tau r} = (\Omega^j g)e^{i\tau r}$. Finally we calculate
	\begin{align*}
		|T^i ge^{i\tau r}| = \Big| \sum_{m=0}^i c_m (T^{i-m}g)T^m e^{i\tau r} \Big| \lesssim \sum_{m=0}^i c_m|T^m g|
	\end{align*}
since we assumed $|\tau| \lesssim 1$. Therefore \eqref{gassump} implies $ge^{i\tau r} \in Z^{n,0}$ and \eqref{specbnd1} follows by Lemma \ref{znqlessr2}. Then \eqref{specbnd2} follows from \eqref{specbnd1} since $\partial_r^p(ge^{i\tau r}) = ((\partial_r + i\tau)^pg)e^{i\tau r}$.  
	
	To prove \eqref{specbnd3}, note the case $p=0$ follows from \eqref{specbnd2}. Now assume \eqref{specbnd3} holds for $a<p$. We calculate
		\begin{equation} 
			(\partial_r + i\tau)^p = \sum_{m=0}^p c_m \partial_r^m(i\tau)^{p-m} = \partial_r^p + \sum_{m=0}^{p-1} c_m\partial_r^m (i\tau)^{p-m} \label{reuse}
		\end{equation}
so that
	\[ |\partial_r^p g | \lesssim | (\partial_r + i\tau)^pg| + \sum_{m=0}^{p-1} |\partial_r^mg| \]
since we assume $|\tau| \lesssim 1$. Then \eqref{specbnd3} follows from \eqref{specbnd2} and the inductive hypothesis.

Finally, to prove \eqref{specbndsmallr}, we see the case $p=0$ follows from \eqref{specbnd2}. Now assume \eqref{specbndsmallr} holds for $a<p$. By \eqref{reuse} we have
	\[ |r^p \partial_r^p g| \lesssim |r^p (\partial_r+i\tau)^p g| + \sum_{m=0}^{p-1} r^{m}|\partial_r^m(i\tau r)^{p-m}| \lesssim \rr^{-2} \]
where the last inequality follows from \eqref{specbnd2}, the assumption $|\tau r| \lesssim 1$, and the inductive hypothesis.
\end{proof}

\begin{lemma} \label{smalltdritau} 
	Let $g \in Z^{n,0}$, possibly depending on $\tau$, satisfy
	\begin{equation}
		\| \la r \ra^q (\partial_r+i\tau)^q T^i\Omega^jS^k g \|_{\mathcal{LE}^*} \lesssim 1, \quad q+i+4j+16k \le n.
	\end{equation}	
	If $\tau \in \R$ and $|\tau| \lesssim 1$ then $v=R_\tau g$ satisfies
		\begin{equation} \label{smalltauresbnd} 
			|(\partial_r + i\tau)^p v_{ijk} | \lesssim |\tau|^{-1} \la r \ra^{-p-1}, \quad p\le n-3 \quad \hbox{and} \quad i + 4j +16k \le n-20-8p.
		\end{equation}
\end{lemma}

\begin{proof}
	If $p = 0$, \eqref{smalltauresbnd} follows from \eqref{smallbnd1}.
	
	To handle $p=1$, we again use \eqref{drtaubnd}. All but the last term are bounded by $|\tau|^{-1} \la r \ra^{-2}$ using \eqref{smallbnd1}. We note the assumption $g \in Z^{n,0}$ does not allow us to use Lemma \ref{znqlessr2}:  $g$ may depend on $\tau$, so $Sg \neq S_r g$. Take $\psi = (\partial_r + i\tau)(rv_{ijk})$. Then the radiation condition \eqref{randcond2} allows us to integrate from infinity as before to find
	\begin{align*} 
		|\psi(r_0)e^{-i\tau r_0}| &= \left| \int_{r_0}^{\infty} (r^{-1}\Delta_\omega v_{ijk} + r(Q_\ell + Q_r)v_{\le i \le j \le k}) e^{-i\tau r} \dd r + \int_{r_0}^\infty rg_{\le i \le j \le k} e^{-i\tau r} \dd r \right| \\
			&\lesssim |\tau|^{-1} \la r_0 \ra^{-1} + \left| \int_{r_0}^\infty rg_{\le i \le j \le k} e^{-i\tau r} \dd r \right|.
	\end{align*}
	For the last term we integrate by parts and use Lemma \ref{specialgbnd} to calculate
	\begin{align*}
		& \left| \int_{r_0}^\infty rg_{\le i \le j \le k} e^{i\tau r} \partial_r \left( \frac{e^{-2i\tau r}}{-2i \tau} \right) \dd r \right| \\
			&\qquad = \left| \frac{-i}{2\tau} r_0 g_{\le i \le j \le k}(r_0) e^{-i\tau r_0} + \frac{1}{2i\tau} \int_{r_0}^\infty e^{-2i\tau r} \Big( ge^{ir\tau} + r\partial_r(ge^{ir\tau})\Big) \dd r \right| \\
			&\qquad \lesssim |\tau|^{-1} \la r_0 \ra^{-1}.		
	\end{align*}

	Thus we have
		\[ |r(\partial_r + i\tau)v_{ijk} | = |(\partial_r + i\tau)(rv_{ijk}) - v_{ijk}| \lesssim |\tau|\la r \ra^{-1} \]
	so that $ |(\partial_r + i\tau) v_{ijk}|\lesssim |\tau|^{-1} \la r \ra^{-2}$, as desired. 
	
	We proceed by induction. Fix $p$ and assume $|(\partial_r+i\tau)^a\vijk| \lesssim |\tau|^{-1}\rr^{-a-1}$ for $a < p$. We again use \eqref{gencase}. All but the $g_{ijk}$ terms are bounded by $|\tau|^{-1} \la r \ra^{-p-1}$ using the inductive hypothesis and \eqref{drtqcomm}. For the $g_{ijk}$ terms we integrate by parts as in the $p=1$ case to find
	\[
		\left| \int_{r_0}^{\infty} \big(r(\partial_r+i\tau)^{p-1} + C(\partial_r+i\tau)^{p-2}\big)g_{\le i \le j \le k} e^{-i\tau r} \dd r \right|  \lesssim |\tau|^{-1}\la r_0 \ra^{-p}.
	\]
	Then \eqref{smalltauresbnd} follows.

\end{proof}

In Proposition \ref{propsmalltau} we establish the pointwise resolvent bounds which will be used in the proof of the main theorem for small $\tau$. The result is the same as that in Proposition 18 in \cite{tat2013} with a more precise statement on the regularity requirements.

\begin{proposition} \label{propsmalltau} 
	Let $g \in Z^{n,0}$, possibly depending on $\tau$, satisfy
	\begin{equation}
		\| \la r \ra^q (\partial_r+i\tau)^q T^i\Omega^jS^k g \|_{\mathcal{LE}^*} \lesssim 1, \quad q+i+4j+16k \le n.
	\end{equation}	
	If $\tau \in \R$ and $|\tau| \lesssim 1$ then $v=R_\tau g$ satisfies the following pointwise bounds:
	\begin{enumerate}
		\item If $\la r \ra \lesssim |\tau|^{-1}$, then 
			\begin{equation} \label{smalltsmallrbnd}
				|(\tau \partial_\tau)^p v| \lesssim 1, \quad 16p \le n-20.
			\end{equation}
			
		\item If $\la r \ra \gtrsim |\tau|^{-1}$, then 
			\begin{equation} \label{smalltlargerbnd}
				|(\tau \partial_\tau)^p \left( ve^{i\tau \la r \ra} \right) | \lesssim (|\tau|\la r \ra)^{-1}, \quad 16p \le n-20.
			\end{equation}
	\end{enumerate}
\end{proposition}

\begin{proof}
\textbf{ 1. Small r: $\rr \lesssim |\tau|^{-1}$} 

We write $ \tau \partial_\tau = -S + r\partial_r$ and find
\[ (\tau \partial_\tau)^p v = \sum_{m=0}^p c_m(r\partial_r)^{p-m} (-S)^m v. \]
Since $(r\partial_r)^p = \sum_{j=0}^p c_j r^j \partial_r^j$, it is sufficient to show $| r^p \partial_r^p v_{00k}| \lesssim 1$ for $16k \le n -20-16p$. As before, we will use \eqref{drtaubnd}, which introduces $\Omega$ and $T$ vector fields, so we will instead bound $|r^p \partial_r^p v_{ijk}|$ then set $i,j =0$. 

When $p=0$, we have $|v_{ijk}| \lesssim 1$ by \eqref{smallbnd1} for $i + 4j + 16k \le n- 20$. When $p=1$, we have $|\partial_r v_{ijk}| \lesssim |\nabla v_{ijk}|  \lesssim \la r \ra^{-1}$ when $i + 4j + 16k \le n- 20$ by \eqref{smallbnd1}. 

Fix $p$ and assume $|r^a \partial_r^av_{ijk}| \lesssim 1$ for $a < p$ when $i + 4j + 16k \le n -20 - 16a$. Applying $r^{p-2}\partial_r^{p-2}$ to $r^2\partial_r^2v_{ijk}$ and commuting yields
	\begin{equation} 
		r^{p}\partial_r^{p} v_{ijk} = r^{p-2}\partial_r^{p-2} (r^2 \partial_r^2 v_{ijk}) - c_1 r^{p -2} \partial_r^{p-2} v_{ijk} - c_2 r^{p -1} \partial_r^{p-1} v_{ijk}. \label{rpdrpv}
	\end{equation}
The last two terms in \eqref{rpdrpv} are bounded by the inductive hypothesis. To handle the first term in \eqref{rpdrpv} we use \eqref{drtaubnd} and obtain
	\[ r^2\partial_r^2 v_{ijk} = -\Delta_\omega v_{ijk} + r^2 (Q_\ell + Q_r) v_{\le i\le j \le k} - r^2\tau^2v_{ijk} + r^2 g_{\le  i \le j \le k} - 2r\partial_rv_{ijk}. \]
Now we calculate
	\begin{align*}
		|r^{p-2}&\partial_r^{p-2} (-\Delta_\omega v_{ijk} + r^2 (Q_\ell + Q_r) v_{\le i\le j \le k} - r^2\tau^2v_{ijk} + r^2 g_{\le i \le j \le k} - 2r\partial_rv_{ijk} )|\\
			&\lesssim \Big| (r^{p-2}\partial_r^{p-4} + c_1 r^{p-1}\partial_r^{p-3} + c_2 r^{p}\partial_r^{p-2}) [(Q_\ell + Q_r)v_{\le i \le j \le k} + g_{\le i \le j \le k}  - \tau^2 v_{ijk}   ] \Big| \\
				&\qquad + |r^{p-2}\partial_r^{p-2} v_{i(j+2)k}| + |(r^{p-2} \partial_r^{p-2} + r^{p-1}\partial_r^{p-1}) v_{ijk}|.
	\end{align*}
The $Q_\ell$ and $Q_r$ terms are handled in a manner analogous to the argument using \eqref{drtqcomm}. Each term on the right hand side is then bounded by the inductive hypothesis, the assumption $|\tau r| \lesssim 1$, and Lemma \ref{specialgbnd}.

	\textbf{ 2. Large r: $\rr \gtrsim |\tau|^{-1} $ } 
	
As in Proposition \ref{proplargetau}, it suffices to prove the proposition for $(\tau \partial_\tau)^k \left(v e^{i\tau r} \right)$. 
	
	If $p=0$, the desired bound follows by \eqref{smallbnd1}.
		
		Fix $p$ and assume $ |(\tau\partial_\tau)^a \Big( ve^{i\tau r} \Big)| \lesssim (|\tau|\rr)^{-1} $ when $a < p$. We write
		\[ (\tau \partial_\tau)^p (ve^{ir\tau}) = \left( \sum_{j=0}^p \sum_{\ell=0}^j c_{j\ell} r^\ell (\partial_r+i\tau)^\ell (-S)^{p-j}v  \right)e^{ir\tau}. \]
		By Lemma \ref{smalltdritau}, each term on the right hand side is bounded by $ |\tau|^{-1} \la r \ra^{-1}$, as desired.
\end{proof}

\end{section}

\begin{section}{Proof of Main Theorem}
We are now ready to prove Theorem \ref{thetheorem}. By assumption the initial data satisfies $u_0 \in Z^{\nu +1, \kappa}$ and $u_1 \in Z^{\nu, \kappa+1}$. Since $P^1: Z^{n,q} \to Z^{n-1,q+\kappa}$, we can write
	\[ R_\tau(-i\tau u_0 + P^1 u_0 -u_1) = R_\tau(\tau g_{\kappa}^{\nu+1} + g_{\kappa+1}^{\nu}) \]
for some $g_{\kappa}^{\nu+1} \in Z^{\nu+1,\kappa}$ and some $g_{\kappa+1}^{\nu} \in Z^{\nu,\kappa+1}$. Therefore \eqref{ufin} becomes
	\begin{equation} \label{upre}
		u(t,x) = \frac{1}{\sqrt{2\pi}} \int_{\R}  R_\tau(\tau \gln{\kappa}{\nu+1} + \gln{\kappa+1}{\nu})e^{i\tau t} \dd \tau. 
	\end{equation}
We will use cutoff functions to break $u$ into high and low frequency components as in \eqref{uhigha} and \eqref{ulowa}.

\subsection{High Frequency Case} \textbf{($|\tau \gtrsim 1$)}
We will decompose the expression $R_\tau(\tau \gln{\kappa}{\nu+1} + \gln{\kappa+1}{\nu})$ in \eqref{upre} via an iterative argument, so we begin by writing the high frequency part of $u$ as
	\begin{equation} \label{uhighitform}
		u_{>1}(t,x) =  \frac{1}{\sqrt{2\pi}} \int_{\R} \chi_{>1}(|\tau|) R_\tau(\tau f_0 + g_0)e^{i\tau t} \dd \tau 
	\end{equation}
for $f_0 = \gln{\kappa}{\nu+1} \in Z^{\nu+1,\kappa}$ and $g_0 = \gln{\kappa+1}{\nu} \in Z^{\nu,\kappa+1}$. 

We approximate $R_{\tau}(\tau f_0 +g_0) \approx \tau^{-1}f_0$ and denote the error by $u_1$. Direct calculation shows
		\begin{align*}
			P_\tau u_1 =  (g_0 - iP^1f_0) + \tau^{-1}(\Delta + P^2)(-f_0)  =: f_1 + \tau^{-1}g_1.
		\end{align*}
	Note $P^2: Z^{p,q} \to Z^{p-2,q+\kappa}$, $P^1: Z^{p,q} \to Z^{p-1,q+\kappa}$, and  $\Delta: Z^{p,q} \to Z^{p-2,q+2}$. To see the latter, write $\Delta$ in spherical coordinates and viewing the derivatives as vector fields:
	\begin{align*}
		\Delta = \partial_r^2 + \frac{2}{r}\partial_r + r^{-2} \Delta_\omega = r^{-2}(S_r^2 -S_r) + 2r^{-2}S_r + r^{-2}\Omega^2.
	\end{align*} 
Thus $f_1 \in Z^{\nu,\kappa+1}$ and $g_1 \in Z^{\nu-1,\kappa+2}$. Now we have
	\[ R_\tau(\tau f_0 + g_0) = \tau^{-1}f_0 + R_\tau(f_1 + \tau^{-1}g_1 ). \]
	
Next we reiterate the process and approximate $R_\tau(f_1 + \tau^{-1}g_1 ) \approx \tau^{-2}f_1$: 
	\begin{align*} 
		R_\tau(f_1 + \tau^{-1}g_1 ) &= \tau^{-2}f_1 + R_\tau(\tau^{-1}(g_1-iP^1f_1) +\tau^{-2}(\Delta+P^2)(-f_1) )\\
			&=: \tau^{-2}f_1 + R_\tau(\tau^{-1}f_2 +\tau^{-2}g_2 )
	\end{align*}
where $f_2 \in Z^{\nu-1,\kappa+2}$ and $g_2 \in Z^{\nu-2,\kappa+3}$. Further reiterating this process a total of $J$ times we obtain the representation
	\begin{equation} \label{rtauiteration}
		R_\tau (\tau f_0 + g_0) = \underbrace{\sum_{j=0}^{J-1} \tau^{-j-1} f_j}_{=:\hat{u}_a} + \underbrace{\tau^{-J}R_{\tau}\left( \tau f_J + g_J \right)}_{=: \hat{u}_b} 
	\end{equation}
	where $f_j = g_{j-1} - iP^1f_{j-1} \in Z^{\nu+1-j,\kappa+j}$ and $g_j = -(\Delta + P^2)f_{j-1} \in Z^{\nu-j,\kappa+1+j}$. We plug \eqref{rtauiteration} into \eqref{uhighitform} and bound each term separately.
	
By Lemma \ref{znqlessr2} we see $|f_j| \lesssim \rr^{-2-\kappa-j}$, and we calculate for any $N \ge 1$:
	\begin{align*}
		\left|\int_{\tau \in \R} \chi_{>1}(|\tau|) \hat{u}_a(\tau)e^{it\tau} \dd \tau \right| 
		&\lesssim \sum_{j=0}^{J-1} \la r \ra^{-\kappa -2-j}\la t \ra^{-N} \left| \int \partial_{\tau}^N(\chi_{>1}(|\tau|) \tau ^{-j-1}) e^{i\tau t} \dd \tau \right| \\
		&\lesssim \la t \ra^{-N} \la r \ra^{-\kappa-2}.
	\end{align*}
	
	By \eqref{largetauresbnd}, we have $ |(\tau \partial_{\tau})^{\ell} (\hat{u}_b(\tau)e^{i\tau \la r \ra})| \lesssim |\tau|^{\ell-J}\la r \ra^{-1}$ for $16\ell \le \nu - J -20$ and $\ell \le \kappa+J$. We use this to calculate
	\begin{align*}
			\left|\int_{\tau \in \R} \chi_{>1}(\tau) \hat{u}_b(\tau)e^{it\tau} \dd \tau \right| 				
				&\approx \la t-r \ra^{-N} \left| \int \tau^{-N} \left[ \sum_{\ell=0}^N (\tau\partial_{\tau})^{\ell}  (\chi_{>1}(|\tau|)\hat{u}_b e^{ir\tau}) \right] e^{i(t-r)\tau} \dd \tau \right| \\
				&\lesssim \la r \ra^{-1} \la t -r \ra^{-N}
		\end{align*}
for $J \ge 2$, $N \le \kappa +J$, and $16N \le \nu - J- 20$. 
	
	Combining the above results, we find 
	\[ |u_{>1}(t,x)| \lesssim \la t \ra^{-N} \la r \ra^{-\kappa-1} + \la r \ra^{-1} \la t -r \ra^{-N}. \]
Theorem \ref{thetheorem} then follows in the high frequency case if we take $J=2$ and $N=\kappa+2$ since the resulting requirement on $\nu$ is $\nu \ge 16\kappa +54$, which is satisfied by our assumption $\nu \ge 31\kappa + 168$.

\subsection{Low Frequency Case} \textbf{($|\tau \lesssim 1$)}
Note Proposition \ref{propexp} shows $|R_0\gl| \lesssim \rr^{-1}$ for $1 \le \l \le \kappa+1$. Therefore the terms of the form $R_0\gl$ in our expressions for $R_\tau \gl$ in Proposition \ref{lowferr} can be included in the terms of the form $F_m(x)$. Thus we see
	\begin{equation} \label{rtaurewrite1}
	\begin{split}
		R_\tau&(\tau \gln{\kappa}{\nu+1} + \gln{\kappa+1}{\nu})\\
		 &= R_\tau\Big(\chi_{>|\tau|^{-1}}(r)(\tau\gln{\kappa}{\nu+1} + \gln{\kappa+1}{\nu})\Big) + \Big(\sum_{m=0}^\kappa \tau^mF_m(x)\enitr\Big) + \tau^\kappa \epsilon(r,\tau) e^{-i\tau \rr} \\
		 	& \qquad + \tau^{\kappa+1}(R_\tau h_{\nu-3\kappa-3})
	\end{split} 
	\end{equation}
where 
	\begin{equation} \label{logliketerm}
	    \epsilon(r,\tau) = \tau^\kappa\rr^{-1}\epsilon_1(r\wedge |\tau|^{-1}) + \tau^{\kappa+1}\big(\epsilon_2(r \wedge |\tau|^{-1}) - \epsilon_2(|\tau|^{-1}) \big) 
	 \end{equation}
with $\epsilon_1, \epsilon_2 \in S_{rad}(\log r)$ and $r \wedge |\tau|^{-1} \approx \hbox{min}(r,|\tau|^{-1})$ is smooth.
	
Using a decomposition as in \eqref{rtauiteration}, the first term on the right hand side of \eqref{rtaurewrite1} becomes
	\begin{equation}
	\begin{split}
		R_\tau\Big(\chi_{>|\tau|^{-1}}(r)(\tau\gln{\kappa}{\nu+1} + \gln{\kappa+1}{\nu})\Big)
			&= \sum_{j=0}^{J-1} \tau^{-j-1}(\chi_{>|\tau|^{-1}}\gln{\kappa+j}{\nu+1-j}) \\
			&\qquad + \sum_{M=J-1}^J \tau^{-M} R_\tau(\chi_{>|\tau|^{-1}} \gln{\kappa+1+M}{\nu-M}). \label{uadecomp}
	\end{split}
	\end{equation}

The term $ \sum_{M=J-1}^J \tau^{-M} R_\tau(\chi_{>|\tau|^{-1}} \gln{\kappa+1+M}{\nu-M})$
can be written as $\tau^{\kappa+1}(R_\tau h_{n})$ for $n = \min(\nu-J,J+\kappa)$ with $h_n$ as in \eqref{hbnd}. This holds for any $J \ge 1$, and we will pick a suitable $J$ once we have determined the necessary regularity for $h_n$ in order for the theorem to hold. Indeed, direct calculation shows
    \[ \big| r^\ell (\partial_r +i\tau)^\ell T^i\Omega^j S^k \tau^{-N}\chi_{>1}(r\tau) g \big| \lesssim \big|r^NT^{\le i} \Omega^{\le j}S_r^{\le k+\ell} g \big|. \]
It follows that
	\[ \| r^\ell (\partial_r+i\tau)^\ell T^i\Omega^jS^k\tau^{-\kappa-1-M}\chi_{>|\tau|^{-1}}(r)\gln{\kappa+1+M}{\nu-M}\|\les \lesssim 1 \]
for $\ell \le M + \kappa +1$ and $i + j +k + \ell \le \nu -M$. 

Motivated by the above calculations, we define 
		\begin{align*} 
		\hat{u}_a(\tau) &:= \sum_{j=0}^{J-1} \tau^{-j-1}(\chi_{>|\tau|^{-1}}(r)\gln{\kappa+j}{\nu+1-j}), \qquad 
		\hat{u}_b(\tau) := \sum_{m=0}^\kappa \tau^m F_m(x)\enitr,\\
		\hat{u}_c(\tau) &:= \tau^\kappa \epsilon(r,\tau)\enitr , \quad \hbox{ and } \quad
		\hat{u}_d(\tau) := \tau^{\kappa+1}\Big(R_\tau h_{\nu-3\kappa-3} + R_\tau h_n \Big)
		\end{align*}
for fixed $J \ge 1$ and $n = \min(\nu-J,J+\kappa)$ and write the low frequency part of $u$ as
	\begin{equation} \label{uabcd}
		u_{<1}(t,x) \simeq \int_\R \chi_{<1}(|\tau|) \big(\hat{u}_a(\tau) + \hat{u}_b(\tau) + \hat{u}_c(\tau) +\hat{u}_d(\tau)\big)e^{it\tau} \dd \tau. 
	\end{equation}

To handle $\int \chi_{<1}(|\tau|)\hat{u}_{{a}}e^{it\tau} \dd \tau$ we calculate for any $N \ge 1$
	\begin{equation} \label{uafirst}
	\begin{split}
		\Big| \int_R \chi_{<1}(|\tau|)& \hat{u}_{a} e^{it\tau} \dd \tau \Big|
		\lesssim \sum_{j=0}^{J-1}\rr^{-\kappa-j-2} t^{-N} \Big| \int_\R \partial_\tau^N \Big( \chi_{<1}(|\tau|)\chi_{>1}(r|\tau|)\tau^{-j-1}\Big) e^{it\tau} \dd \tau \Big|\\
			&\lesssim \sum_{j=0}^{J-1} \sum_{\ell=0}^N \sum_{i=0}^\ell \rr^{-\kappa-j-2} t^{-N} \Big| \int_\R \chi_{<1}^{(N-\ell)}(|\tau|)\chi_{>1}^{(\ell-i)}(r|\tau|) r^{\ell-i}\tau^{-j-i-1} e^{it\tau} \dd \tau \Big| \\
			&\lesssim \rr^{-\kappa-2+N}\la t \ra^{-N}.
	\end{split} 
	\end{equation}
If $r \ge \frac{t}{2}$, we take $N =1$.  If $r < \frac{t}{2}$, we take $N= \kappa+2$. This yields
	\begin{equation}
		\Big| \int_\R \chi_{<1}(|\tau|) \hat{u}_{a} e^{it\tau} \dd \tau \Big| \lesssim \la t \ra^{-1} \la t - r \ra^{-\kappa-1}. \label{uaresult}
	\end{equation}

For the $\hat{u}_b$ term in \eqref{uabcd} we find
	\begin{equation}
	\begin{split}
		\Big| \int_\R \chi_{<1}(|\tau|) \hat{u}_b e^{it\tau} \dd \tau \Big| &\lesssim \frac{1}{\la t - r \ra^N} \sum_{m=0}^\kappa \left| \int \partial_\tau^N \Big(\chi_{<1}(|\tau|)\tau^m F_m(x)\Big)e^{i(t-\rr)\tau} \dd \tau \right| \\
			&\lesssim \la r \ra^{-1}\la t -r \ra^{-N}
	\end{split} \label{ubresult}
	\end{equation}
for any $N$.

Next we consider the $\hat{u}_c$ term in \eqref{uabcd}. Recall $\hat{u}_c = \tau^\kappa \epsilon(r,\tau) \enitr$ with $\epsilon(r,\tau)$ as in \eqref{logliketerm}. Here we define
    \[ \epsilon_j^m(r) := \int_0^r \bam(\rho) \partial_\rho \epsilon_j(\rho) \dd \rho \]
so that $ \epsilon_j(r) = \sum_{m \ge 0} \epsilon_j^m(r)$. We note $|\epsilon_j^m| \lesssim 1$ uniformly in $m$ and for $N \ge 1$ we have
    \begin{equation} \label{ederivbd}
		|\partial_r^N\epsilon_j^m(r)| \lesssim 2^{-mN} \mathbf{1}_{\{r \approx 2^m\}}(r). 
	\end{equation}
Furthermore, we see
	\begin{equation} \label{ecut}
		 \epsilon_j^m(r) \equiv \begin{cases}  0, \qquad r \ll 2^m \\
	  c_m, \qquad r \gg 2^m. \end{cases} 
	\end{equation}
Now we define
    \[ \varphi^m (r,\tau) := \tau^\kappa \rr^{-1} \epsilon_1^m(r\wedge |\tau|^{-1}) + \tau^{\kappa+1}\big[\epsilon_2^m(r\wedge |\tau|^{-1})-\epsilon_2^m(|\tau|^{-1})\big] \]
so that $\hat{u}_c = \sum_{m\ge 0} \varphi^m \enitr$. We fix $r$ and break the summation up into the cases $2^m \ll r$, $2^m \gg r$, and $2^m \approx r$ separately.

When $2^m \ll r $ we have $\varphi^m =\tau^\kappa \rr^{-1}\epsilon_1^m(|\tau|^{-1})$. Thus we find
    \begin{equation} \label{mllogr}
	\begin{split}
		\sum_{m \ll \log r} \Big|\int_\R \chi_{<1}(|\tau|)\varphi^m e^{i(t-\rr)\tau} \dd \tau \Big|
			&\lesssim \sum_{m \ll \log r} \rr^{-1} 2^{m(N- (\kappa+1))} |t-\rr|^{-N}\\
			&\lesssim \la t+r \ra^{-1} \la t-r \ra^{-1}.
	\end{split} 
	\end{equation}
To obtain the first inequality we change variables by $\tau \mapsto 2^{-m}\tau$ and use the fact that $\epsilon_1^m(2^m|\tau|^{-1}) = 0$ for $2^{-m}|\tau| \gg 2^{-m}$ along with the observations $|\epsilon_1^m(r)| \lesssim 1$ and \eqref{ederivbd}. The second inequality uses different arguments depending on the relative sizes of $r$ and $t$. If $t < 2r$, we break up the sum into $2^{-m}<|t-\rr|$ (where we take $N=\kappa+2$) and $|t-\rr| \le 2^m <r$ (where we take $N=0$). If $t \ge 2r$, we use the fact that $|t-\rr|\lesssim \la t+r \ra^{-1}$ and set $N=\kappa+2$.

When $2^m \gg r$, we have $\varphi^m = \tau^{\kappa+1}\epsilon_2^m(|\tau|^{-1})$. Thus we find
	\begin{equation} \label{mglogr}
	\begin{split}
		\sum_{m \gg \log r} \Big|\int_\R \chi_{<1}(|\tau|) \varphi^m e^{i(t-\rr)\tau} \dd \tau \Big|
			&\lesssim \sum_{m \gg \log r}  2^{m(N- (\kappa+2))}  |t-\rr|^{-N}\\
			&\lesssim \la t+r\ra^{-1}\la t-r\ra^{-\kappa-1}.
	\end{split} 
	\end{equation}
The first inequality is obtained the same way we found the first inequality in \eqref{mllogr}. The second inequality again uses different arguments depending on the relative sizes of $t$ and $r$. If $t \le 2r$ we have $\rr^{-1} \lesssim \la t+r\ra^{-1}$ and $\rr^{-1} \lesssim \la t-r\ra^{-1}$ so the inequality follows taking $N=0$. It $t >2r$ we have $|t-\rr|^{-1} \lesssim \la t+r\ra^{-1}$ and we break the sum into $2^m < |t-\rr|$ (where we take $N=\kappa+3$) and $2^m \ge |t-\rr|$ (where we take $N=0$).

When $2^m \approx r$, we have that $\varphi^m$ is composed of terms as in the $2^m \ll r$ and $2^m \gg r$ cases. Thus we argue as above with the added benefit that the summation is finite in $m$ to find
    \begin{equation} \label{melogr}
        \sum_{m \approx \log r} \Big|\int_\R \chi_{<1}(|\tau|) \varphi^m e^{i(t-\rr)\tau} \dd \tau \Big| \lesssim \la t+r\ra^{-1}\la t-r\ra^{-\kappa-1}.
    \end{equation}
Now \eqref{mllogr}, \eqref{mllogr}, and \eqref{melogr} yield
    \begin{equation} \label{ucresult}
        \left| \int_\R \chi_{<1}(|\tau|)\hat{u}_c(\tau)e^{it\tau} \dd \tau \right| \lesssim \la t+r\ra^{-1}\la t-r\ra^{-\kappa-1}.
    \end{equation}

Finally we handle the $\hat{u}_d$ term. We do our calculations for $h_{n}$ then apply the results to $h_{\nu-3\kappa-3}$. We will use Proposition \ref{propsmalltau}. Note $\tau^\ell \partial_\tau^\ell$ can be written as a linear combination of $(\tau\partial_\tau)^a$ with $1 \le a \le \ell$, so \eqref{smalltsmallrbnd} and \eqref{smalltlargerbnd} hold for $(\tau \partial_\tau)^\ell$ replaced by $\tau^\ell \partial_\tau^\ell$.

We split up the $\rr \lesssim |\tau|^{-1}$ and $\rr \gtrsim |\tau|^{-1}$ cases using cutoff functions. When $\rr \lesssim |\tau|^{-1}$, we first consider the case where $t \le 2r$. Here we have
	\begin{align*}
		\left| \int_\R \chi_{<1}(\rr|\tau|)\tau^{\kappa+1} (R_\tau h_n) e^{it\tau} \dd \tau  \right| \lesssim \rr^{-\kappa-2} \lesssim \la t \ra^{-\kappa-2}
	\end{align*}
by \eqref{smalltsmallrbnd}. When $2r < t$, we calculate
	\begin{align*}
		\left| \int_\R \chi_{<1}(\rr|\tau|)\tau^{\kappa+1} (R_\tau h_n) e^{it\tau} \dd \tau  \right| &\le \left| \int_0^{\frac{1}{t}} |\tau|^{\kappa+1} \dd \tau \right| + \left| \int_{\frac{1}{t}}^\infty \chi_{<1}(r|\tau|) \tau^{\kappa+1}R_\tau h e^{it\tau} \dd \tau \right|\\
			&\lesssim t^{-\kappa-2} 
	\end{align*}
for $16(\kappa+4) < n -20$. The inequality for the $\int_{\frac{1}{t}}^\infty \cdot \dd\tau$ term follows from integration by parts $\kappa + 4$ times.

Now we consider when $|\tau|^{-1} \lesssim \rr$. Take $\phi(x,\tau) := \chi_{<1}(|\tau|) \chi_{>1}(\rr|\tau|) \tau^{\kappa+1}(R_\tau h_n)e^{i\tau r}$ and use \eqref{smalltlargerbnd} to calculate
	\begin{equation}
	\begin{split}
		|\partial_\tau^N \phi| \lesssim |\tau|^{\kappa-N}\rr^{-1} + \sum_{i=0}^{N-1} \chi_{>1}^{(N-i)} (\rr|\tau|) \rr^{N-\kappa-1}
	\end{split} \label{derivbnds}
	\end{equation}
for $16N \le n-20$, since $\rr |\tau| \approx 1$ on the support of $\chi_{>1}^{(j)}(\rr |\tau|)$ for $j \ge 1$. If $2r \le t$, then $|t - \rr|^{-1} \lesssim \la t+ r \ra^{-1}$. We use \eqref{derivbnds} to find for $N \ge \kappa+2$ and $16N \le n-20$
	\begin{equation} \label{2ndpiece1}
	\begin{split}
		\Big| \int_\R &\phi e^{i\rr\tau} e^{i(t-\rr)\tau} \dd \tau \Big|\\
			&= t^{-1} \la t-r \ra^{-N} \Big| \int \partial_\tau \Big[ (\partial_\tau^N\phi)e^{-i\rr\tau}\Big] e^{it\tau} \dd \tau \Big|\\
			&\lesssim t^{-1} \la t-r \ra^{-N} \Big( \int_{r^{-1} < |\tau| < 1} | \partial_\tau^{N+1}\phi| \dd \tau  + \rr \int_\R |\partial_\tau^N \phi| \dd \tau \Big)\\
			&\lesssim  t^{-1} \la t-r \ra^{-N} \Big(\rr^{-1}\int_{\rr^{-1}}^\infty |\tau|^{\kappa-N-1} \dd \tau + \int_{|\tau|\approx \rr^{-1}} \rr^{N-\kappa} \dd \tau +  \int_{\rr^{-1}}^\infty |\tau|^{\kappa - N} \dd \tau \Big)\\
			&\lesssim \frac{1}{t \la t-r \ra^{\kappa+1}} \frac{\rr^{N-\kappa-1}}{\la t+ r \ra^{N-\kappa-1}}.
	\end{split} 
	\end{equation}

If $ t < 2r$ then $1 \lesssim \frac{\rr}{\la t+ r \ra}$, $\rr^{-1} \lesssim t^{-1}$, and $\rr^{-1} < |t-\rr|^{-1} < \infty$. We write
    \begin{equation} \label{theidea}
	\begin{split}
		\Big| \int_\R \phi e^{i\rr\tau} e^{i(t-\rr)\tau} \dd \tau \Big| &\le \int_0^{|t-\rr|^{-1}} \tau^{\kappa} \rr^{-1} \dd \tau + \Big| \int_{|t-\rr|^{-1}}^\infty \phi(\tau) e^{i(t-\rr)\tau} \dd \tau \Big| \\
			&\lesssim t^{-1} |t-\rr|^{-\kappa-1} + \rr^{-1}|t-\rr|^{-\kappa-1}\\
			&\lesssim \frac{1}{t \la t-r \ra^{\kappa+1}} \frac{\rr^{N-\kappa-1}}{\la t+ r \ra^{N-\kappa-1}}
    \end{split} 
	\end{equation}
for $2r >t$ and $16(\kappa+4) \le n-20$.

Combining \eqref{2ndpiece1}, and \eqref{theidea} then yields
	\begin{equation} 
		\Big|\int_\R \chi_{<1}(|\tau|) \tau^{\kappa+1} (R_\tau h_n) e^{it\tau} \dd \tau \Big| \lesssim \frac{1}{\la t + r \ra^{\kappa+2}} + \frac{1}{t \la t-r \ra^{\kappa+1}} \frac{\rr^{N-\kappa-1}}{\la t+ r \ra^{N-\kappa-1}} \label{udresult}
	\end{equation}
for $16(\kappa+4) \le n-20$. Thus $n = \min(\nu-J,\kappa+J) \ge 16\kappa+84$ for fixed $J \ge 1$ so that the results hold for $\nu \ge 31 \kappa + 168$.

The statement of the main theorem then follows for $|\tau| \lesssim 1$ by \eqref{uabcd}, \eqref{uaresult}, \eqref{ubresult}, \eqref{ucresult}, and \eqref{udresult}.
\end{section}

\printbibliography

@article{ bcmp,
  title={Localized energy for wave equations with degenerate trapping},
  author={Booth, Robert and Christianson, Hans and Metcalfe, Jason and Perry, Jacob},
  journal={arXiv preprint arXiv:1712.05853},
  year={2017}
}

@article{ bonyhaf2,
  title={Improved local energy decay for the wave equation on asymptotically Euclidean odd dimensional manifolds in the short range case},
  author={Bony, Jean-Francois and H{\"a}fner, Dietrich},
  journal={Journal of the Institute of Mathematics of Jussieu},
  volume={12},
  number={3},
  pages={635--650},
  year={2013},
  publisher={Cambridge University Press}
}

@article{mst,
  title={Local energy decay for scalar fields on time dependent non-trapping backgrounds},
  author={Metcalfe, Jason and Sterbenz, Jacob and Tataru, Daniel},
  journal={arXiv preprint arXiv:1703.08064},
  year={2017}
}

@article{tat2013,
  title={Local decay of waves on asymptotically flat stationary space-times},
  author={Tataru, Daniel},
  journal={American Journal of Mathematics},
  volume={135},
  number={2},
  pages={361--401},
  year={2013},
  publisher={Johns Hopkins University Press}
}

@article{kss,
  title={Almost global existence for some semilinear wave equations},
  author={Keel, Markus and Smith, Hart F and Sogge, Christopher D},
  journal={Journal d'Analyse Math{\'e}matique},
  volume={87},
  number={1},
  pages={265--279},
  year={2002},
  publisher={Springer}
}

@article{price,
  title={Nonspherical perturbations of relativistic gravitational collapse. I. Scalar and gravitational perturbations},
  author={Price, Richard H},
  journal={Physical Review D},
  volume={5},
  number={10},
  pages={2419},
  year={1972},
  publisher={APS}
}

@article{donn,
  title={A proof of Price's Law on \\ Schwarzschild black hole manifolds for all angular momenta},
  author={Donninger, Roland and Schlag, Wilhelm and Soffer, Avy},
  journal={Advances in Mathematics},
  volume={226},
  number={1},
  pages={484--540},
  year={2011},
  publisher={Elsevier}
}

@article{metato,
  title={Price's law on nonstationary space-times},
  author={Metcalfe, Jason and Tataru, Daniel and Tohaneanu, Mihai},
  journal={Advances in Mathematics},
  volume={230},
  number={3},
  pages={995--1028},
  year={2012},
  publisher={Elsevier}
}

@article{finster,
  title={Decay of solutions of the wave equation in the Kerr geometry},
  author={Finster, Felix and Kamran, Niky and Smoller, Joel and Yau, S-T},
  journal={Communications in Mathematical Physics},
  volume={264},
  number={2},
  pages={465--503},
  year={2006},
  publisher={Springer}
}

@article{dafrod,
  title={Lectures on black holes and linear waves},
  author={Dafermos, Mihalis and Rodnianski, Igor},
  journal={Clay Math. Proc},
  volume={17},
  pages={97--205},
  year={2013}
}

@article{dafrod2,
  title={A proof of the uniform boundedness of solutions to the wave equation on slowly rotating Kerr backgrounds},
  author={Dafermos, Mihalis and Rodnianski, Igor},
  journal={Inventiones Mathematicae},
  volume={185},
  number={3},
  pages={467--559},
  year={2011},
  publisher={Springer}
}

@article{mettat,
  title={Global parametrices and dispersive estimates for variable coefficient wave equations},
  author={Metcalfe, Jason and Tataru, Daniel},
  journal={Mathematische Annalen},
  volume={353},
  number={4},
  pages={1183--1237},
  year={2012},
  publisher={Springer}
}

@article{mametato,
  title={Strichartz estimates on Schwarzschild black hole backgrounds},
  author={Marzuola, Jeremy and Metcalfe, Jason and Tataru, Daniel and Tohaneanu, Mihai},
  journal={Communications in Mathematical Physics},
  volume={293},
  number={1},
  pages={37},
  year={2010},
  publisher={Springer}
}

@article{toh,
  title={Strichartz estimates on Kerr black hole backgrounds},
  author={Tohaneanu, Mihai},
  journal={Transactions of the American Mathematical Society},
  volume={364},
  number={2},
  pages={689--702},
  year={2012}
}

@inproceedings{bonyhaf1,
  title={Local energy decay for several evolution equations on asymptotically Euclidean manifolds},
  author={Bony, Jean-Fran{\c{c}}ois and H{\"a}fner, Dietrich},
  booktitle={Annales scientifiques de l'{\'E}cole Normale Sup{\'e}rieure},
  volume={45},
  number={2},
  pages={311--335},
  year={2012}
}

@article{mora,
  title={The decay of solutions of the exterior initial-boundary value problem for the wave equation},
  author={Morawetz, Cathleen S},
  journal={Communications on Pure and Applied Mathematics},
  volume={14},
  number={3},
  pages={561--568},
  year={1961},
  publisher={Wiley Online Library}
}

@article{bonyhaf3,
  title={The semilinear wave equation on asymptotically Euclidean manifolds},
  author={Bony, Jean-Fran{\c{c}}ois and H{\"a}fner, Dietrich},
  journal={Communications in Partial Differential Equations},
  volume={35},
  number={1},
  pages={23--67},
  year={2009},
  publisher={Taylor \& Francis}
}

@article{hmssz,
  title={On abstract Strichartz estimates and the Strauss conjecture for nontrapping obstacles},
  author={Hidano, Kunio and Metcalfe, Jason and Smith, Hart and Sogge, Christopher and Zhou, Yi},
  journal={Transactions of the American Mathematical Society},
  volume={362},
  number={5},
  pages={2789--2809},
  year={2010}
}

@inproceedings{metnaksog,
  title={Global existence of solutions to multiple speed systems of quasilinear wave equations in exterior domains},
  author={Metcalfe, Jason and Nakamura, Makoto and Sogge, Christopher D},
  booktitle={Forum Mathematicum},
  volume={17},
  number={1},
  pages={133--168},
  year={2005},
  organization={Walter de Gruyter}
}

@incollection{mettat1,
  title={Decay estimates for variable coefficient wave equations in exterior domains},
  author={Metcalfe, Jason and Tataru, Daniel},
  booktitle={Advances in phase space analysis of partial differential equations},
  pages={201--216},
  year={2009},
 publisher={Springer}
}

@article{sogwan,
  title={Concerning the wave equation on asymptotically Euclidean manifolds},
  author={Sogge, Christopher D and Wang, Chengbo},
  journal={Journal d'Analyse Mathematique},
  volume={112},
  number={1},
  pages={1--32},
  year={2010},
  publisher={Springer}
}

@article{ralston,
  title={Solutions of the wave equation with localized energy},
  author={Ralston, James V},
  journal={Communications on Pure and Applied Mathematics},
  volume={22},
  number={6},
  pages={807--823},
  year={1969},
  publisher={Wiley Online Library}
}

@article{blue2003,
author = "Blue, P. and Soffer, A.",
fjournal = "Advances in Differential Equations",
journal = "Adv. Differential Equations",
number = "5",
pages = "595--614",
publisher = "Khayyam Publishing, Inc.",
title = "Semilinear wave equations on the Schwarzschild manifold. I. Local decay estimates",
volume = "8",
year = "2003"
}

@article{dafermos2009red,
  title={The red-shift effect and radiation decay on black hole spacetimes},
  author={Dafermos, Mihalis and Rodnianski, Igor},
  journal={Communications on Pure and Applied Mathematics: A Journal Issued by the Courant Institute of Mathematical Sciences},
  volume={62},
  number={7},
  pages={859--919},
  year={2009},
  publisher={Wiley Online Library}
}

@article{andblue15,
 author = {Lars Andersson and Pieter Blue},
 journal = {Annals of Mathematics},
 number = {3},
 pages = {787--853},
 publisher = {Annals of Mathematics},
 title = {Hidden symmetries and decay for the wave equation on the Kerr spacetime},
 volume = {182},
 year = {2015}
}

@article{dafrodshlap16,
 author = {Mihalis Dafermos and Igor Rodnianski and Yakov Shlapentokh-Rothman},
 journal = {Annals of Mathematics},
 number = {3},
 pages = {787--913},
 publisher = {Annals of Mathematics},
 title = {Decay for solutions of the wave equation on Kerr exterior spacetimes III: The full subextremal case |a| < M},
 volume = {183},
 year = {2016}
}

@article{metsog06,
  title={Long-time existence of quasilinear wave equations exterior to star-shaped obstacles via energy methods},
  author={Metcalfe, Jason and Sogge, Christopher D},
  journal={SIAM journal on mathematical analysis},
  volume={38},
  number={1},
  pages={188--209},
  year={2006},
  publisher={SIAM}
}

@article{lintoh18,
author = {Hans Lindblad and Mihai Tohaneanu},
title = {Global existence for quasilinear wave equations close to Schwarzschild},
journal = {Communications in Partial Differential Equations},
volume = {43},
number = {6},
pages = {893-944},
year  = {2018},
publisher = {Taylor & Francis},
}

@article{luk2010null,
  title={The null condition and global existence for nonlinear wave equations on slowly rotating Kerr spacetimes},
  author={Luk, Jonathan},
  journal={Journal of the European Mathematical Society},
  volume={15},
  number={5},
  pages={1629-1700},
  year={2013}
}

@article{yang2015global,
  title={Global stability of solutions to nonlinear wave equations},
  author={Yang, Shiwu},
  journal={Selecta Mathematica},
  volume={21},
  number={3},
  pages={833--881},
  year={2015},
  publisher={Springer}
}

@article{smithsogge,
  title={Global Strichartz estimates for nontrapping perturbations of the Laplacian.},
  author={Smith, Hart F and Sogge, Christopher D},
  journal={Comm.Partial Differential Equations},
  volume={25},
  number={11-12},
  pages={2171-2183},
  year={2000}
}

@article{burq1998,
author = "Burq, Nicolas",
journal = "Acta Math.",
number = "1",
pages = "1-29",
publisher = "Institut Mittag-Leffler",
title = "D\'{e}croissance de l'\'{e}nergie locale de l'\'{e}quation des ondes pour le probl\`{e}me ext\'{e}rieur et absence de r\'{e}sonance au voisinage du r\'{e}el",
volume = "180",
year = "1998"
}

@article{sbierski2015characterisation,
  title={Characterisation of the energy of Gaussian beams on Lorentzian manifolds: with applications to black hole spacetimes},
  author={Sbierski, Jan},
  journal={Analysis \& PDE},
  volume={8},
  number={6},
  pages={1379--1420},
  year={2015},
  publisher={Mathematical Sciences Publishers}
}

@article{sterb, 
author={Sterbenz, Jacob}, 
journal={International Mathematics Research Notices}, 
title={Angular regularity and Strichartz estimates for the wave equation}, 
year={2005}, 
volume={2005}, 
number={4}, 
pages={187-231}
}

@article{alinhac,
author = {Alinhac, Serge},
year = {2006},
pages = {705-720},
title = {On the Morawetz-Keel-Smith-Sogge Inequality for the Wave Equation on a Curved Background},
volume = {42},
journal = {Publications of The Research Institute for Mathematical Sciences}
}

\end{document}